\documentclass[twoside,11pt]{article}

\usepackage{blindtext}
\usepackage{verbatim}
\usepackage[abbrvbib, preprint]{jmlr2e}

\usepackage[nothm]{YK}
\usepackage{tikz,forest}
\usetikzlibrary{arrows.meta}
\usepackage{nicefrac}
\usepackage{wrapfig}
\usepackage{subcaption}
\usepackage{algorithm}
\usepackage{algpseudocode}
\usepackage{graphicx}
\usepackage{booktabs}
\usepackage{array, makecell}
\usepackage{amssymb}
\usepackage{pifont}

\forestset{
    .style={
        for tree={
            base=bottom,
            child anchor=north,
            align=center,
            s sep+=1cm,
    straight edge/.style={
        edge path={\noexpand\path[\forestoption{edge},thick,-{Latex}] 
        (!u.parent anchor) -- (.child anchor);}
    },
    if n children={0}
        {tier=word, draw, thick, rectangle}
        {draw, diamond, thick, aspect=2},
    if n=1{%
        edge path={\noexpand\path[\forestoption{edge},thick,-{Latex}] 
        (!u.parent anchor) -| (.child anchor) node[pos=.2, above] {Y};}
        }{
        edge path={\noexpand\path[\forestoption{edge},thick,-{Latex}] 
        (!u.parent anchor) -| (.child anchor) node[pos=.2, above] {N};}
        }
        }
    }
}

\usepackage[most]{tcolorbox}

\tcbset{
  mybox/.style={
    colback=blue!5!white,   
    colframe=blue!75!black, 
    fonttitle=\bfseries,   
    coltitle=black,         
    boxrule=0.8pt,          
    arc=2mm,                
    left=4mm,right=4mm,top=2mm,bottom=2mm,
    enhanced,               
  }
}

\newcommand{\op}{\mathrm{op}}

 \usepackage{lastpage}

\ShortHeadings{Cost of Adaptation in Transfer Learning}{Chakraborty and Maity}
\firstpageno{1}

\begin{document}

\title{The Statistical Cost of Adaptation in Multi-Source Transfer Learning
}

\author{\name Abhinav Chakraborty \email ac4662@columbia.edu \\
       \addr Department of Statistics\\
       Columbia University \\
       New York, NY 10027, USA
       \AND
       \name Subha Maity \email subha.maity@uwaterloo.ca \\
       \addr Department of Statistics and Actuarial Science\\
       University of Waterloo\\
       Waterloo, ON N2L-3G1, Canada}

\editor{}

\maketitle
\begin{abstract}

Multi-source transfer learning can improve target-domain estimation by leveraging related source data, but its benefits depend on unknown source-to-target biases. This raises a fundamental question: can a bias-agnostic estimator perform as well as an oracle that knows the true bias configuration? To study this, we introduce the \emph{intrinsic cost of adaptation}, defined as the smallest worst-case ratio between the risk of any bias-agnostic estimator and the oracle risk. An intrinsic cost of one means oracle performance is achievable without knowing the biases, whereas a larger cost quantifies the unavoidable price of adaptation.

Focusing on parametric estimation, we show that multi-source transfer behaves fundamentally differently from the single-source setting: adaptation is not always possible, 
even with only two sources. For a fixed number of sources, we characterize the intrinsic cost of adaptation and identify a phase transition separating regimes where oracle performance is achievable from those where it is not. As the number of sources grows, we further show that the adaptation cost increases.

When adaptation over the full bias configuration space is impossible, additional structure can substantially reduce the cost. We study settings with ordered biases, clustered source parameters, and sufficiently separated non-informative sources, and propose estimators tailored to each regime, with supporting theoretical and empirical results. Overall, our results delineate the statistical limits of multi-source transfer, clarifying when oracle performance is attainable, when structural assumptions help, and when adaptation is fundamentally impossible.

\end{abstract}


\section{Introduction}


The paradigm of \emph{multi-source transfer learning} aims to aggregate data from multiple domains to enhance the performance of a statistical task in a specific domain of interest, known as the \emph{target domain}. This paradigm is particularly valuable in scientific and engineering applications where data scarcity is a significant challenge. For instance, healthcare analytics frequently rely on combining data from loosely related subpopulations \citep{cai2019survey, zhang2020multi}, while identifying patterns in regional satellite imagery often leverages data aggregated from other geographical areas \citep{zhang2010multi}.

In this work, we focus on parameter estimation. Given data from $(m+1)$ many domains (where $m$ is a positive integer) with underlying distributions $\{P_k\}_{k = 0}^m$ and a distribution functional $\mathcal{T}$, our goal is to estimate the target parameter $\theta_0^\star := \mathcal{T}(P_0) \in \reals^d$. Here, $k = 0$ indexes the target domain, and $P_0$ represents the target distribution. The remaining $m$ domains are referred to as the source domains, with corresponding source parameters denoted by $\{\theta_k^\star := \mathcal{T}(P_k)\}_{k = 1}^m \subset \reals^d$. Collectively, we refer to $\{\theta_k^\star\}_{k = 0}^m$ as the local parameters.


To evaluate the efficiency of multi-source transfer learning, it is standard practice to impose structural assumptions on the bias between the target and source domains. In the context of parameter estimation, \citet{li2022transfer,tian2023transfer,blanchard2024estimation} quantify this target-to-source bias directly through the difference $\theta_k^\star - \theta_0^\star$. Specifically, for a given subset of informative source indices $\cA \subseteq \{1, \dots, m\} := [m]$ and a predefined bias threshold $h > 0$, the local parameters are assumed to reside within the following class:
\begin{equation}
    \label{eq:bias-structure-1}
    \Theta_0(h, \cA, \|\cdot\|) :=  \big \{(\theta_0^\star, \dots,  \theta_m^\star): \|\theta_k^\star - \theta_0^\star \| \le h \text{ for all } k \in \cA \big\},
\end{equation}
where $\|\cdot\|$ represents a suitable norm. For instance, in the settings of linear regression and generalized linear models, \citet{li2022transfer} and \citet{tian2023transfer}, respectively, examine the estimation of a high-dimensional target parameter $\theta_0^\star$ subject to a sparsity constraint, defining the parameter class via the $\|\cdot\|_1$ norm. Conversely, for high-dimensional mean estimation without sparsity restrictions, \citet{blanchard2024estimation} analyzes the estimation of $\theta_0^\star$ using the parameter class defined by the $\|\cdot\|_2$ norm. Assuming $h$ and $\cA$ are given, these studies establish minimax optimal convergence rates for estimating $\theta_0^\star$ within their respective parameter classes. In practical scenarios, however, $h$ and $\cA$ are rarely known in advance, underscoring the need to adapt to these unknown quantities.

While \citet{li2022transfer} and \citet{tian2023transfer} provide high-probability guarantees for detecting the set $\cA$ (assuming additional conditions on $\{\theta_k^\star\}_{k = 0}^m$), they do not explicitly quantify how detection errors propagate to the final estimation error of $\theta_0^\star$. Consequently, the {cost of adapting} to unknown $h$ and $\cA$, if there is any, remains uncharacterized in their work. In contrast, for high-dimensional mean estimation, \citet{blanchard2024estimation} shows that even when all parameters are identical (i.e., $h = 0$ and $\cA = [m]$), a learner unaware of this ideal structure incurs a penalty in the mean squared estimation error (MSE). Relative to the oracle minimax rate, this penalty is bounded by $\cO(\max\{1, m \sqrt{\log(m)/d}\})$, demonstrating that adaptation is achievable when $m \sqrt{\log(m)/d} = \cO(1)$. Ultimately, while \citeauthor{blanchard2024estimation} highlights a potential cost of adaptation via this relative penalty, it remains unclear whether this penalty is an unavoidable fundamental limit or merely an artifact of the specific algorithm used.

Understanding the cost of adaptation is crucial in multi-source transfer learning. As we demonstrate in Section~\ref{sec:bias-adaptation}, achieving general adaptation is often fundamentally impossible; \ie, no algorithm can adapt to unknown parameters without incurring a non-trivial penalty. By investigating this cost, we can guide practitioners in determining whether pursuing an adaptive algorithm is a viable endeavor or precluded by fundamental limits.

Motivated by this, in this paper, we attempt to characterize the cost of adaptation for the following class of local parameters: 
\begin{equation}
    \label{eq:bias-structure-gen}
    \Theta(\bh, \|\cdot \|) :=  \big \{(\theta_0^\star, \dots,  \theta_m^\star): \|\theta_k^\star - \theta_0^\star \| \le h_k; ~ k = 1, \dots, m \big\},  
\end{equation} where $\bh = (h_1, \dots , h_m) \in [ 0, 1]^m$ is the vector consisting of the upper bounds for the target-to-source bias, henceforth referred to as the \emph{bias configuration}. For convenience, we assume $h_k \le 1$, but given any fixed $K > 0$, one may work with the upper bound $h_k \le K$ without altering the fundamental behaviors of adaptation to $\bh$. Compared to the previously defined classes $\Theta_0(h, \cA, \|\cdot \|)$, the $\Theta(\bh, \|\cdot \|)$ classes are strictly more general. By setting $h_k = h$ for $k \in S$ and $h_k = 1$ for $k \in S^\complement$, we recover the former as a special case. Thus, analyzing the cost of adaptation over $\{\Theta(\bh, \|\cdot \|): \bh\in [0, 1]^m\}$ provides a more granular understanding of the problem.

We now provide an informal definition of the intrinsic cost of adaptation (formalized in Section~\ref{subsec:intrinsic-cost}). Consider the general family of parameter configurations $\{\Theta(\bh, \|\cdot \|): \bh\in [0, 1]^m\}$. For an estimator of $\theta_0^\star$ that does not depend on $\bh$, we define its \emph{cost of adaptation} as the worst-case ratio of its MSE to the oracle minimax MSE for estimating $\theta_0^\star$ within $\Theta(\bh, \|\cdot \|)$ (where the oracle knows $\bh$). The \emph{intrinsic cost} is then defined as the minimum cost of adaptation over all such $\bh$-independent estimators. Defining this intrinsic cost is a non-trivial task in itself, as it requires first establishing the oracle minimax rate for $\Theta(\bh, \|\cdot \|)$. In this paper, we focus entirely on the $\|\cdot\|_2$ norm, which is naturally connected to the bias-variance decomposition in MSE. Notably, our oracle minimax rate analysis applies directly to $\Theta_0(h, \cA, \|\cdot\|_2)$ as a special case.

Having defined the intrinsic cost of adaptation, the broad goal of our paper is as follows:
\begin{tcolorbox}[mybox]
{\bf Goal:} Analyze the intrinsic cost of adaptation in multi-source transfer learning settings and then identify when general adaptation can or cannot be achieved.
\end{tcolorbox}
\noindent To the best of our knowledge, this is one of the first works to explicitly define and study the intrinsic cost of adaptation in multi-source transfer learning problems. In Section~\ref{sec:bias-adaptation}, we analyze this cost, assuming identical source sample sizes (\ie~ $n_1 = \dots = n_m = n$). While a comparable assumption regarding sample sizes appears in \citet{hanneke2022no,hanneke2026more}, our approach diverges by allowing the target sample size ($n_0$) to differ from the source sample size ($n$). We subsequently establish theoretical bounds for the adaptation cost based on four key parameters: the parameter dimension ($d$), the number of source domains ($m$), and the respective source and target sample sizes ($n$ and $n_0$). 
Notably, the cost depends on the ratio $n/n_0$, in addition to $m$ and $d$; this dependence would be obscured if we let $n_0 = n$.

\begin{table}[t]
\centering
\renewcommand{\arraystretch}{1.3}
\begin{tabular}{lccc}
\toprule 
& {Intrinsic cost}
 & \makecell{Adaptation\\ is possible} & \makecell{Adaptation\\ is impossible} \\
\midrule

\cite{hanneke2026more} 
& \ding{56} & $m = \cO\!\left(\mathrm{polylog}(n)\right)$ 
& $m = \cO\!\left(e^{c n}\right),\; c > 0$ \\[0.5em]

\cite{blanchard2024estimation} & \ding{56} & $m = \cO\!\left(\frac{\sqrt{d} n_0}{n}\right)$ & \ding{56}\\[0.5em]

Ours ($m = 1$) & $\tilde\cO\left(1\right)$ & always & never  \\[0.5em]

Ours (fixed $m \ge 2$) & $\cO\left(\frac{n}{dn_0} \vee 1\right)$ & $n = \cO(dn_0)$ & $n \gg dn_0$  \\[0.5em]

Ours (growing $m$) & Section \ref{subsec:cost-general-m} 
& $m = \cO\Big(d \wedge \sqrt{\frac{d n_0}{n}}\Big)$ 
& $m \gg  \!d \wedge \frac{dn_0}{n}$ \\

\bottomrule
\end{tabular}
\caption{A concise view of our results and comparison with previous works. The $\tilde \cO$ represents an order up-to a poly-logarithmic term.}
\label{tab:adaptation-concise}
\end{table}

Table \ref{tab:adaptation-concise} provides a brief preview of our findings. For a single source domain ($m = 1$), the intrinsic cost is $\tilde{\mathcal{O}}(1)$, suggesting that adaptation is generally possible in this setting. For a fixed finite number of source domains ($m \ge 2$), we characterize the intrinsic cost up to constants as $\mathcal{O}(\max{n/(dn_0), 1})$. This points to a phase transition: general adaptation is achievable when $n = \mathcal{O}(dn_0)$, whereas adaptation becomes impossible when $n \gg dn_0$. When the number of source domains $m$ grows, we obtain lower and upper bounds on the intrinsic cost. Although these bounds are not as tight as in the fixed-$m$ case, they still delineate the feasible and infeasible adaptation regimes up to a polynomial gap. Relative to the exponential separation appearing in prior work, this gap is substantially smaller in our setting; see Table \ref{tab:adaptation-concise}.


For finite $m$, our analysis of the intrinsic cost challenges a prevailing phenomenon in the multi-source transfer learning literature. Adaptation is always possible when $m = 1$; thus, we focus on $m \ge 2$. Previous works, such as \citet{hanneke2022no} and \citet{hanneke2026more}, show that adaptation is achievable when $m = \mathcal{O}(\mathrm{polylog}(n))$. Similarly, for non-parametric classification with finite $m$, \citet[Section 7]{Cai2021TransferClassifier} proposes a data-driven classifier that adapts to the oracle rate up to a polylogarithmic factor, demonstrating the feasibility of general adaptation. \citet{cai2024transfer} observes a comparable phenomenon in non-parametric regression. Together, these results might suggest that general adaptation is always possible whenever $m$ is finite. However, our analysis indicates that this is not the case in our setting. In our parameter estimation task, if $m \ge 2$ but is fixed, general adaptation is impossible when $n \gg dn_0$. Therefore, the feasibility of general adaptation for finite $m$ observed in prior work is an artifact of those specific problem settings and does not universally generalize.

 \citet{hanneke2022no} and \citet{hanneke2026more} are among the few prior works that investigate the existence of regimes where general adaptation fails in multi-source transfer learning problems. Both studies examine a non-parametric binary classification task with equal sample sizes ($n$) across all domains, characterizing the regimes in terms of $m$ and $n$. \citet{hanneke2022no} establishes impossibility by fixing $n$ while allowing $m$ to grow. \citet{hanneke2026more} improves upon this by allowing $n$ to grow but requires $m$ to grow exponentially with $n$ to demonstrate impossibility. Furthermore, \citet{hanneke2022no,hanneke2026more} shows that general adaptation (up to a logarithmic factor) is achievable when $m = \mathcal{O}(\mathrm{polylog}(n))$. Consequently, in their framework, the gap between the feasible and impossible adaptation regimes grows exponentially with $n$. In contrast, our setting with growing $m$ exhibits only a polynomial gap ($\sqrt{dn_0/n}$ versus $dn_0/n$, as shown in Table \ref{tab:adaptation-concise}). It is important to note, however, that our parameter estimation task fundamentally differs from their non-parametric binary classification. Moreover, the primary objective of both \citeauthor{hanneke2022no} and \citeauthor{hanneke2026more} is to prove the existence of impossibility regimes, rather than to precisely characterize their boundaries.

To address regimes where general adaptation is impossible, in Section \ref{sec:adaptation-general} we introduce additional structural assumptions on the feasible set of $\bh$ and the local parameters to restore the feasibility of adaptation (at least partially). Specifically, in Section \ref{subsec:ordered-bias}, we assume the source domains can be ordered by increasing values of $h_k$. In Section \ref{sec:clustering-method}, we assume the source parameters $\{\theta_k^\star\}_{k=1}^m$ possess a clustering structure. Finally, in Section \ref{sec:test-then-combine}, we assume that the parameters $\theta_k^\star$ of non-integrable source domains are sufficiently distinct from $\theta_0^\star$ to be identified via a simple hypothesis test. For each of these restricted configurations, we develop tailored methodologies that meaningfully reduce the cost of adaptation. 

The remainder of the paper is organized as follows. In Section \ref{subsec:related-work}, we review related literature. In Section \ref{sec:problem-formulation}, we formulate the multi-source transfer learning problem, establish minimax rates for estimating $\theta_k^\star$ within $\Theta(\boldsymbol{h}, \|\cdot\|_2)$, and formally define the intrinsic cost of adaptation. In Section \ref{sec:bias-adaptation}, we present a theoretical analysis of this intrinsic cost, characterizing the conditions under which general adaptation is possible or impossible. For regimes where general adaptation fails, in Section \ref{sec:adaptation-general}, we suggest further restrictions on the problem setting under which adaptation is possible and propose accompanying adaptive estimators. In Section \ref{sec:simuations}, we evaluate the performance of our proposed methods against relevant baselines using simulated data. Finally, in Section \ref{sec:discussion}, we conclude the paper with a discussion.

\subsection{Related Work}
\label{subsec:related-work}

\paragraph{Classical multiple-source adaptation theory.}
The broader machine-learning literature on multi-source domain adaptation is older and much
wider. A seminal theoretical contribution is \citet{mansour2008domain}, which
derives target-risk guarantees for combining predictors from multiple sources and shows that naive
convex aggregation can fail under distribution shift. Subsequent works, such as those by \citet{hoffman2018algorithms} and \citet{duan2023adaptive}, develop algorithms and theoretical guarantees for weighted or distribution-aware multi-source adaptation. However, their objectives and guarantees are not directly comparable to ours. Specifically, \citet{hoffman2018algorithms} studies prediction risk in classification and general supervised learning under distributional divergence assumptions, while \citet{duan2023adaptive} analyzes estimation error under several relatedness assumptions on the parameter space. In contrast, our work focuses on minimax parameter estimation over $\Theta(\bh, \|\cdot \|)$ and investigates the fundamental statistical cost that one must pay for an algorithm that is adaptive to $\bh$.

\paragraph{Structured-source models: ordered, clustered, and personalized transfer.}
When unrestricted adaptation is too ambitious, a natural strategy is to impose additional structure
on the source family. One important example is \emph{ordered} relatedness, where sources can be
ranked from most to least compatible with the target. This type of side information already appears
in \cite{hanneke2022no}, where it is established that when such a ranking is available, a rank-based aggregation rule can
achieve near-oracle performance. Our ordered-bias estimator follows the same general idea, but in a
parameter-estimation model with explicit source-specific bias levels ($h_k$'s).

Because the candidate aggregations are nested prefixes of the ordered list, our construction is also
close in spirit to adaptive estimation methods for ordered families of estimators, such as Lepski's
method and its Goldenshluger--Lepski refinements \cite{lepski1997optimal,goldenshluger2011bandwidth}.
The point is not that these methods directly solve our problem, but that they suggest a useful design
principle: when candidates are naturally ordered, one can search for the largest stable aggregation
rather than search over all subsets.

Another important structural idea is \emph{clustering}. In personalized and federated learning, it is
common to assume that clients or sources partition into a small number of latent groups with
similar parameters or distributions. \citet{sattler2020clustered},
\citet{mansour2020three}, and \citet{long2023multi}, among many
others, study clustered or personalized federated learning from algorithmic and statistical
viewpoints. These works motivate our clustered-source section:
once the sources can be compressed into a small number of effective groups, the adaptation problem
becomes easier. Our analysis differs from that literature in emphasis. Rather than
proposing a new federated optimization algorithm, we study the statistical price of using clustering
as a structural restriction that can restore partial adaptivity.



\section{Problem Formulation}
\label{sec:problem-formulation}
In a multi-source transfer learning setting, the learner observes a dataset from each available domain. For each index $k \in \{0, \dots, m\}$, let $\mathcal{D}_k = \{Z_{k, i}: i = 1, \dots, n_k \}$ denote the dataset associated with the $k$-th domain. This dataset comprises $n_k \ge 1$ independent samples drawn randomly from a domain-specific distribution $P_k$. For parameter estimation, the primary objective of multi-source transfer learning is to estimate a target parameter $\theta_0^\star$ by leveraging the combined data from all domains, $\{\mathcal{D}_k\}_{k=0}^m$. 

While the problem is formulated abstractly at this stage, for clarity of exposition, we assume that one can construct a ``well-behaved'' estimator $\widetilde{\theta}_k$ for each domain-specific parameter $\theta_k^\star$ using only the corresponding local data $\mathcal{D}_k$. Once these local estimators $\{\widetilde{\theta}_k\}_{k=0}^m$ are obtained, we shift our focus entirely to them. That is, rather than working directly with the raw datasets $\{\mathcal{D}_k\}_{k=0}^m$, our subsequent analysis and estimator construction rely solely on the joint distribution of these domain-specific estimators.

The remainder of this section is organized as follows. In Section \ref{subsec:domain-specific}, we formalize the notion of a ``well-behaved'' estimator sequence $\{\widetilde{\theta}_k\}_{k = 0}^m$. Denoting the final estimator of $\theta_0^\star$ as $\widehat{\theta}_0$, we measure the estimation error using the mean squared error (MSE): $\Ex [ \|\widehat{\theta}_0 - \theta_0^\star\|_2^2]$, where the expectation is taken over the joint distribution of $\{\widetilde{\theta}_k\}_{k = 0}^m$. 
In Section \ref{subsec:oracle-minimax-rate}, we characterize the minimax optimal rate of convergence for the MSE within $\Theta(\bh, \|\cdot\|_2)$. Building on this optimal rate, in Section \ref{subsec:intrinsic-cost} we introduce the \emph{intrinsic cost of adaptation}, defined as the sub-optimality that any estimator $\widehat{\theta}_0$ must incur relative to the oracle rate when $\bh$ is unknown. This intrinsic cost is the central quantity of interest for understanding the feasibility of bias adaptation.

\subsection{The Domain-Specific Estimators}
\label{subsec:domain-specific}

We formalize the notion of a ``well-behaved'' estimator by assuming it satisfies a sub-Gaussian concentration inequality. Specifically, we assume there exist constants $\tau, C_\tau > 0$, independent of the sample sizes $n_k$ and the dimension $d$, such that for any $k \in \{0, \dots, m\}$:
\begin{equation}\label{eq:sub-gaussian}
    \max_{\|a\|_2 = 1} P \left ( a^\top (\widetilde{\theta}_k - \theta_k^\star) > t \right) \le C_\tau \exp \left(- \frac{n_k t^2}{2\tau^2}\right) \quad \text{for all } t > 0\,.
\end{equation} 
Such well-behaved estimators naturally emerge in a broad spectrum of parameter estimation problems, as illustrated by the following examples.

\begin{example}[Normal mean estimation] \label{example:normal-mean}
    Let $ \cD_k := \{Z_{k, i}: i = 1, \dots , n_k\} $ be i.i.d.\ $\bN(\theta_k^\star, \Sigma_k)$, and assume the largest eigenvalue of $\Sigma_k$ is bounded above by $\tau^2$. Let the domain-specific estimator be  $\widetilde{\theta}_k := n_k^{-1}\sum_{i = 1}^{n_k} Z_{k, i}$. Then, $\widetilde{\theta}_k \sim \bN(\theta_k^\star, \Sigma_k / n_k)$, and the condition in \eqref{eq:sub-gaussian} holds with $C_\tau = 1$. 
\end{example}

\begin{example}[Linear regression]\label{example:linear-regression}
    Suppose $\cD_k:= \{Z_{k, i}:= (X_{k, i}, Y_{k, i}):  i = 1 \dots, n_k\} \in \reals^d \times \reals$ are independent pairs distributed according to the linear regression model:
    \[
    Y_{k, i} = X_{k, i}^\top \theta_k^\star +\eps_{k, i}, \quad \eps_{k, i} \overset{\text{i.i.d.}}{\sim} \bN(0, \sigma_k^2)\,. 
    \]
    Define the scaled Gram matrix as $\widehat{\Sigma}_k := n_k^{-1} \sum_{i = 1}^{n_k} X_{k, i} X_{k, i}^\top$, and let the domain-specific estimator be the ordinary least squares coefficient $\widetilde{\theta}_k := \widehat{\Sigma}_k^{-1} \{n_k^{-1} \sum_{i = 1}^{n_k} X_{k, i} Y_{k, i}\}$. Consequently,
    \[
    \widetilde{\theta}_k \sim \bN\left(\theta_k^\star, \frac{\sigma_k^2}{n_k} \widehat{\Sigma}_k^{-1} \right)\,.
    \]
    Assuming the largest eigenvalues of $\sigma_k^2 \widehat{\Sigma}_k^{-1}$s are bounded above by $\tau^2$, $\widetilde{\theta}_k$ satisfies the condition in \eqref{eq:sub-gaussian} with $C_\tau = 1$. 
\end{example}


\begin{example}[M estimation]
\label{example:m-estimation}
    For $k$-th domain, let $\cD_k := \{Z_{k, i}\}_{i = 1}^{n_k} \subset \cZ $ be a random sample from $P_k$, where $\cZ$ is the sample space for the distribution.  Let $\ell: \cZ \times \reals^d \to \reals $ be a loss function such that for every $z \in \cZ$, the $\ell(z , \cdot)$ is twice continuously differentiable and convex. Define the local estimator as an M estimator
    \[
    \widetilde{\theta}_k := \arg\min_{\theta \in \reals^d} \frac{1}{n_k} \sum_{i = 1}^{n_k} \ell(Z_{k, i}, \theta)\,,
    \] and the local parameter value $\theta_k^\star$ as the minimizer of $\Ex[\ell(Z_{k, i}, \cdot)]$. Under suitable conditions, in Appendix \ref{supp:m-estimation}, we establish the concentration inequality in \eqref{eq:sub-gaussian}.

\end{example}

By directly imposing the sub-Gaussian concentration assumption \eqref{eq:sub-gaussian} on the estimators $\{\widetilde{\theta}_k\}_{k = 0}^m$, we can unify the study of various parameter estimation problems within multi-source transfer learning frameworks, as illustrated in the preceding examples. Because our subsequent analysis focuses on establishing minimax rates based on these summary statistics rather than on the raw datasets, a brief clarification is needed. In the Gaussian mean model (Example~\ref{example:normal-mean}), the estimators $\widetilde{\theta}_k$ are sufficient statistics, so passing from the raw data to $\{\widetilde{\theta}_k\}_{k=0}^m$ entails no loss of information. For this reason, the normal model provides a natural benchmark for the oracle minimax rate. Our upper-bound analysis, however, is not specific to the Gaussian setting; it applies more generally to any problem for which the local estimators satisfy \eqref{eq:sub-gaussian}.


Before proceeding with our performance study for estimating the target parameter, we note a final few details. While the bias terms $h_k$ may be conceptualized as shrinking toward zero as the sample size tends to infinity, we treat the variance proxy $\tau$ as a fixed constant independent of the sample size. Moreover, we assume that $d\tau^2 / n_0 \le 1$ is a minimal assumption regarding the quality of $\widetilde \theta_0$, which lends some technical convenience to the proofs.

\subsection{An oracle minimax rate} \label{subsec:oracle-minimax-rate}

In this subsection, we formally characterize the minimax optimal rate of convergence for estimating $\theta_0^\star$ when the bias configuration is pre-specified.
To formalize this, we fix a bias configuration $\bh \in [0, 1]^m$ and define a suitable probability class for the joint distribution of $\{\widetilde{\theta}_k\}_{k=0}^m$ such that $(\theta_0^\star, \dots, \theta_m^\star) \in \Theta(\bh, \|\cdot \|_2)$. This class is defined as follows:
\begin{definition}[Probability class with general bias structure] \label{def:setup}
    Fix a bias configuration vector $\bh \in [0, 1]^m$, a vector of domain-specific sample sizes $\bn = (n_0, \dots, n_m) \in \bbN^{m+1}$, and constants $\tau, C_\tau > 0$. Let $\cP(\bh, \bn, \tau, C_\tau)$ denote the class of joint probability distributions of $\{\widetilde{\theta}_k\}_{k=0}^m$ such that the estimators are independent, and there exists a parameter set $(\theta_0^\star, \dots, \theta_m^\star) \in \Theta(\bh)$ satisfying the concentration bound in \eqref{eq:sub-gaussian}. 
\end{definition}

For convenience, we often omit the constant $C_\tau$ and simply write $\cP(\bh, \bn, \tau)$. We conduct a minimax study within this class to determine the optimal rate of convergence for the MSE, $\Ex [ \|\widehat \theta_0 - \theta_0^\star\|_2^2]$. 
Before presenting the results, we clarify that this minimax study is performed strictly within the class $\cP(\bh, \bn, \tau)$, which inherently depends on the bias $\bh$, sample sizes $\bn$, and scale $\tau$. Consequently, the estimator itself will depend on these parameters. To emphasize this dependence, we denote the estimator as $\widehat \theta_{0, \bh}$. The following theorem establishes the oracle minimax rate.

\begin{theorem}[Oracle minimax rate] \label{thm:oracle-minimax-rate}
Let $h_0 := 0$. Given any $\{0\} \subset S \subset \{0, \dots, m\}$, let $h_S := \max_{k \in S} h_k$ denote the maximum achievable bias magnitude, and $N_S = \sum_{k \in S} n_k$ denote the total sample size for the estimators $\{\widetilde \theta_k : k \in S\}$. Define the rate function:
\begin{equation} \label{eq:loca-loptimal-rate}
    \fR(\bh, \bn, \tau) := \min_{\{0\}\subset S \subset \{0,\dots,m\}}
\left\{
 \frac{d\tau^2}{N_S}
+  h_S^2
\right\}\,.
\end{equation}
Then, there exist constants $0 < c \le C < \infty$, depending on $C_\tau$ but independent of $\bh$, $\bn$, $\tau$, or $d$, such that:
\[
c \cdot \fR(\bh, \bn, \tau) ~\le~  \inf_{\widehat \theta_{0, \bh}} \; \sup_{P \in \cP(\bh, \bn, \tau)} \; \Ex _P \left [ \big \|\widehat\theta_{0, \bh}-\theta_0^\star\big\|_2^2 \right] ~\le~ C \cdot  \fR(\bh, \bn, \tau)\,.
\] 
In other words, $\fR(\bh, \bn, \tau)$ represents the \emph{minimax optimal rate} of convergence for the MSE within the class $\cP(\bh, \bn, \tau)$.
\end{theorem}
A detailed proof of this theorem is deferred to Appendix \ref{proof:thm:oracle-minimax-rate}.


To the best of our knowledge, the oracle minimax rate for estimating the target parameter in multi-source transfer learning under a general bias configuration $\bh$ has not been explicitly characterized in this generality. We therefore view the resulting rate characterization as a contribution of independent interest, in addition to its use as the benchmark for our adaptation analysis. We now provide a few remarks discussing the intuition behind this optimal rate, the construction of a rate-optimal estimator, and how our results compare with existing literature.

\begin{remark}[Oracle rate as a bias-variance trade-off] \label{remark:bias-variance-tradeoff}
The optimal rate in \eqref{eq:loca-loptimal-rate} captures the classic bias-variance trade-off. Given a subset of domains $S$, $h_S^2$ represents the squared maximum bias, while $d \tau^2 / N_S$ represents the variance associated with the pooled sample size. Expanding the subset $S$ decreases the variance term $d \tau^2 / N_S$ but inevitably increases the bias term $h_S^2$. The optimal rate $\fR(\bh, \bn, \tau)$ is achieved by finding the subset $S$ that minimizes the sum $(h_S^2 + d \tau^2 / N_S)$, striking a balance between squared bias and variance.
    
We can construct an estimator $\widehat \theta_{0, \bh}$ that achieves this optimal trade-off. Let $S_{\bh}$ be a subset that achieves the minimum in \eqref{eq:loca-loptimal-rate}, and define:
\[
\widehat \theta_{0, \bh} := \sum_{k \in S_{\bh}} \omega_{k, \bh} \widetilde \theta_k, \qquad \omega_{k, \bh} := \frac{n_k}{N_{S_{\bh}}}\,. 
\] 
In the proof of Theorem \ref{thm:oracle-minimax-rate}, we show that this estimator achieves the upper bound with $C = 4 C_\tau \vee 1$. Note that $\widehat \theta_{0, \bh}$ is an \emph{idealized} (oracle) estimator because its construction requires explicit knowledge of $\bh$. Nevertheless, it serves two crucial theoretical purposes:
\begin{enumerate}
    \item Analyzing the MSE of $\widehat \theta_{0, \bh}$ establishes the upper bound in Theorem \ref{thm:oracle-minimax-rate}, confirming $\fR(\bh, \bn, \tau)$ as the optimal rate. 
    \item It reveals that, once the optimal subset $S_{\bh}$ is identified, the ideal estimator is simply a sample-size-weighted average of the local estimators $\{\widetilde \theta_k : k \in S_{\bh}\}$. Therefore, the primary challenge in achieving the optimal rate without knowing $\bh$ lies entirely in identifying the subset $S_{\bh}$ that yields the best bias-variance trade-off.
\end{enumerate}
\end{remark}

\begin{remark} \label{remark:oracle-rate-comparison}
Below, we contextualize our oracle minimax rate by comparing it with established results in the literature.

\begin{enumerate}
    \item {\bf Comparison with \citet{chen2025minimax}:} 
    With a single source domain ($m = 1$), the optimal subset $S$ in the minimization problem in \eqref{eq:loca-loptimal-rate} is restricted to either $S = \{0\}$ or $S = \{0, 1\}$. Consequently, the optimal rate simplifies to:
    \[
  \textstyle   \fR (\bh, \bn, \tau)  = \left (\frac{d \tau^2}{n_0 +n_1} + h_1^2\right)  \wedge \frac{d\tau^2}{n_0} \asymp 
    \frac{d\tau^2}{n_0 +n_1} + \left( h_1^2 \wedge \frac{d\tau^2}{n_0} \right) \,.
    \] 
    This serves as a direct generalization of the optimal rate established in \citet[Section 3]{chen2025minimax} to the multi-dimensional parameter $\theta_0^\star$. 
    
    \item {\bf Comparison with \citet{li2022transfer} and \citet{tian2023transfer}:} 
    Both of these works investigate high-dimensional transfer learning under a sparsity constraint for the target parameter—the former in linear regression and the latter in generalized linear models. Given a bias level $h > 0$, an index set $\cA \subset [m]$, and a sparsity level $s$, they analyze the parameter space:
    \[
    \Big \{(\theta_0^\star, \dots,  \theta_m^\star): \|\theta_0^\star\|_0 \le s, \text{ and } \|\theta_k^\star - \theta_0^\star \|_1 \le h \text{ for } k \in \cA \Big\}\,.
    \] 
    As seen in \citet[Section 2.2]{li2022transfer} and \citet[Section 3.1]{tian2023transfer}, this is fundamentally our space $\Theta_0(h, \cA, \|\cdot\|_1 )$ defined in \eqref{eq:bias-structure-1}, but with an added sparsity constraint on $\theta_0^\star$. Over this space, they establish the following minimax lower bound for the $\ell_2$-error rate (for some $c > 0$):
    \begin{equation} \label{eq:sparsity-l2-rate}
   \textstyle  P \left(\inf\limits_{\widehat \theta_0}~ \sup\limits_{\theta_0^\star}~ \big\|\widehat \theta_0 - \theta_0^\star \big\|_2^2 \ge c \left \{ \frac{s\log d}{N_{\{0\} \cup \cA} } +  \frac{s\log d }{n_0} \wedge h \sqrt{\frac{\log d }{n_0}} \wedge h^2  \right\}  \right) \ge \frac12 \,.
    \end{equation}
 
    In contrast, our rate is derived for the $\|\cdot\|_2$ norm without any sparsity constraints, corresponding to the space $\Theta_0(h, \cA, \|\cdot\|_2 )$. This formulation is a special case of $\Theta(\bh, \|\cdot \|_2)$ where $h_k = h \cdot \mathbf{1}\{k \in \cA\} + 1\cdot \mathbf{1}\{k \in \cA^\complement\}$. To achieve the optimal bias-variance trade-off in \eqref{eq:loca-loptimal-rate}, we only need to evaluate two candidate sets: $S = \{0\}$ and $S = \{0\} \cup \cA$. The resulting optimal MSE rate under $\Theta_0(h, \cA, \|\cdot\|_2 )$ is:
    \begin{equation}\label{eq:oracle-rate-same-bias}
   \textstyle  \fR (\bh, \bn, \tau)  = \frac{d\tau^2}{n_0} \wedge \left \{  \frac{d\tau^2}{N_{\{0\} \cup \cA}} + h^2\right\} \asymp 
    \frac{d\tau^2}{N_{\{0\} \cup \cA}} + \left( \frac{d\tau^2}{n_0} \wedge h^2 \right) \,.
    \end{equation}
    To bridge their setting with ours, we can remove the sparsity constraint in \eqref{eq:sparsity-l2-rate} by setting $s = d$. Ignoring the logarithmic factors ($\log d$), our rate \eqref{eq:oracle-rate-same-bias} matches theirs, with the sole exception of the $h (\log d / n_0)^{1/2}$ term, which is a consequence of their specific problem setting.

    \item {\bf Comparison with \citet{blanchard2024estimation}:} 
 In their section 5, \citet{blanchard2024estimation} investigates high-dimensional mean estimation over the space $\Theta_0(h, \cA, \|\cdot\|_2 )$, assuming the data follows multivariate normal distributions (analogous to our Example \ref{example:normal-mean}). To facilitate an easy comparison with their minimax results \citep[Section 5]{blanchard2024estimation}, we set their covariance matrices to $\Sigma_k = \tau^2 \bI_d$. Under this simplification, the rate established in \citet[Theorem 4]{blanchard2024estimation} reduces to:
    \[
   \textstyle \frac{d\tau^2}{N_{\{0\} \cup \cA}} + \left( \frac{d\tau^2}{n_0} \wedge h^2 \right) \,.
    \] 
    This exactly recovers our optimal rate derived in  \eqref{eq:oracle-rate-same-bias}.
\end{enumerate}
    
\end{remark}

\subsection{The intrinsic cost of adaptation} \label{subsec:intrinsic-cost}

In Theorem \ref{thm:oracle-minimax-rate}, we established the minimax optimal rate of convergence for the MSE when estimating $\theta_0^\star$ over the class $\cP(\bh, \bn, \tau)$, assuming the bias vector $\bh$ is known. Building on this oracle rate, we now formalize the \emph{intrinsic cost of adaptation}. This cost quantifies the performance penalty incurred when the exact bias $\bh$ is unknown but is assumed to belong to a specified set of possible configurations $\cH \subset [0, 1]^m$.

\begin{definition}[Intrinsic cost of adaptation] \label{def:intrinsic-cost}
    Let $\cH \subset [0, 1]^{m}$ be a set of possible bias configurations. Consider any estimator $\widehat \theta_0$ that depends on the domain-specific estimators $\{\widetilde \theta_k\}_{k = 0}^m$, the sample sizes $\bn$, and the sub-Gaussian parameter $\tau$, but \emph{not} on the true bias $\bh \in \cH$. We define the adaptation cost of $\widehat \theta_0$ over $\cH$ as:
    \[
    \fC \big (\widehat \theta_0, \cH\big) :=  \sup_{\bh\in \cH}  \;\left\{\frac{ \sup\limits_{P \in \cP(\bh, \bn, \tau)} ~\Ex_P \left [  \|\widehat\theta_0-\theta_0^\star\|_2^2 \right]}{\fR(\bh, \bn, \tau)} \right\},.
    \] 
    The intrinsic cost of adaptation is then defined as the minimum possible adaptation cost achievable by any such estimator:
    \[
    \fC^\star(\cH) :=  \inf_{\widehat \theta_0}\; \fC \big (\widehat \theta_0, \cH\big)\,.
    \]
\end{definition}


The cost above provides a mathematical formalization of the unavoidable penalty one must pay for not knowing $\bh$ in advance. If an estimator achieves $\fC(\widehat \theta_0, \cH) = \cO(1)$, it attains the oracle rate $\fR(\bh, \bn, \tau)$ up to a constant factor without requiring knowledge of $\bh$. Such an estimator is said to be \emph{adaptive to the unknown bias} over $\cH$. In practice, however, strictly achieving an $\cO(1)$ adaptation cost is often challenging; an estimator is typically considered acceptable if this cost grows by no more than a poly-logarithmic factor. 

Moving forward, we refer to the cost of adapting to the unconstrained configuration set $\cH = [0, 1]^m$ as the cost of general adaptation. If the intrinsic cost of general adaptation, \ie, $\fC^\star([0, 1]^m)$, grows faster than a poly-logarithmic factor, then no single algorithm can uniformly adapt to all possible bias configurations $\bh \in [0, 1]^m$. In such cases, we conclude that general adaptation is theoretically impossible.

\section{Adaptation to unknown bias levels} \label{sec:bias-adaptation}

Before turning to the multi-source setting, it is useful to recall the case $m=1$, where adaptation to the unknown bias is fundamentally simpler. With only one source domain, there is no ambiguity about which source should be used for transfer; the problem reduces to deciding, in a data-driven way, whether the source is sufficiently close to the target to be useful for transfer. As shown in Remark \ref{remark:adaptation-m-1}, this can be done without incurring any nontrivial adaptation cost, so adaptation is possible over the full bias range in the one-source setting. This is in sharp contrast to the main message of this section: once $m\ge 2$, the difficulty changes qualitatively because the learner must not only decide whether transfer is beneficial but also identify which of the candidate source domains is closest to the target.

\subsection{General adaptation with a fixed number of source domains}

We begin with the regime in which the number of source domains $m$ is fixed and does not grow with the problem parameters. This regime already captures the main statistical obstruction to adaptation under unknown biases: the learner must determine which source domains are sufficiently close to the target to be useful for transfer, even though the bias configuration is not known in advance. Throughout this subsection, we assume equal source sample sizes, \ie,
\(
n_1=\cdots=n_m=n.
\) 

Our first result shows that for fixed $m$, uniform adaptation over the full configuration space $[0,1]^m$ is possible only when the source sample size is not too large relative to the target sample size and the dimension. In particular, near-oracle adaptation can hold only in the regime $n\lesssim d n_0$. By contrast, when $n\gg d n_0$, the intrinsic cost of adaptation necessarily grows, showing that even for a fixed number of source domains, unrestricted adaptation over the full bias configuration space becomes statistically impossible without incurring a substantial cost.

\begin{theorem}[Intrinsic cost of adaptation for fixed $m$]
	\label{thm:cost-fixed-m}
	Assume that $m\ge 2$ is fixed and that $n_1=\cdots=n_m=n$. Then there exist positive constants $c_L(m) >0$ and $c_U(m)>0$, depending only on $m$, such that
	\[
	c_L(m)\left\{\left(\frac{n}{d n_0}\right)\vee 1\right\}
	\;\le\;
	\mathfrak C^\star([0,1]^m)
	\;\le\;
	c_U(m)\left\{\left(\frac{n}{d n_0}\right)\vee 1\right\}.
	\]
\end{theorem}

The proof of Theorem~\ref{thm:cost-fixed-m} is deferred to 
Appendix~\ref{supp:proof:thm:cost-fixed-m}; here, we provide a sketch 
for the lower bound proof in the simplest $m = 2$ setting. This sketch 
will offer some insight into the difficulty of general adaptation for 
a fixed $m$ regime. The upper bound, on the other hand, is obtained by 
letting $m$ be fixed within Corollary~\ref{cor:adaptation-cost-ub-general-m}; 
hence, the tight upper bound is achieved by the model-selection-based 
estimator described in Section~\ref{subsubsec:model-selection}.

\paragraph{Proof sketch:}
The key idea is to exhibit two configurations that share the same unordered bias magnitudes but differ in which source is closer to the target. Any estimator that adapts uniformly to both configurations must identify the better source, and this identification step leads to the fundamental adaptation cost.

\medskip
\noindent\textit{Step 1: A two-point construction.}
Fix $0\le g_1\le g_2\le \tau/(2\sqrt{n_0})$, let $e_1=(1,0,\dots,0)^\top\in\mathbb{R}^d$, and define two distributions:
\begin{equation}\label{eq:two-source-lower-bound-construction}
    \begin{aligned}
        (P_0):\quad &\{\widetilde\theta_0,\widetilde\theta_1,\widetilde\theta_2\}
        \sim N\!\left(0_d,\tfrac{\tau^2}{n_0}I_d\right)
        \otimes N\!\left(g_1 e_1,\tfrac{\tau^2}{n}I_d\right)
        \otimes N\!\left(g_2 e_1,\tfrac{\tau^2}{n}I_d\right),\\
        (P_1):\quad &\{\widetilde\theta_0,\widetilde\theta_1,\widetilde\theta_2\}
        \sim N\!\left((g_1+g_2)e_1,\tfrac{\tau^2}{n_0}I_d\right)
        \otimes N\!\left(g_1 e_1,\tfrac{\tau^2}{n}I_d\right)
        \otimes N\!\left(g_2 e_1,\tfrac{\tau^2}{n}I_d\right).
    \end{aligned}
\end{equation}
The two distributions induce bias configurations $(h_1,h_2)=(g_1,g_2)$ and $(h_1,h_2)=(g_2,g_1)$, respectively—same in magnitude but reversed in roles. Moreover, $P_0$ and $P_1$ differ only in the first coordinate of $\widetilde\theta_0$; all other coordinates are uninformative for distinguishing them.

\medskip
\noindent\textit{Step 2: Lower bound on minimax risk.}
Because $P_0$ and $P_1$ differ only in one coordinate, the problem reduces to estimating a one-dimensional Gaussian mean, giving
\[
    \inf_{\widehat\theta_0}\sup_{P\in\mathcal{P}_0(g_1,g_2)}
    \mathbb{E}\!\left[\|\widehat\theta_0-\theta_0^\star\|_2^2\right]
    \gtrsim g_2^2.
\]

\medskip
\noindent\textit{Step 3: Upper bound on oracle risk.}
Since both configurations $(g_1,g_2)$ and $(g_2,g_1)$ belong to $G:=\{(g_1,g_2),(g_2,g_1)\}$, the oracle risk satisfies
\[
    \sup_{\mathbf{h}\in G}\mathfrak{R}(\mathbf{h},\mathbf{n},\tau)
    \lesssim
    \frac{d\tau^2}{n_0}\wedge\!\left(\frac{d\tau^2}{n}\vee g_1^2\right).
\]

\medskip
\noindent\textit{Step 4: Adaptation cost for a single pair.}
Combining Steps 2 and 3 yields
\begin{equation}\label{eq:lb-m2-eq-1}
    \sup_{\mathbf{h}\in G}
    \frac{
        \sup_{P\in\mathcal{P}(\mathbf{h},\mathbf{n},\tau)}
        \mathbb{E}\!\left[\|\widehat\theta_0-\theta_0^\star\|_2^2\right]
    }{
        \mathfrak{R}(\mathbf{h},\mathbf{n},\tau)
    }
    \gtrsim
    \frac{g_2^2}{\dfrac{d\tau^2}{n_0}\wedge\!\left(\dfrac{d\tau^2}{n}\vee g_1^2\right)}.
\end{equation}
Any adaptive estimator must incur an error of order $g_2^2$ to distinguish between the two configurations, whereas the oracle risk is governed by the (potentially much smaller) denominator.

\medskip
\noindent\textit{Step 5: Extension to a class of bias vectors.}
Since \eqref{eq:lb-m2-eq-1} holds for every pair $(g_1,g_2)$, it applies simultaneously over any set
\[
    \mathcal{G}\subset\!\left\{(g_1,g_2):0\le g_1\le g_2\le \frac{\tau}{2\sqrt{n_0}}\right\}.
\]
Taking the supremum over $\mathcal{G}$ gives
\begin{equation}\label{eq:adaptation-cost-lb-m2}
    \mathfrak{C}^\star(\mathcal{G}^{\mathrm{symm}})
    \gtrsim
    \max_{(g_1,g_2)\in\mathcal{G}}
    \frac{g_2^2}{\dfrac{d\tau^2}{n_0}\wedge\!\left(\dfrac{d\tau^2}{n}\vee g_1^2\right)},
\end{equation}
where $\mathcal{G}^{\mathrm{symm}}:=\mathcal{G}\cup\{(g_2,g_1):(g_1,g_2)\in\mathcal{G}\}$.
Finally, taking $\mathcal{G}=\{(g_1,g_2):0\le g_1\le g_2\le \tau/(2\sqrt{n_0})\}$ and optimizing the right-hand side of \eqref{eq:adaptation-cost-lb-m2} yields
\[
    \mathfrak{C}^\star([0,1]^2)
    \ge
    \mathfrak{C}^\star(\mathcal{G}^{\mathrm{symm}})
    \gtrsim
    \left(\frac{n}{dn_0}\right)\vee 1.
\]

In short, even in the two-source setting, uniform adaptation over a 
general bias configuration is obstructed by the need to distinguish 
between configurations that differ only in which source is closer to 
the target. This source-ordering difficulty is the fundamental mechanism 
underlying the adaptation cost, and it will also guide the construction 
of structured adaptable classes below.

\subsubsection{A structured configuration set for two sources}

The lower bound above shows that unrestricted adaptation over $[0,1]^2$ becomes impossible once
$n\gg d n_0$. On the other hand, by Theorem \ref{thm:cost-fixed-m} specialized to $m=2$,
uniform adaptation over the full two-source configuration space is achievable, up to logarithmic
factors, when $n\lesssim d n_0$. This naturally raises the following question: when
$n\gg d n_0$, is there a smaller and more structured subset of $[0,1]^2$ on which near-oracle
adaptation becomes possible?

We now identify a two-source configuration set $\cH_0\subset[0,1]^2$ for which the intrinsic cost of adaptation remains bounded up to logarithmic factors. The construction is guided directly by the lower-bound argument above. That argument shows that the main obstruction comes from having to adapt simultaneously to a configuration $(g_1,g_2)$ and to its reversal $(g_2,g_1)$, since these two points have the same unordered bias magnitudes but different best sources. Thus the issue is fundamentally an ordering obstruction.

\paragraph{Step 1: Necessity.}
For any pair $\{(g_1,g_2), (g_2,g_1)\}$
with $0\le g_1\le g_2\le \tau/(2\sqrt{n_0})$, the lower-bound argument gives
\[
\fC^\star\big(\{(g_1,g_2),(g_2,g_1)\}\big)
\gtrsim
\frac{g_2^2}{
	\left(\frac{d\tau^2}{n_0}\right)
	\wedge
	\left\{
	\left(\frac{d\tau^2}{n}\right)\vee g_1^2
	\right\}
}.
\]
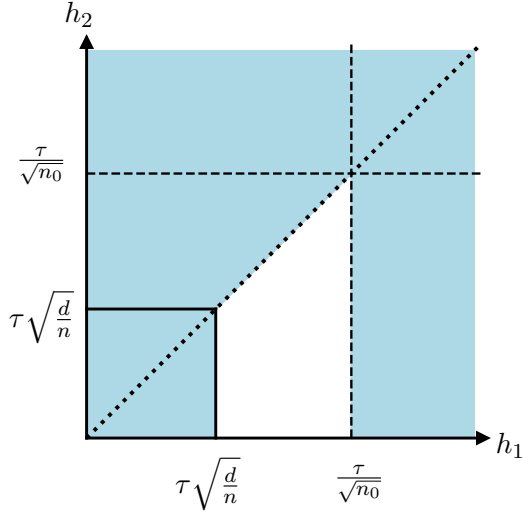
\begin{figure}
    \centering
    \centering
  \begin{tikzpicture}
		\node[inner sep=0pt] (fig) at (0,0)
		{\includegraphics[width=.4\textwidth]{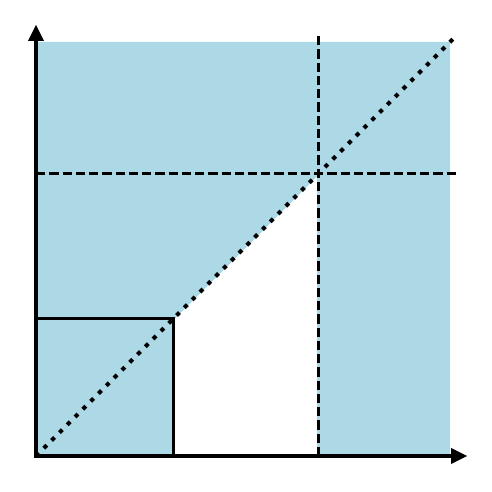}};
		\node  at (-1., -3.2) {$\tau\sqrt{\frac{d}{n}}$};
		\node  at (1, -3.2) {$\frac{\tau}{\sqrt{n_{0}}}$};
        \node  at (-3.2,-1) {$\tau\sqrt{\frac{d}{n}}$};
		\node  at (-3.2, 1) {$\frac{\tau}{\sqrt{n_{0}}}$};
		\node at (3, -2.7) {$h_1$};
		\node at (-2.7, 3.) {$h_2$};
	\end{tikzpicture}
	\caption{An adaptable set of bias configurations ($\cH_0$), presented as the blue region.}
	\label{fig:adaptable-bias-1}
\end{figure}

This expression is small only when reversing the order of the two sources does not substantially change the statistical difficulty. Accordingly, we isolate the largest part of the ordered region
$\{(g_1,g_2):0\le g_1\le g_2\le 1\}$ on which this lower bound remains of constant order.

Let $\cG_{\max}$ be a maximal subset of
$\{(g_1,g_2):0\le g_1\le g_2\le 1\}$ such that
\[
\max_{\substack{(g_1,g_2)\in\cG_{\max}:\\ 0\le g_1\le g_2\le \tau/\sqrt{n_0}}}
\frac{g_2^2}{
	\left(\frac{d\tau^2}{n_0}\right)
	\wedge
	\left\{
	\left(\frac{d\tau^2}{n}\right)\vee g_1^2
	\right\}
}
\le 1.
\]
A direct calculation shows that 
\[
\begin{aligned}
	\cG_{\max}
	&=
	\left\{
	(g_1,g_2):
	0\le g_1\le g_2\le \tau\sqrt{\frac dn}
	\right\}
    \cup
	\left\{
	(g_1,g_2):
	\tau\sqrt{\frac dn}<g_1=g_2\le \frac{\tau}{\sqrt{n_0}}
	\right\}
	\\
	&\quad\cup
	\left\{
	(g_1,g_2):
	\frac{\tau}{\sqrt{n_0}}<g_2\le 1,\ 0\le g_1\le g_2
	\right\}.
\end{aligned}
\]
These three pieces have a simple interpretation. In the first region, both biases are small enough that reversing their order is not yet costly. In the second, only the diagonal survives inside the critical square $[0,\tau/\sqrt{n_0}]^2$, since there the two sources are equally good. In the third, the larger bias already exceeds the target-noise scale, so the particular lower-bound mechanism above is no longer the dominant bottleneck.

An adaptable class need not be symmetric. The natural candidate is therefore obtained by keeping the entire ordered half-plane
\[
\{(g_1,g_2):0\le g_1\le g_2\le 1\},
\]
and then adding only the harmless part of the reverse-ordered region. We define
\[
\begin{aligned}
    \cH_0
:=
\cG_{\max}^{\mathrm{symm}}
\cup
\{(g_1,g_2):0\le g_1\le g_2\le 1\}, \\
\text{where} ~~\cG_{\max}^{\mathrm{symm}}
:=
\cG_{\max}
\cup
\{(g_2,g_1):(g_1,g_2)\in\cG_{\max}\}.
\end{aligned}
\]
Equivalently, $\cH_0$ contains every configuration with the canonical ordering $h_1\le h_2$, while in the reverse-ordered half-plane $h_2<h_1$ it contains only those configurations for which the lower bound does not already force a large adaptation cost; see Figure \ref{fig:adaptable-bias-1}.



This construction also makes it clear why some ordering restrictions are unavoidable. If
$g_2\gg \tau\sqrt{d/n}$ and $g_1\ll g_2$, then applying the same lower bound to
$\{(g_1,g_2),(g_2,g_1)\}$ yields
\[
\fC^\star\big(\cH_0\cup\{(g_2,g_1)\}\big)
\gtrsim
\frac{g_2^2}{
	\left(\frac{d\tau^2}{n}\right)\vee g_1^2
}
\gg 1.
\]
Thus, once a reverse-ordered point enters the critical region, the estimator is forced to distinguish between two configurations with identical unordered bias magnitudes but different best sources, and near-oracle adaptation becomes impossible. In this sense, the ordered structure built into $\cH_0$ is not an arbitrary modeling choice; it is dictated by the lower-bound.

\paragraph{Step 2: Sufficiency.}
We now show that the same restriction is sufficient, up to logarithmic factors. The estimator is
constructed by first forming one candidate transfer estimator from each source separately, and then
selecting between them according to which source appears closer to the target.

Fix a $\delta\in(0,1)$ and define the target confidence ball
\[
\cB_0
:=
\left\{
\theta:
\|\theta-\widetilde\theta_0\|_2
\le
\frac{\tau}{\sqrt{n_0}}
\left(\sqrt d+\sqrt{\log(1/\delta)}\right)
\right\}.
\]
For each $k\in\{1,2\}$, define the pooled estimator and the corresponding confidence ball
\[
\widetilde\theta_{0,k}
:=
\frac{n_0\widetilde\theta_0+n\widetilde\theta_k}{n_0+n}, \qquad \cB_{0,k}
:=
\left\{
\theta:
\|\theta-\widetilde\theta_{0,k}\|_2
\le
\frac{\tau}{\sqrt{n_0+n}}
\left(\sqrt d+\sqrt{\log(1/\delta)}\right)
\right\}.
\]
We then define
\[
\check\theta_k
:=
\widetilde\theta_{0,k}\,\bbI\{\cB_0\cap\cB_{0,k}\neq\varnothing\}
+
\widetilde\theta_0\,\bbI\{\cB_0\cap\cB_{0,k}=\varnothing\}.
\]
Thus, $\check\theta_k$ pools the target and the $k$th source only when the corresponding
confidence regions intersect; otherwise, it falls back to the target-only estimator. If the $k$th
source were known in advance to be the more useful one, then $\check\theta_k$ would already
achieve the optimal target--source tradeoff
\[
\fR((h_k),(n_0,n),\tau)
=
\frac{d\tau^2}{n_0}
\wedge
\left(
\frac{d\tau^2}{n_0+n}+h_k^2
\right)
\asymp
\frac{d\tau^2}{n_0}
\wedge
\left(
\frac{d\tau^2}{n}+h_k^2
\right),
\]
up to logarithmic factors; see Theorem \ref{thm:intersection-oracle-max}. The remaining task is, therefore, to identify which of the two sources has the smaller bias.
To this end, define
\begin{equation}\label{eq:estimator-m2}
	T
:=
\|\widetilde\theta_1-\widetilde\theta_0\|_2^2
-
\|\widetilde\theta_2-\widetilde\theta_0\|_2^2 \qquad \text{and} \qquad  \widehat\theta_0
	:=
	\begin{cases}
		\check\theta_1
		&
		\text{if }T\le \dfrac{\tau^2\log(1/\delta)}{n_0},
		\\[0.8em]
		\check\theta_2
		&
		\text{if }T> \dfrac{\tau^2\log(1/\delta)}{n_0}.
	\end{cases}
\end{equation}
The estimator combines the target with the first source when $T\le \tau^2\log(1/\delta)/n_0$; otherwise, it combines with the second source. This rule is
aligned with the geometry of $\cH_0$, which favors the first source throughout the canonical
ordered region $h_1\le h_2$ and admits reverse-order points only when they are statistically
benign. The shifted threshold is also chosen to match precisely the geometry of the adaptable set.

\begin{lemma} \label{lemma:upper-bound-cost-m2-restricted}
For a universal constant $C >0$ defined in Lemma \ref{lem:diff-sq-concentration}, assume that the $n$ and $n_0$ satisfy $n \ge 8C^2 n_0 \left({d}/{\log n} + 2\right)$.
Then, for a universal constant $c_U >0$ such that the estimator in \eqref{eq:estimator-m2} with  $\delta = n^{-2}$ satisfies
	\[
	\fC^\star(\widehat\theta_0,\cH_0)\le c_U(\log n)^2.
	\]
\end{lemma}

Therefore, the restriction from $[0,1]^2$ to $\cH_0$ is both necessary and sufficient, up to logarithmic factors. More importantly, the two-source analysis does more than show that unrestricted adaptation fails: it also reveals how an adaptable class should be organized. The set $\cH_0$ preserves a global ordering and admits reverse-order configurations only when they are statistically harmless. This same perspective suggests that, in the many-source setting, adaptable classes should again be organized by an appropriate ordering or ranking structure among the sources (see Section \ref{subsec:ordered-bias}). 

\subsection{General adaptation with a growing number of source domains}

\label{subsec:cost-general-m}

We now move from the fixed-$m$ regime to let $m$ grow. The two-source analysis above shows that adaptation becomes difficult once the learner must identify which sources are closest to the target. When $m$ grows, the same source-selection difficulty persists, but it is compounded by a much larger combinatorial search space, and the resulting adaptation cost increases with $m$. In this subsection, we quantify that cost over the full configuration space $[0,1]^m$ and then study what can still be achieved by a general model-selection procedure.

\subsubsection{A lower bound on the cost of general adaptation}

As before, we again assume equal source sample sizes so that
$n_1=\cdots=n_m=n$, and study the intrinsic cost of adaptation over the full bias configuration
space $[0,1]^m$. The following theorem gives a general lower bound in terms of $m$, $n$, $n_0$,
and $d$.


\begin{theorem}\label{thm:cost-lower-bound-general-m}
	There exists a universal constant \(c_L>0\) such that
	\[
	\mathfrak C^\star([0,1]^m)
	\;\ge\;
	c_L\max \left[
	\left\{\frac{n_0+mn}{d}\,
	\left(
	\frac{ d}{n_0}\wedge 
	\frac{1}{n} \sqrt{\frac{d}{m}}
	\right)\right\}\,,
	\;
	\frac{n_0+mn}{d\,n_0}\right].
	\]
\end{theorem}
While the detailed proof is deferred to Appendix~\ref{proof:thm:cost-lower-bound-general-m}, it is
helpful to note that the lower bound is already witnessed by a highly structured subclass in which
the source means form only two clusters. The proof sketch below explains how this clustered
structure yields the two terms in the theorem.




\smallskip

\paragraph{Proof sketch:} The lower bound is driven by two complementary clustered constructions, corresponding to the two terms in the maximum.

\medskip

\noindent
{\it First construction: random-sign two-cluster family.}
Consider a finite packing \(v_1,\dots,v_M\subset \bbS^{d-1}\) of the unit sphere with pairwise-separated directions and \(\log M \asymp d\). For each $k$, we place the target at \(\theta_0^\star=+\alpha v_k\) and assign each source independently, with an i.i.d.\ Rademacher sign, to one of the two cluster centers, $+\alpha v_k $ or $ -\alpha v_k$.
Conditional on the signs, the sources take only two distinct values, so the resulting model has the required two-cluster structure. At the same time, a positive fraction of the sources are aligned with the target, so the oracle risk is of order
\[
\fR(\bh,\bn,\tau)\asymp \frac{d\tau^2}{n_0+mn}.
\]
Contrarily, after marginalizing over the random signs, distinguishing the different directions \(v_k\) becomes difficult. The target still contributes a first-order Kullback–Leibler term of size \(n_0\alpha^2/\tau^2\), but the source contribution is only of second order. More precisely, one compares each resulting mixture law to a common Gaussian reference and bounds the corresponding divergence by \(\chi^2\): the one-source Rademacher mixture has a divergence of order \(n^2\alpha^4/\tau^4\), and hence the \(m\) independent source coordinates contribute \(mn^2\alpha^4/\tau^4\) overall. Therefore, the mixture laws indexed by different packed directions remain close provided
\[
\alpha^2
\asymp
\min\!\left\{
\frac{\tau^2 d}{n_0},
\,
\frac{\tau^2}{n}\sqrt{\frac{d}{m}}
\right\}
\]
and a Fano argument then yields the lower bound
\[
\inf_{\widehat\theta_0}
\sup_{\bh\in[0,1]^m}
\frac{
	\sup_{P\in\mathcal P(\bh,\bn,\tau)}
	\mathbb E_P\!\left[\|\widehat\theta_0-\theta_0^\star\|_2^2\right]
}{
	\fR(\bh,\bn,\tau)
}
\;\gtrsim\;
\frac{n_0+mn}{d\tau^2}\,
\min\!\left\{
\frac{\tau^2 d}{n_0},
\,
\frac{\tau^2}{n}\sqrt{\frac{d}{m}}
\right\}.
\]

\medskip

\noindent
{\it Second construction: deterministic balanced two-cluster family.}
The second construction (which is similar to the one in the two-source lower bound~\eqref{eq:two-source-lower-bound-construction}) is simpler and yields a complementary term. Choose two vectors \(\mu_1,\mu_2\in\mathbb R^d\) such that
\[
\|\mu_1-\mu_2\|_2^2\asymp \frac{\tau^2}{n_0},
\]
place half of the sources at \(\mu_1\) and the other half at \(\mu_2\), and let the target be either \(\mu_1\) or \(\mu_2\). Again, this model has the same two-cluster structure. In either case, an oracle can identify the correct cluster and pool the target with roughly \(m/2\) aligned sources, giving oracle risk:
\[
\cO\left(\frac{d\tau^2}{n_0+mn}\right).
\]
However, from the learner's point of view, deciding whether \(\theta_0=\mu_1\) or \(\theta_0=\mu_2\) is merely a two-point testing problem at the target noise scale \(\tau^2/n_0\). By Le Cam's method, this forces the squared-error risk to be at least of order \(\tau^2/n_0\), and hence
\[
\inf_{\widehat\theta_0}
\sup_{\bh\in[0,1]^m}
\frac{
	\sup_{P\in\mathcal P(\bh,\bn,\tau)}
	\mathbb E_P\!\left[\|\widehat\theta_0-\theta_0^\star\|_2^2\right]
}{
	\fR(\bh,\bn,\tau)
}
\;\gtrsim\;
\frac{n_0+mn}{d\,n_0}.
\]

Together, these two clustered constructions yield Theorem~\ref{thm:cost-lower-bound-general-m}.

\begin{remark}
The clustered lower-bound constructions above also fit the restricted bias structure
\(\Theta_0(h,\cA,\|\cdot\|)\) from \eqref{eq:bias-structure-1}, which has been studied in the earlier literature. Indeed, in both constructions, if we let \(\cA\) denote the set of sources that lie in the same cluster as the target, then for every \(k\in\cA\) we have \(\theta_k^\star=\theta_0^\star\). Hence each such instance belongs to \(\Theta_0(0,\cA,\|\cdot\|_2)\), and therefore also to \(\Theta_0(h,\cA,\|\cdot\|_2)\) for any \(h\ge 0\). In this sense, the hard family used in the proof is not only clustered, but is already contained in the more restrictive class considered in earlier work. Consequently, the same construction can also be used to quantify the cost of adaptation for that restricted model.
\end{remark}

Theorem \ref{thm:cost-lower-bound-general-m} reflects two distinct obstructions. The first construction captures the genuinely many-source difficulty: as $m$ grows, the source-selection problem becomes more complex, and the cost of adaptation increases with $m$. The second construction is closer to the two-source phenomenon discussed earlier and shows that even in a simple two-cluster model, adaptation already incurs a nontrivial cost at the target noise scale. In this sense, the growing-$m$ regime combines a new combinatorial difficulty with the basic source-selection obstruction already present for fixed $m$.

If $mn \lesssim n_0$, \ie, if the combined source sample size is no larger than the target sample size up to constants, then transfer learning is not essential: the target-only estimator $\widetilde \theta_0$ already attains the optimal rate $d\tau^2/n_0$. We therefore solely focus on the transfer-relevant regime $n_0 \ll mn$.
\begin{corollary}[Transfer-relevant regime \(n_0\ll mn\)]
	\label{cor:adaptation-cost-lb-general-m}
	If \(n_0\ll mn\), then
	\[
	\mathfrak C^\star([0,1]^m)
	\;\gtrsim\;
	\left(
	\frac{mn}{n_0}
	\wedge
	\sqrt{\frac{m}{d}}
	\right)
	\vee
	\frac{mn}{d\,n_0} \vee 1.
	\]
	in particular if \(n_0=n\), then
	\[
	\mathfrak C^\star([0,1]^m)
	\;\gtrsim\;
	\frac{m}{d}
	\;\vee\;
	1.
	\]
\end{corollary}

\subsubsection{An upper bound on the cost of general adaptation}

\label{subsubsec:model-selection}

To obtain a general upper bound, we now analyze a model-selection-based estimator.  Given a family $\cM \subset 2^{[m]}$ of subsets of source domains, the idea is to build a collection of candidate transfer estimators, one for each subset, and then use held-out target data to select the best candidate. The idea builds on the model-selection-based aggregation in \citet{tsybakov2003optimal} to make it suitable for multi-source transfer learning, where users have access to only $\{\widetilde \theta_k\}_{k = 1}^m$, the target data, and $\bn$.
We begin by splitting the target data into two equal parts and then constructing two independent estimators $\widetilde\theta_0^{(1)}$ and $\widetilde \theta_0^{(2)}$, such that both satisfy the sub-Gaussian assumption \eqref{eq:sub-gaussian} with an effective sample size $n_0/2$. We use $\widetilde \theta_0^{(2)}$ to construct candidate integrated estimators and $\widetilde \theta_0^{(1)}$ for validation.

Let $M = |\cM|$. Given an enumeration $\cM = \{S_1, \dots, S_M\}$, define the candidate estimators
\[
\xi _j = \frac{\sum_{k \in S_j} n_k \widetilde \theta_k + \frac{n_0}{2} \widetilde \theta_0^{(2)}}{\sum_{k \in S_j} n_k + \frac{n_0}{2}}, \quad j = 1, \dots, M\,.
\]
We then define the model-selection estimator as the candidate estimator with the minimum validation error:
\[
\widehat \theta_0^{\mathrm{MS}} := \xi_{\hat j}, \qquad \hat j:= \arg\min_{j = 1, \dots, M} \textstyle \left\|\xi_j - \widetilde \theta_0^{(1)}\right\|_2^2.
\]
\begin{theorem}\label{thm:model-selection}
	For $j = 1, \dots, M$, define $N_j := \sum_{k \in S_j} n_k + n_0 / 2$.  There exists a universal $C_1 > 0$ such that
	\[
	\begin{aligned}
		\Ex\left[\left\|\widehat \theta_0^{\mathrm{MS}} - \theta_0^\star\right\|_2^2\right] & \le C_1 \cdot \left[\min_{j = 1, \dots, M} \left(\textstyle \frac{ 1}{N_j^2} \Big \|\sum_{k \in S_j} n_k  (\theta_k^\star - \theta_0^\star) \Big\|_2^2  + \frac{d\tau^2}{N_j} \right)
		+\tau^2 \frac{\log M \wedge d}{n_0}\right]\,. 
	\end{aligned}
	\]
\end{theorem} 
This oracle inequality shows that the selected estimator performs nearly as well as the best candidate in the model class, up to the additional validation penalty of $\cO(\tau^2(\log M\wedge d)/n_0)$.

\begin{corollary}[Upper bound for intrinsic cost]
	\label{cor:adaptation-cost-ub-general-m}
	Focus on the transfer-relevant regime: $n_0 \ll mn$. Let $\cM = 2^{\{1, \dots, m\}}$ be the set of all candidate estimators. Then, the above theorem implies the following upper bound:
	\[
	 \fC^\star([0, 1]^m) \le \fC(\widehat \theta_0^{\mathrm{MS}}, [0, 1]^m)  \lesssim \left\{\left(1 \wedge \frac{m}{d}\right) \frac{m n }{n_0} \right\} \vee 1\,. 
	\]
\end{corollary}

\begin{corollary}\label{cor:adaptation-cost-lb-general-m-implications}
	Corollary \ref{cor:adaptation-cost-ub-general-m} shows that model selection achieves general adaptation when $m = \cO(d \wedge \sqrt{dn_0/ n})$. On the other hand, Corollary \ref{cor:adaptation-cost-lb-general-m} implies that general adaptation is impossible when $m \gg d \wedge (dn_0 / n)$. The lower and upper bounds, therefore, lead to the following conclusions.
	\begin{itemize}
		\item Let $d \lesssim  n_0/n $. In this regime, $\cO(d \wedge \sqrt{dn_0/ n})  = \cO(d \wedge (dn_0 / n)) = \cO(d)$. Consequently, the boundary $m = \cO(d)$ fully characterizes whether general adaptation is possible: general adaptation is achieved by model selection when $m= \cO(d)$, and general adaptation is impossible when $m \gg d$. 
		
		\item Alternatively, let $n_0 / n \ll d$. Then $\cO(d \wedge \sqrt{dn_0/ n}) = \cO( \sqrt{dn_0/ n}) \ll d \wedge dn_0/ n$. In this regime, Corollary \ref{cor:adaptation-cost-ub-general-m} implies that general adaptation is achieved by model selection when $m = \cO( \sqrt{dn_0/ n})$, while Corollary \ref{cor:adaptation-cost-lb-general-m} shows that general adaptation is impossible when $m \gg d \wedge dn_0/ n$. Since there is a gap between $\cO( \sqrt{dn_0/ n})$ and $\cO(d \wedge (dn_0 / n))$, the two corollaries do not fully characterize whether adaptation is possible in the intermediate regime $ \sqrt{dn_0/ n} \ll m \lesssim d \wedge (dn_0 / n)$.
	\end{itemize}

\end{corollary}
In summary, the analysis in Section~\ref{sec:bias-adaptation} suggests that the intrinsic cost of adaptation is governed by the difficulty of identifying useful sources. In the single-source case, the learner only faces the question of whether transfer should occur, and adaptation is correspondingly more tractable. In contrast, for multiple sources, uncertainty about which subset of domains yields the best bias-variance trade-off leads to qualitatively harder behavior, reflected in the fixed-$m$ result of Theorem~\ref{thm:cost-fixed-m}. In the growing-$m$ regime, Theorem~\ref{thm:cost-lower-bound-general-m} and Corollary~\ref{cor:adaptation-cost-ub-general-m} further show that this difficulty acquires a combinatorial dimension. Overall, these results help explain why unrestricted adaptation can fail in multi-source transfer learning, and they motivate the additional structural assumptions introduced in Section~\ref{sec:adaptation-general}.
\section{Adaptable configurations under general settings} \label{sec:adaptation-general}

In the previous section, we analyzed the cost of general adaptation and identified several scenarios where adapting to the bias is impossible. Given these limitations, it is natural to explore specific conditions under which the cost of adaptation is meaningfully reduced. In this section, we investigate three such restrictions, focusing on reduced configuration sets and potentially imposing additional constraints on the parameters $\{\theta_k^\star\}_{k= 0}^m$.


Specifically, we explore three structural restrictions motivated by the existing literature, each serving as a mechanism to incorporate available prior knowledge. In Section~\ref{subsec:ordered-bias}, we consider source domains arranged in increasing order of their bias levels \citep{hanneke2022no}. In Section~\ref{sec:clustering-method}, we assume the parameters $\{\theta_k^\star\}_{k= 0}^m$ exhibit a clustering structure \citep{duan2023adaptive}. Finally, in Section~\ref{sec:test-then-combine}, we posit that the parameters $\theta_k^\star$ for non-integrable source domains ($S_{\bh}^\complement$ in Remark~\ref{remark:bias-variance-tradeoff}) are sufficiently distant from $\theta_0^\star$ to be detected via a simple hypothesis test \citep{tian2023transfer}. 
By leveraging these restrictions, we propose tailored methodologies for each setting that substantially reduce the cost of adaptation.

Before proceeding to the individual subsections, a final remark is in order. The methodologies discussed in Sections \ref{subsec:ordered-bias} and \ref{sec:test-then-combine} require the scale parameter $\tau$ as input. Because our primary focus in this paper is on adaptation to $\bh$ rather than $\tau$, our analysis assumes that the true value of $\tau$ is known. Nevertheless, for practical applications, we propose a data-driven proxy for $\tau$ in Appendix \ref{supp:tau-proxy} when it is not known.

\subsection{Adaptation under ordered bias}
\label{subsec:ordered-bias}

Returning to the adaptable configuration set for two source domains (Figure~\ref{fig:adaptable-bias-1}), we observe that the region $ \{(h_1,h_2): 0 \le h_1 \le h_2 \le 1\}$ lies entirely within the configuration set. This suggests a natural structural assumption: the bias in the first source domain is no larger than that in the second. 
Motivated by this observation, for an arbitrary $m$, $\bn$, and $d$, we assume that the source domains are arranged according to their bias levels. The assumption is formalized as the following configuration set:


\begin{assumption}[Ordered biases]
\label{assmp:ordered-biases}
The source domains are ordered according to their bias levels $h_k$, \ie, $h_1 \le \cdots \le h_m$. Accordingly, we restrict the configuration set to
\[
\cH_{\mathrm{ordered}}
:=
\Bigl\{
\bh \in [0,1]^m : h_1 \le \cdots \le h_m
\Bigr\}.
\]
\end{assumption}

The above assumption substantially simplifies the optimization within the oracle rate~\eqref{eq:loca-loptimal-rate}. In the unrestricted setting, the optimization ranges over all possible subsets $\{0\} \subseteq S \subseteq \{0,\dots,m\}$, yielding $2^m$ candidates. Under the ordered bias assumption, however, we only need to consider prefix sets of the form
\[
S = \{0,\dots,k\}, \qquad k=0,\dots,m.
\]
This structural constraint substantially reduces the search space from $2^m$ arbitrary subsets to just $m+1$ prefixes.

Recall the model selection procedure discussed in Section~\ref{subsubsec:model-selection}: given a class $\cM $ of candidate sets, the resulting estimator incurs an adaptation penalty of order $\cO((\log |\cM| \wedge d)/n_0)$
(\cf\ Theorem~\ref{thm:model-selection}). Since the number of candidates is now reduced from $M = 2^m$ to $M = m+1$, one might intuitively hope to apply this procedure directly, thereby improving the penalty term from
\[
\textstyle \cO\!\left(\frac{m \wedge d}{n_0}\right)
\quad \text{to} \quad
\cO\!\left(\frac{\log m \wedge d}{n_0}\right).
\]
Unfortunately, even this reduced penalty is generally too large to achieve adaptation; specifically, the procedure fails because
\[
\textstyle \fR(\bh,\bn,\tau) \ll \frac{\log m \wedge d}{n_0}
\qquad \text{whenever} \qquad
h_m^2 \vee \frac{d\tau^2}{N_{\mathrm{total}}}
\ll \frac{\log m \wedge d}{n_0}.
\]

To successfully achieve adaptation, we now propose an alternative estimator that finds the optimal bias-variance trade-off by looking at a running intersection of  a sequence of high-probability confidence regions for $\theta_0^\star$. This construction is partly inspired by the rank-based procedure of \citet[Section~4.2]{hanneke2022no}, who utilizes an analogous intersection-based approach for a related transfer learning problem in binary classification.

\paragraph{The Intersection-Based Estimator.}
For $r = 0,\dots,m$, define the prefix estimators 
\begin{equation}
\label{eq:prefix-estimator-section}
\widehat\theta^{(r)}
:=
\sum_{k=0}^r w_{k,r}\,\widetilde\theta_k,
\qquad
w_{k,r}
:=
\frac{n_k}{\sum_{j=0}^r n_j}.
\end{equation}
Then, given a confidence level $\delta \in (0,1)$, define the confidence regions
\begin{equation}
\label{eq:rad-def-section}
\mathcal B_r
:=
\left\{
\theta \in \reals^d :
\|\theta - \widehat\theta^{(r)}\|_2 \le 2\rho_r
\right\}, \quad \rho_r
:=
\textstyle \frac{\tau}{\sqrt{\sum_{j=0}^r n_j}}
\left(
\sqrt d + \sqrt{\log\frac{m+1}{\delta}}
\right).
\end{equation}
Finally, the estimator is defined by selecting the largest $t$ such that all confidence regions up to $t$ intersect:
\begin{equation}
\label{eq:t-hat-section}
\textstyle \widehat\theta_0
:=
\widehat\theta^{(\hat t)},
\qquad
\hat t
:=
\max\left\{
t : \bigcap_{r=0}^t \mathcal B_r \neq \varnothing
\right\}.
\end{equation}
The following theorem provides an appropriate specification for $\delta$ and demonstrates that $\widehat\theta_0$ successfully adapts over $\cH_{\mathrm{ordered}}$.

\begin{theorem}[Adaptation with the intersection estimator]
\label{thm:intersection-oracle-max} 
There exists a universal constant $C>0$ such that, for any $\bh \in \cH_{\mathrm{ordered}}$,
\begin{equation}
\label{eq:intersection-oracle-final-max}
\sup_{P \in \cP(\bh, \bn, \tau)} P \left( \big\|\hat\theta_0-\theta_0^\star\big\|_2^2
> 
C
\cdot \fR(\bh, \bn, \tau)  \log\left(\frac{m+1}{\delta}
\right) \right) \le \delta\,.
\end{equation}
Moreover, by setting $\delta = d\tau^2 /N_{\mathrm{total}}$, there exists a constant $C_1 > 0$ such that 
\[
\fC(\widehat \theta_0, \cH_{\mathrm{ordered}}) \le C_1 \log \left(\frac{m N_{\mathrm{total}}}{d \tau^2}\right)\,.
\] 
Consequently, $\widehat \theta_0$ adapts to the unknown bias within $\cH_{\mathrm{ordered}}$ up to a logarithmic factor.
\end{theorem}

\begin{remark}[Adaptation for $m = 1$] 
\label{remark:adaptation-m-1}
When there is only one source domain ($m=1$), the ordering of source domains trivially holds. Consequently, a direct application of the preceding theorem indicates that the intrinsic cost of adaptation is $\mathcal{O}(1)$ up to a logarithmic factor, implying that adaptation is generally feasible. 
    
\end{remark}

As an interesting digression, our adaptation result can be directly applied to non-parametric regression with a fixed design. This yields a $k$-nearest neighbor (kNN) estimator with a data-driven choice of $k$ that automatically adapts to the underlying smoothness of the regression function. We detail this application below.

\begin{remark}[Adaptive kNN regressor for non-parametric regression]
    Suppose we observe data $\{(X_k, Y_k)\}_{k = 1}^m$, where $X_k = k/m$ forms a uniform fixed design on $[0,1]$, and $Y_i = f(X_i) + \varepsilon_i$. Here, $\{\varepsilon_i\}_{i = 1}^m$ are i.i.d.\ zero-mean sub-Gaussian random variables with variance proxy $1$, and $f: [0, 1] \to \reals$ is the unknown regression function. Assume that $f$ is $\alpha$-H\"older smooth for some $0 < \alpha \le 1$; that is, there exists an $L > 0$ such that
    \[
    |f(x) - f(x')| \le L |x - x'|^\alpha \qquad \text{for all } x, x' \in [0, 1]\,.
    \] 
    Our goal is to estimate the function value at a specific point $x_0 \in [0, 1]$, leading to the target parameter $\theta_0^\star = f(x_0)$, and the estimation error $\Ex \left[\bigl\{\hat f(x_0) - f(x_0)\bigr\}^2\right]$.

    To apply our intersection-based estimator, we treat each data point $(X_k, Y_k)$ as an individual source domain. We set $\widetilde \theta_k = Y_k$, $\theta_k^\star = f(X_k)$, $n_k = 1$, and $\tau = 1$. For the target domain, we set $\widetilde \theta_0 = 0$ and $n_0 = 0$. The smoothness condition directly bounds the bias:
    \[
    |\theta_k^\star - \theta_0^\star| = |f(X_k) - f(x_0)| \le L |X_k - x_0|^\alpha := h_k\,. 
    \] 
    Although the exact values of $h_k$ are unknown when $\alpha$ is unknown, $h_k$ is strictly increasing with respect to the distance $|X_k - x_0|$. Therefore, we can naturally order the sample $\{(X_k, Y_k)\}_{k = 1}^m$ by sorting it according to $|X_k - x_0|$. Let the rearranged sample be $\{(X_{(k)}, Y_{(k)})\}_{k = 1}^m$, with corresponding ordered bias levels $\{h_{(k)}\}_{k = 1}^m$. 

    Applying the intersection-based estimator to this rearranged sample, the prefix estimators $\widehat \theta^{(r)}$ defined in \eqref{eq:prefix-estimator-section} correspond exactly to the kNN estimators for $f(x_0)$ using $r$ neighbors. The final estimator from \eqref{eq:t-hat-section} becomes $\hat f(x_0) := \widehat \theta^{(\hat t)}$, where $\hat t$ serves as a purely data-driven choice for the number of neighbors. Theorem \ref{thm:intersection-oracle-max} guaranties that this choice adapts to the oracle rate $\fR(\bh,\bn, \tau)$, modulo a logarithmic factor.

    Examining the oracle rate more closely, we find:
    \[
    \fR(\bh,\bn, \tau) = \min_{1 \le k \le m} \left\{h_{(k)}^2 + \frac{1}{k} \right\} \asymp   \min_{1 \le k \le m} \left\{\left(\frac{k}{m}\right)^{2\alpha} + \frac{1}{k} \right\}\,,
    \] 
    since $h_{(k)} \asymp (k/m)^\alpha$. This optimal bias-variance trade-off is achieved at $k^\star \asymp m^{2\alpha / (2\alpha + 1)}$, yielding the rate
    \[
    \fR(\bh,\bn, \tau) = \cO \left(m^{-\frac{2\alpha}{2\alpha+1}}\right)\,.
    \]
    This achieves the minimax optimal rate for nonparametric regression under an $\alpha$-H\"older smoothness assumption \citep{tsybakov2008nonparametric_book}. Consequently, our framework yields a $k$-NN regressor that automatically adapts to the unknown smoothness without requiring prior knowledge of $\alpha$. Furthermore, our method for selecting $\hat{t}$ via intersection closely parallels the adaptive bandwidth selection strategy for local polynomial estimators based on the intersection of confidence intervals, as proposed by \citet{katkovnik2002adaptive}.
\end{remark}

\subsection{Adaptation Under a Clustered Source Structure} \label{sec:clustering-method}

As established in Corollary~\ref{cor:adaptation-cost-lb-general-m}, adapting to a general bias configuration becomes increasingly challenging as  $m$ grows. A generic strategy to address this is the model selection estimator described in Section~\ref{subsubsec:model-selection}; however, as Corollaries~\ref{cor:adaptation-cost-ub-general-m} and \ref{cor:adaptation-cost-lb-general-m-implications} demonstrate, in the balanced regime ($n_0 = \dots = n_m = n$), the upper bound for the cost of general adaptation restricts the success of this strategy to $m = \cO(\sqrt{d})$. 

To overcome this $m = \cO(\sqrt{d})$ bottleneck, we assume in this subsection that the local parameters $\{\theta_k^\star\}_{k = 1}^m$ are clustered into a small number of groups. This aligns with standard assumptions in personalized federated learning/multitask learning, where clients can often be partitioned into a few latent groups that share similar underlying distributions or parameters \citep{sattler2020clustered,mansour2020three,long2023multi,duan2023adaptive}. By exploiting this clustering structure, we can aggregate the local estimators $\widetilde{\theta}_k$ from source domains within the same cluster. This aggregation effectively reduces the number of distinct sources, enabling us to satisfy the required condition $m_{\mathrm{effective}} = \cO(\sqrt{d})$ for successful adaptation. 

To provide provable guaranties for this intuition, the remainder of this section focuses on a two-cluster setting with $n_0 = \dots = n_m = n$. In this specific regime, under this structure, the restriction can be relaxed from $m = \cO(\sqrt{d})$ to $m = \cO(d)$. Notably, our lower bound on the cost of adaptation (Theorem~\ref{thm:cost-lower-bound-general-m}, and specifically Corollary~\ref{cor:adaptation-cost-lb-general-m}) dictates that any further relaxation is provably impossible (since the lower bound constructions there are essentially clustering-based). 

Exploiting this cluster structure, however, comes at a cost. If the separation between clusters is too small, the uncertainty inherent in recovering the true latent partition introduces an additional estimation error. This cost must be carefully accounted for in the final risk bound. As we establish below, analyzing a simple two-cluster source structure with equal sample sizes is sufficient to capture this core statistical trade-off.

\paragraph{Two-Cluster Model with Equal Sample Sizes.}
To isolate the statistical difficulty as clearly as possible, we restrict our attention to a scenario where all domains share an equal sample size. We assume the source parameters $\{\theta_k^\star\}_{k = 1}^m$ exhibit a two-cluster structure. Specifically, there exist two vectors \(\mu_1, \mu_2 \in \reals^d\) and a disjoint partition 
\[
[m] =G_1 \cup G_2, \qquad G_1 \cap G_2 = \varnothing,
\] 
such that for any $k \in [m]$ and $r \in \{1, 2\}$, we have $\theta_k^\star = \mu_r$ whenever $k \in G_r$. We define the following key quantities: the cluster sizes $m_r := |G_r|$, the minimum cluster size $m_{\min} := m_1 \wedge m_2$, the between-cluster separation $\Delta := \|\mu_1 - \mu_2 \|_2$, and the bias magnitude of the $r$-th cluster relative to the target domain $b_r := \|\theta_0^\star - \mu_r\|_2$. For our analysis, we further assume that the cluster sizes are balanced; that is, there exists a constant $0 < c < 1/2$ such that $m_{\min} \ge c m$. 

Under this two-cluster structure, we can explicitly characterize the oracle rate from \eqref{eq:loca-loptimal-rate}, which serves as the foundation for our subsequent analysis. This model can be embedded into the general bias structure in \eqref{eq:bias-structure-gen} using the specific bias configuration $\bh^{\mathrm{cl}}$, defined as:
\[
h_k^{\mathrm{cl}}
=
\begin{cases}
b_1, & k \in G_1,\\
b_2, & k \in G_2.
\end{cases}
\] 
Due to this highly structured configuration, we only need to evaluate the candidate sets $S \in \{\varnothing, G_1, G_2, [m]\}$ to determine the oracle rate in \eqref{eq:loca-loptimal-rate}. This evaluation yields the rate:
\[
\fR(\bh^{\mathrm{cl}}, \bn, \tau) \asymp \min \left\{ \frac{d\tau^2}{n}, \frac{d\tau^2}{m_1 n} + b_1^2, \frac{d\tau^2}{m_2 n} + b_2^2, \frac{d\tau^2}{m n} + b_1^2 \vee b_2^2 \right\}\,.
\] 
Finally, applying $m_{\min} \ge cm$, the oracle rate simplifies to:
\begin{equation}\label{eq:oracle-rate-cluster}
    \fR(\bh^{\mathrm{cl}}, \bn, \tau) \asymp \min \left\{ \frac{d\tau^2}{n},  \frac{d\tau^2}{m n} + b_1^2 \vee b_2^2 \right\} := M_{\mathrm{cl}}\,.
\end{equation}

\paragraph{Adaptive Estimator.} 
Our proposed estimator is formally detailed in Algorithm~\ref{alg:two-cluster-adaptive-short}. It consists of two primary components: a clustering step to recover the latent source partition (detailed in Algorithm~\ref{alg:sample-split-clustering}) and a target-side model selection step to choose among the resulting candidate estimators. Throughout these steps, we heavily utilize sample splitting. In the $k$-th domain, we split the sample $\cD_k$ and construct three independent estimators $\widetilde \theta_k^{(1)}, \widetilde \theta_k^{(2)}, \widetilde \theta_k^{(3)}$ of the local parameter $\theta_k^\star$. Relying on these independent estimators provides crucial technical advantages for the theoretical analysis of the algorithm's performance (as seen in the simulation studies, sample splitting is not necessary in practice). 

\begin{algorithm}[t]
\caption{Adaptive Transfer Under a Two-Cluster Source Structure}
\label{alg:two-cluster-adaptive-short}
\begin{algorithmic}[1]
\Require Domain-specific samples \(\{\cD_k\}_{k = 0}^m\)

\State Split the target sample $\cD_0$ into two independent halves to compute $\widetilde\theta_0^{(1)}$ and $\widetilde\theta_0^{(2)}$.

\State For each source \(k\in[m]\), split \(\mathcal D_k\) into three independent parts to compute $\widetilde\theta_k^{(1)}, \widetilde\theta_k^{(2)}$, and $\widetilde\theta_k^{(3)}$. 

\State Apply Algorithm~\ref{alg:sample-split-clustering} to \(\{\widetilde\theta_k^{(1)},\widetilde\theta_k^{(2)}\}_{k=1}^m\) to obtain the clustering partition $\{\hat G_1, \hat G_2\}$.

\State Form the three candidate estimators:
\[
\hat\theta^{(0)}:=\widetilde\theta_0^{(1)},
\qquad
\hat\theta^{(r)}
:=
\frac{\widetilde\theta_0^{(1)}+\sum_{i\in\hat G_r}\widetilde\theta_k^{(3)}}{1+|\hat G_r|},
\quad r \in \{1,2\}.
\]
\State Select the best candidate via:
\[
\hat r \in \arg\min_{j\in\{0,1,2\}}
\bigl\|
\widetilde\theta_0^{(2)}-\hat\theta^{(j)}
\bigr\|_2^2.
\]
\Return $\widehat \theta_0 := \hat \theta^{(\hat r)}$. 
\end{algorithmic}
\end{algorithm}

\begin{algorithm}[t]
\caption{Sample-Split Clustering of the Source Estimators}
\label{alg:sample-split-clustering}
\begin{algorithmic}[1]
\Require Domain-specific estimators \(\{\widetilde\theta_k^{(1)},\widetilde\theta_k^{(2)}\}_{k=1}^m\)

\State Using \(\{\widetilde \theta_k^{(1)}\}_{k=1}^m\), compute the sample mean and covariance:
\[
\bar \theta^{(1)} := \frac{1}{m}\sum_{k=1}^m \widetilde \theta_k^{(1)},
\qquad
\hat\Sigma^{(1)} := \frac{1}{m}\sum_{k=1}^m
\bigl(\widetilde \theta_k^{(1)} - \bar \theta^{(1)}\bigr)\bigl(\widetilde \theta_k^{(1)} - \bar \theta^{(1)}\bigr)^\top.
\]
\State Let \(\hat u\) be a leading eigenvector of \(\hat\Sigma^{(1)}\). Run one-dimensional \(2\)-means clustering on the projections:
\[
t_k^{(1)} := \hat u^\top \bigl(\widetilde \theta_k^{(1)} - \bar \theta^{(1)}\bigr),
\qquad k \in [m],
\]
to obtain cluster centers \(\hat c_1 < \hat c_2\).

\State For each \(k \in [m]\), classify the independent second-half estimator \(\widetilde \theta_k^{(2)}\) according to:
\[
\hat z_k \in \arg\min_{r \in \{1,2\}}
\left|
\hat u^\top \bigl(\widetilde \theta_k^{(2)} - \bar \theta^{(1)}\bigr) - \hat c_r
\right|.
\]
\Return the estimated clusters:
\[
\hat G_r := \{ k \in [m] : \hat z_k = r \},
\qquad r \in \{1,2\}.
\]
\end{algorithmic}
\end{algorithm}

\begin{theorem}[MSE Bound for the Adaptive Estimator]
\label{thm:two-cluster-mse}
Let \(\widehat \theta_0\) be the estimator defined in Algorithm~\ref{alg:two-cluster-adaptive-short}. There exists a universal constant $C > 0$ such that
\[
\begin{aligned}
    \Ex\left[\|\widehat \theta_0-\theta_0^\star\|_2^2\right] \;\le\; C\left(M_{\mathrm{cl}}+\frac{\tau^2}{n}\widetilde\Delta_{m,d}+\frac{\tau^2}{n} \right),\\
    \text{where} \quad \widetilde\Delta_{m,d}
:=
\frac{d+\log(md)}{m}
+
\frac{\sqrt{m\bigl(d+\log(md)\bigr)}}{m}
+
\log(md).
\end{aligned}
\]
\end{theorem} 

The MSE upper bound consists of three distinct components. The first term, $M_{\mathrm{cl}}$, is the oracle rate under the two-cluster model previously established in \eqref{eq:oracle-rate-cluster}. The third term, proportional to $\tau^2/n$, corresponds to the penalty incurred from the model selection step. Recall that in Theorem~\ref{thm:model-selection}, model selection incurs a penalty of $\cO(\tau^2(\log M \wedge d)/n_0)$. In our current setting, $n_0 = n$ and the selection step in Algorithm~\ref{alg:two-cluster-adaptive-short} choose among $M = 3$ estimators, which simplifies the penalty to $\cO(\tau^2/n)$. Finally, the middle term, $\tau^2 \widetilde \Delta_{m, d}/n$, captures the estimation error inherent to recovering the cluster partitions themselves.
The recovery analysis underlying the middle term follows the same general template as standard exact-recovery arguments for two-component sub-Gaussian mixtures, combining a spectral estimate of the cluster direction with a one-dimensional separation argument; see, for example, \citet{fei2018hidden,jiang2020recovery}. Our contribution is to integrate this recovery step into the transfer-learning risk analysis and to quantify its effect relative to the oracle benchmark $M_{\mathrm cl}$.

 We end this subsection with a corollary that discusses the adaptation when $m = \cO(d)$.
\begin{corollary}[Adaptation up to Logarithmic Cost when \(m= \cO( d)\)]
\label{cor:two-cluster-log-adaptation}
Assume the setting of Theorem~\ref{thm:two-cluster-mse}. If the number of sources scales as $m = \cO(d)$, then
\[
\Ex\left[\|\widehat \theta_0-\theta_0^\star\|_2^2\right]
\;\le\;
C\cdot \log(md) \cdot\fR(\bh^{\mathrm{cl}},\bn,\tau).
\]
Consequently, for \(m = \cO(d)\), adaptation is achievable up to a logarithmic factor under the two-cluster source structure.
\end{corollary}

\subsection{Eliminating non-informative sources}
\label{sec:test-then-combine}

In the previous subsection, we explored a two-cluster model with equal sample sizes. We established that adaptation is possible when $m = \cO(d)$, a significantly more relaxed condition than the $m = \cO(\sqrt{d})$ requirement for the model selection estimator (\cf\ Section \ref{subsubsec:model-selection}). While this simple model effectively captures the core statistical challenges of clustering-based estimators, it may not reflect a generic practical scenario. As an alternative, in this subsection we introduce a separation-based structure. This approach accommodates arbitrary $m$ and varying sample sizes, albeit under stricter separation requirements. 

Our separation structure is motivated by \citet{tian2023transfer}. While they investigate a parameter set $\Theta_0(h, \cA, \|\cdot \|_1)$ based on a fixed set of integrable domains $\cA \subset \{1, \dots, m\}$, we instead seek a sufficient condition on $\bh$ that makes $S^\star := \cA \cup \{0\}$ the optimal set for the oracle rate in \eqref{eq:loca-loptimal-rate}. Furthermore, for any source $k \notin S^\star$, we assume the true distance $\|\theta_k^\star - \theta_0^\star\|_2$ is large enough to be reliably detected via the empirical distance $\|\widetilde \theta_k - \widetilde \theta_0\|_2$. Once these non-informative sources ($k \notin S^\star$) are detected and discarded with high confidence, we simply aggregate the remaining local estimators $\widetilde\theta_k$ to construct our final estimator for $\theta_0^\star$.

We begin by establishing a sufficient condition on the $\bh$. For convenience, let $h_0 = 0$. Given a subset $S$ such that $\{0\} \subset S \subset \{0, \dots, m\}$, we define:
\[
h_S := \max_{k \in S} h_k, \quad N_S := \sum_{k\in S} n_k\,. 
\] 
The oracle optimal rate in \eqref{eq:loca-loptimal-rate} is given by:
\[
\fR(\bh, \bn, \tau) = \min\limits_{\{0\} \subset S \subset \{0, \dots, m\}} \left(h_S^2 + \frac{d\tau^2}{N_S}\right) \asymp \min\limits_{\{0\} \subset S \subset \{0, \dots, m\}} \left(h_S^2 \vee \frac{d\tau^2}{N_S}\right)\,. 
\]
In the following lemma, we provide necessary and sufficient conditions on $\bh$ such that a prespecified set $S^\star$  minimizes the above rate $f(S) := h_S^2 \vee ({d\tau^2}/{N_S})$.

\begin{lemma} \label{lemma:minimizer-iff-bias-1}
 The set $S^\star$ is a minimizer of $\min\limits_{\{0\} \subset S \subset \{0, \dots, m\}} f(S)$ if and only if  $S^\star = \{k \in \{0, \dots, m\} \mid h_k \le h^\star\}$, where 
\[
(h^\star)^2 \le  {\frac{d\tau^2}{\sum_{k: h_k < h^\star} n_k}}\,,
\]
and $h_k > \max \{ h^\star , \tau \sqrt{d/N_{S^\star}}\}$ for any $k \notin S^\star$. 
\end{lemma} 
Building on this lemma, we define our configuration sets. Given sets $S$ and $T$ such that $\{0\} \subset T \subsetneq S \subset \{0, \dots, m\}$, and a bias threshold $h \le \tau \sqrt{d/ N_T}$, we define $\cH(S, T, h)$ as the set of all configurations $\bh$ satisfying:
\[
h_k = \begin{cases}
    < h & \text{if } k \in T\\
    h & \text{if } k \in S \setminus T\\
    > \max \left \{h, \tau \sqrt{\frac{d}{N_S}}\right\} & \text{if } k \notin S \,. 
\end{cases}
\] 
For any bias configuration $\bh \in \cH(S, T, h)$, Lemma \ref{lemma:minimizer-iff-bias-1} implies that $S$ optimally balances the bias-variance trade-off for the oracle rate \eqref{eq:loca-loptimal-rate}. Collectively, we define:
\begin{equation}
    \cH(S) = \bigcup_{\{0\} \subset T \subsetneq S } \bigcup_{0 \le h \le \tau \sqrt{\frac{d}{N_T}}} \cH(S, T, h)\,,
\end{equation} 
where $S$ serves as the optimal set. For the base case of $S = \{0\}$, we define $\cH(\{0\}):= \{\bh \mid h_k > \tau \sqrt{d/n_0}\}$ such that for any $\bh \in \cH(\{0\})$, the target domain $\{0\}$ alone is the optimal set for the oracle rate in \eqref{eq:loca-loptimal-rate}.

Next, to eliminate a non-integrable source domain ($k \in S^\complement$) with high confidence, we require sufficient separation between $\theta_k^\star$ and $\theta_0^\star$. 
Given a $\kappa > 0$, we assume:
\begin{equation}\label{eq:separation-pairwise-testing}
    \|\theta_k^\star - \theta_0^\star \|_2 > \kappa \tau \max\left\{\sqrt{\frac{d}{n_0}}, \sqrt{\frac{d}{n_k}}\right\}\,.
\end{equation}
The precise value of $\kappa$ is specified in Theorem \ref{thm:adaptation-test-then-combine}. 
Given the above separation condition, we redefine our probability class as follows:

\begin{definition}
    Given a set $\{0\} \subset S \subset \{0, \dots, m\}$, parameters $\bn$, $\tau, \kappa > 0$, and a bias configuration $\bh \in \cH(S)$, we define $\cP'(\bh, \bn, \tau, \kappa, S)$ as the class of all probability distributions of $(\widetilde \theta_0, \dots, \widetilde \theta_m)$ such that:
    \begin{enumerate}
        \item The estimators $\widetilde \theta_0, \dots, \widetilde \theta_m$ are independent and satisfy \eqref{eq:sub-gaussian}.
        \item For $k \in A$, the bias satisfies $\|\theta_k^\star - \theta_0^\star\|_2 \le h_k$; for $k \notin A$, the separation condition in \eqref{eq:separation-pairwise-testing} holds with scale $\kappa \tau$. 
    \end{enumerate} 
\end{definition}

\begin{remark}[Separation conditions in the existing literature]
    A similar condition appears in \citet[Assumption 5]{tian2023transfer}, which suffices to establish a high-probability bound for correctly identifying their integrable set. However, compared to \citet[Proposition 1]{blanchard2024estimation}, which requires a separation of $\cO(d^{1/4})$, we assume a stronger separation of $\cO(d^{1/2})$. In high-dimensional mean estimation, \citeauthor{blanchard2024estimation} shows that $\cO(d^{1/4})$ is sufficient when utilizing an unbiased estimate of $\|\theta_k^\star - \theta_0^\star\|_2^2$. By contrast, we operate under a general parameter estimation setting that satisfies \eqref{eq:sub-gaussian} and estimate the parameter difference simply as $\|\widetilde \theta_k - \widetilde \theta_0\|_2^2$. This is a biased estimate. While characterizing the bias and constructing an (approximate) unbiased estimate is non-trivial in our general framework, specific applications might allow for an unbiased estimate of $\|\theta_k^\star - \theta_0^\star\|_2^2$, potentially relaxing the required separation to $\cO(d^{1/4})$.
\end{remark}

\paragraph{The Estimator:} 
Our estimation strategy is inspired by \citet[Algorithm 2]{tian2023transfer}. First, we identify a suitable set of integrable domains by eliminating those with excessive empirical distance from the target. Then, we compute a weighted average of the remaining local estimators to produce the final estimate:
\begin{equation}
    \label{eq:test-then-combine}
    \widehat S := \left \{ k \in [m]: \|\widetilde \theta_k - \widetilde \theta_0 \|_2 \le \tau\sqrt{\frac{d\alpha \log(n_0 \vee n_k )}{n_0 \wedge n_k }} \right\} \cup \{0\}\,, \qquad  \widehat \theta_0 := \frac{\sum_{k \in \widehat S} n_k \widetilde \theta_k}{\sum_{k \in \widehat S} n_k}\,. 
\end{equation} 

In the following theorem, we demonstrate that for an appropriately chosen $\alpha > 0$, this estimator successfully adapts to the unknown bias within the reduced probability class $\cP'(\bh, \bn, \tau, \kappa, S)$.

\begin{theorem}\label{thm:adaptation-test-then-combine} 
Let $\kappa = 3{\alpha^{1/2}}$. For any $\{0\} \subset S \subset \{0, \dots, m\}$, configuration $\bh \in \cH(S)$, and distribution $P \in \cP'(\bh, \bn, \tau, \kappa, S)$, the correct integrable domains are identified with high probability:
\[
P\left(\widehat S = S\right) \ge  1 - \frac{m}{n_0^{\alpha d}}\,.
\] 
Furthermore, for $\alpha \ge \max\left\{\frac{9^2 \log \left\{{m N_{\mathrm{total}}}/(d \tau^2)\right\}}{d \log (n_0)}, 1\right\}$, where $N_{\mathrm{total}} := \sum_{k = 0}^m n_k$, there exists a universal constant $C >0$ such that:
\[
\sup_{\{0\} \subset S \subset \{0, \dots, m\}}\;\sup_{\bh \in \cH(S)}\; \sup\limits_{P \in \cP'(\bh, \bn, \tau, \kappa, S)}\;\frac{ \Ex \big [ \|\widehat \theta_0 - \theta_0^\star \|_2^2\big] }{\bR(\bh, \bn, \tau)} \le C\,. 
\]
Thus, the elimination-based procedure adapts to the unknown $S$ and $\bh\in \cH(S)$.
\end{theorem}


We conclude this section by summarizing our main contributions. While general adaptation is often fundamentally impossible, we demonstrate that the cost of adaptation can be significantly improved under specific structural assumptions. Specifically, we consider three types of restrictions: ordered biases (Section \ref{subsec:ordered-bias}), clustering of source parameters (Section \ref{sec:clustering-method}), and sufficient separation between the target and the parameters of non-informative sources (Section \ref{sec:test-then-combine}). For each scenario, we propose a tailored estimator. Our proposed methods fully recover adaptation in the first and third cases. In the second case, our estimator recovers adaptation when $m = \mathcal{O}(d)$; otherwise, it achieves the theoretical lower bound (Corollary \ref{cor:adaptation-cost-lb-general-m}) for the cost of adaptation.
\section{Synthetic Data experiments}
\label{sec:simuations}

In this section, we evaluate and compare the \emph{clustering} (Section \ref{sec:clustering-method}) and the \emph{elimination-based} method (Section \ref{sec:test-then-combine}) under various data-generating configurations. Because the adaptation of the intersection-based estimator (Section~\ref{subsec:ordered-bias}) under the ordered-bias assumption has already been established for arbitrary $d$, $\bn$, and $\tau$, we exclude it from our simulations to maintain a streamlined comparison.
Assuming equal sample sizes ($n$) across all domains for a given $m$ and $d$, we generate the local parameter estimates as follows:
\[
\textstyle \widetilde{\theta}_k \sim N \left(\theta_k^\star, \frac{1}{n} \bbI_d\right), \quad k = 0, \dots, m,
\] 
where $\bbI_d$ is the $d \times d$ identity matrix, and $\mu_0 = \mathbf{0}_d$ is the zero vector in $\reals^d$. The remaining true parameters, $\theta_1^\star, \dots, \theta_m^\star$, are specified later according to the desired experimental configuration. 
Given $\{\widetilde{\theta}_k: k = 0, \dots, m\}$, we implement two methods\footnote{Codes for reproducing the results are provided in \url{https://github.com/smaityumich/MTL-intrinsic_cost}} as follows:
\begin{itemize}
    \item \textbf{Elimination-based estimator:} We compute the estimator defined in \eqref{eq:test-then-combine} by fixing the tuning parameters at $\tau = \alpha = 1$.
    \item \textbf{Clustering-based estimator:} We apply Algorithm \ref{alg:clustering-estimator}, which is a modified version of the sample-splitting approach in Algorithm \ref{alg:sample-split-clustering}. This algorithm requires a scaling constant $C > 0$, which we set to $C = 2 \log (mn)$, as well as the number of clusters ($K$), which we will specify later.
\end{itemize}

To benchmark the two estimators above, we compare them against two baseline estimators: a \emph{naive} estimator and an \emph{oracle} estimator. The naive estimator simply outputs the target domain estimate, $\widetilde{\theta}_0$, as the final estimator. Conversely, the oracle estimator leverages the $\bh$ to find the optimal subset $S$ that minimizes the mean squared error (MSE) of the pooled estimator $\widetilde{\theta}_S$:
\begin{equation}
    \label{eq:oracle-estimator}
    \min_{\{0\} \subset S \subset \{0, \dots, m\}}\left\{\mathrm{MSE} \left(\widetilde{\theta}_S\right) := \frac{1}{|S|^2} \bigg\|\sum_{k \in S} \theta_k^\star\bigg\|^2_2 + \frac{d}{n|S| }\right\}, \quad \text{where} \quad \widetilde{\theta}_S := \frac{1}{|S|} \sum_{k \in S}\widetilde{\theta}_k.
\end{equation}
The oracle estimator then outputs $\widetilde{\theta}_{S_0}$, where $S_0$ is the subset that achieves this minimum. Generally, this optimization is computationally prohibitive for large $m$ because it requires searching through $2^m$ possible subsets. However, in our first two experimental configurations, the problem structure simplifies significantly, rendering it computationally feasible.

\subsection{Three configurations}

\begin{figure}[]
    \centering
    \begin{subfigure}[b]{0.69\textwidth}
        \centering
         \includegraphics[width=\textwidth]{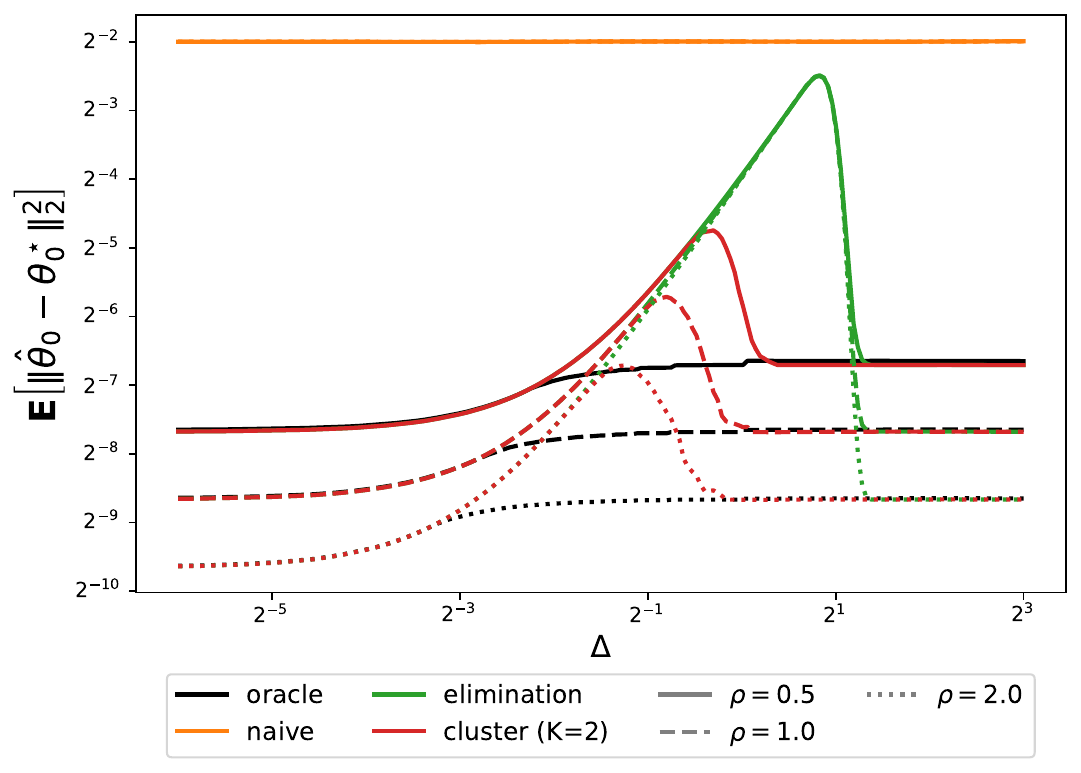}
    \end{subfigure}
    \hfill
    \begin{subfigure}[b]{0.3\textwidth}
        \centering
        \includegraphics[width=\textwidth]{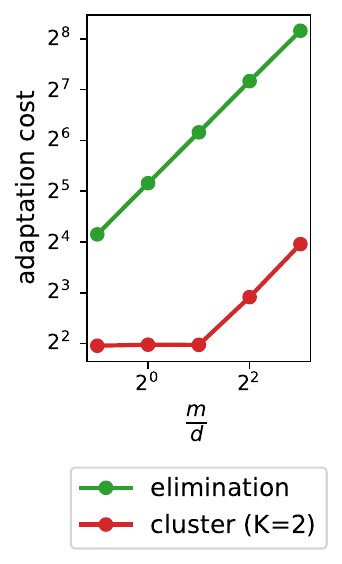}
    \end{subfigure}
    \caption{Comparison of estimators under the cluster configuration. \emph{Left:} MSE across different values of $\Delta$. \emph{Right:} Maximum adaptation cost of the elimination and clustering-based estimators as a function of $\rho = m/d$.}
    \label{fig:cluster-compare}
\end{figure}

\paragraph{Cluster configuration (Figure \ref{fig:cluster-compare}):} 
Given an $m$, we set $\theta_k^\star = \mathbf{0}_d$ for $k = 1, \dots, \lfloor m/2\rfloor$, and $\theta_k^\star = \Delta( \sqrt{d/n}) \mathbf{e}_1$ for $k = \lfloor m/2\rfloor + 1, \dots, m$, where $\mathbf{e}_1 = (1, 0, \dots, 0)^\top \in \reals^d$. Under this setup, the estimates $\widetilde{\theta}_k$ are drawn from a multivariate normal distribution, forming two distinct clusters centered at $\mathbf{0}_d$ and $(\Delta \sqrt{d/n}) \mathbf{e}_1$. 

For small values of $\Delta$, the oracle estimator optimally integrates all $\widetilde{\theta}_k$. Conversely, for large $\Delta$, it excludes the second cluster and only integrates the target domain and the first $\lfloor m/2 \rfloor$ sources. This reduces the candidate pool for the oracle estimator to just two options:
\[
\textstyle \frac{1}{m+1}\sum_{k = 0}^m \widetilde{\theta}_k \qquad \text{or} \qquad \frac{1}{\lfloor m/2\rfloor+1}\sum_{k = 0}^{\lfloor m/2\rfloor} \widetilde{\theta}_k. 
\] 
Consequently, the optimization in \eqref{eq:oracle-estimator} simplifies to choosing the option with the lower MSE:
\[
\widehat{\theta}_0^{(\mathrm{oracle})} = \begin{cases}
    \frac{1}{m+1}\sum_{k = 0}^m \widetilde{\theta}_k & \text{if} \quad \frac{d\Delta^2}{n}\left( \frac{ m - \lfloor m/2 \rfloor }{m+1}\right)^2 + \frac{d}{n(m+1)} \le \frac{d}{n(\lfloor m/2\rfloor + 1)}, \\
    \frac{1}{\lfloor m/2\rfloor+1}\sum_{k = 0}^{\lfloor m/2\rfloor} \widetilde{\theta}_k & \text{otherwise}.
\end{cases}
\] 

We evaluate three scenarios: setting $d = 100$ and $n = 400$, we vary $\rho \in \{ 0.5, 1, 2\}$ and let $m = \rho d$. For each setting, we test a range of $\Delta$ values (shown in the left panel of Figure \ref{fig:cluster-compare}). The MSE for estimating $\theta_0^\star$ is approximated by averaging $\|\widehat{\theta}_0 - \theta_0^\star \|_2^2$ over 500 independent iterations. For the clustering algorithm (Algorithm \ref{alg:clustering-estimator}), we set the number of clusters at $K = 2$. 

As seen in the left panel of Figure \ref{fig:cluster-compare}, the MSE of oracle estimator achieves a rate of $\cO(d/mn) = \cO\{(4\rho d)^{-1}\}$ for larger $\Delta$, meaning the MSE halves when $\rho$ doubles. The naive estimator maintains a relatively constant MSE near $1/4$, perfectly matching its theoretical MSE of $d/n = 1/4$. 
When $\Delta$ is either sufficiently small or sufficiently large, both the elimination-based and clustering-based estimators perform comparably to the oracle estimator. The performance gap at moderate values of $\Delta$ reflects the \emph{cost of adaptation}. We define this cost as the maximum ratio between the MSE of the proposed estimator and that of the oracle estimator. The right panel of Figure \ref{fig:cluster-compare} plots this adaptation cost against $\rho = m/d$. The cost increases linearly with $\rho$, particularly for larger values, corroborating the theoretical lower bound established in Corollary \ref{cor:adaptation-cost-lb-general-m}.


\begin{figure}[]
    \centering
    \includegraphics[width=\textwidth]{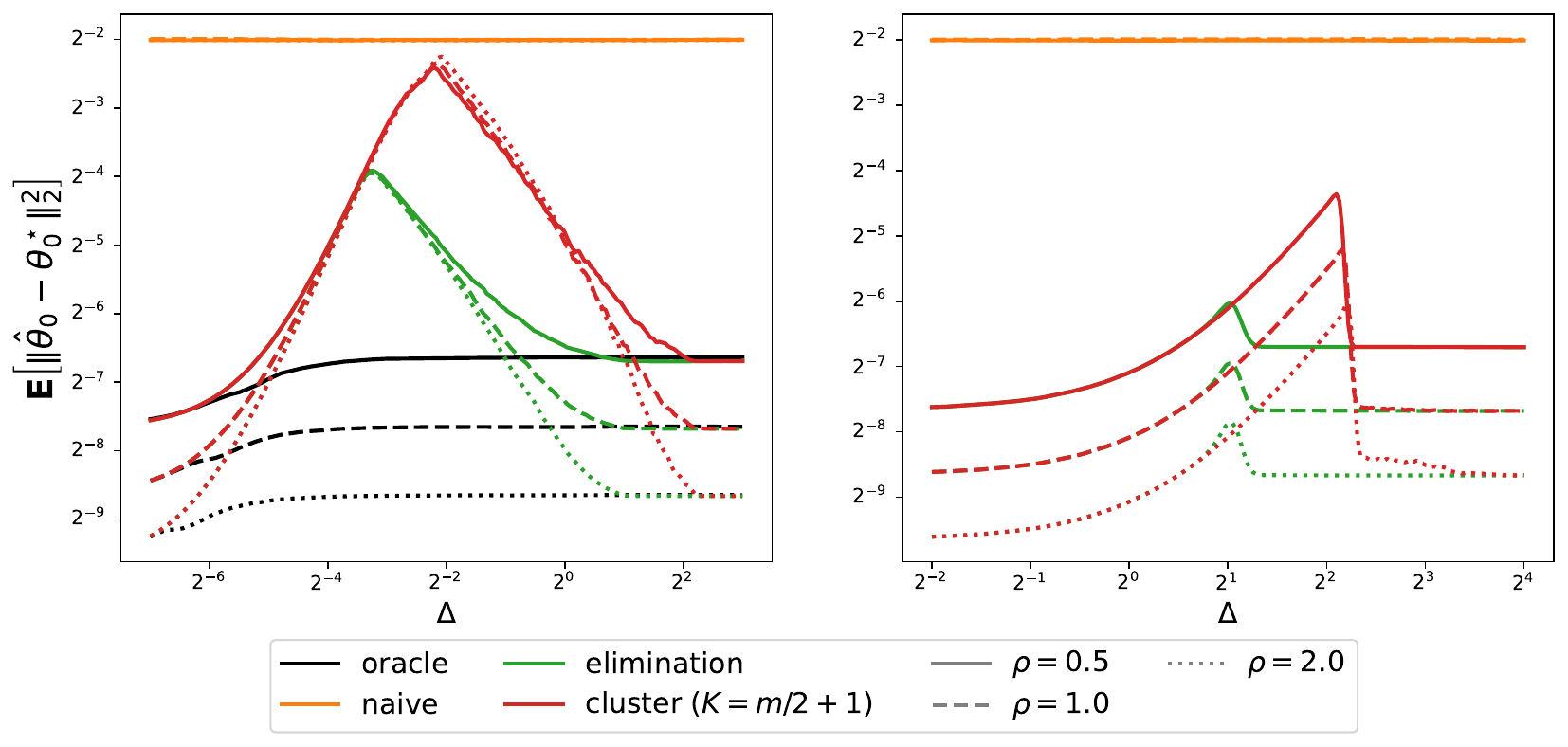}
   
    \caption{Comparison of estimators under Separation I (left) and Separation II (right) configurations.}
    \label{fig:sep-compare}
\end{figure}

\paragraph{Separation I configuration (Figure \ref{fig:sep-compare}, left panel):} 
Next, we explore a separation configuration. Given an $m$, we again set $\theta_k^\star = \mathbf{0}_d$ for $k = 1, \dots, \lfloor m/2\rfloor$. However, the remaining parameters are generated as:
\[
 \textstyle \theta_k^\star := h_k \mathbf{e}_1, \qquad  h_k \sim \text{Uniform}\left(\Delta \sqrt{\frac{d}{n}}, 20 \Delta \sqrt{\frac{d}{n}}\right), \quad k = \left\lfloor \frac{m}{2}\right\rfloor + 1, \dots, m,
\] 
varying $\Delta$ between $2^{-7}$ and $2^2$. Because the non-zero $\theta_k^\star$ values differ from $\theta_0^\star$ only along the first coordinate, we can simplify the oracle optimization \eqref{eq:oracle-estimator}. We sort the source domains by $h_k$ in ascending order such that $h_{(\lfloor m/2\rfloor + 1)} \le \dots \le h_{(m)}$. The oracle estimator is then computed as:
\[
\widehat{\theta}_0^{(\mathrm{oracle})}:= \frac{1}{k_0 +1}\sum_{k = 0}^{k_0}\widetilde{\theta}_{(k)}, \quad   k_0 := \arg\min_{ \left \lfloor \frac{m}{2} \right \rfloor \le k \le m  } \left \{ \left(\frac{\sum_{j = \lfloor m/2 \rfloor + 1}^k h_{(j)}}{k+1} \right)^2 + \frac{d}{(k+1) n} \right\},
\] 
where $\widetilde{\theta}_{(k)} = \widetilde{\theta}_k$ for $k \le \lfloor m/2 \rfloor$, and the remaining $\widetilde{\theta}_{(k)}$ are ordered according to $h_{(k)}$. 

Unlike the previous setup, this data lacks a distinct two-cluster structure, especially for large $\Delta$. But we still compute the clustering-based estimator with $K = \lfloor m/2 \rfloor + 1$ and $C$ as before. As shown in the left panel of Figure \ref{fig:sep-compare}, both elimination- and clustering-based estimators perform similarly to the oracle at the extremes of $\Delta$. However, for moderate $\Delta$, the elimination-based estimator outperforms the clustering-based method, as the underlying data structure is poorly suited for clustering.

\paragraph{Separation II configuration (Figure \ref{fig:sep-compare}, right panel):} 
In our final configuration, we set $\theta_k^\star = \mathbf{0}_d$ for $k = 1, \dots, \lfloor m/2\rfloor$ and generate the remaining parameters as follows:
\[
 \textstyle \theta_k^\star := \Big(\Delta \sqrt{\frac{d}{n}}\Big) u_k, \qquad   u_k \sim \text{Uniform}\left(\bbS^{d-1}\right), \quad k = \left\lfloor \frac{m}{2}\right\rfloor + 1, \dots, m,
\] 
where $\bbS^{d-1} := \{x \in \reals^d: \|x\|_2=1\}$ and $\Delta$ range from $2^{-2}$ to $2^4$. For $k \ge \left\lfloor \frac{m}{2}\right\rfloor + 1$, the $\theta_k^\star$ differ from $\theta_0^\star$ by the exact same magnitude but in uniformly random directions. Because this geometry prevents a straightforward simplification of \eqref{eq:oracle-estimator}, we omit the oracle estimator from this comparison. The naive, clustering-based, and elimination-based estimators are implemented exactly as in Separation I. 

The MSE trends (Figure \ref{fig:sep-compare}, right panel) mirror those of Separation I. The elimination and clustering-based methods perform comparably at extreme values of $\Delta$, while the elimination-based estimator demonstrates superior performance at moderate $\Delta$ values due to the absence of a clear cluster structure.

\paragraph{Comparison of elimination vs. clustering-based estimators:} 
The relative performance of the two proposed estimators is heavily dictated by the underlying data geometry. The clustering-based estimator excels in the two-cluster configuration, whereas the elimination-based estimator dominates in the Separation I and II configurations, where distinct clustering is absent. 

Crucially, however, the penalty for applying the ``wrong" method is asymmetric. In the Separation configurations (Figure \ref{fig:sep-compare}), the MSE of the clustering-based estimator never exceeds the elimination-based estimator's MSE by more than a factor of $8$ ($2^3$). Conversely, in the two-cluster configuration (Figure \ref{fig:cluster-compare}, left panel, $\rho = 2$, $\Delta = 2$), the elimination-based estimator's MSE can be up to $64$ ($2^6$) times worse than that of the clustering-based estimator. This suggests that the potential benefits of the clustering-based estimator when clusters are present significantly outweigh its drawbacks when they are not.

\section{Discussion}
\label{sec:discussion}

This paper studies the statistical cost of adapting to unknown target-to-source biases in multi-source transfer learning. We introduced the intrinsic cost of adaptation, which compares the risk of any bias-agnostic estimator with the oracle minimax rate achievable when the bias configuration is known. This perspective shows that adaptation is not merely an algorithmic difficulty, but can be a fundamental statistical limitation.

Our main finding is that multi-source adaptation becomes qualitatively harder once there is more than one source. With a single source, the learner only needs to decide whether transfer is useful. With multiple sources, the learner must also identify which sources are useful. This source-selection problem creates an unavoidable adaptation cost. For fixed \(m \ge 2\), unrestricted adaptation over the full bias space is possible only when \(n \lesssim dn_0\); when \(n \gg dn_0\), no estimator can uniformly attain the oracle rate. When \(m\) grows, the difficulty increases further because the selection problem becomes combinatorial.

Our results also show how adaptation can be restored. Ordered biases reduce the candidate sets to nested prefixes, clustered sources reduce many domains to a few effective groups, and separation conditions allow non-informative sources to be eliminated through testing. These examples suggest that successful adaptive transfer learning requires a structure that reduces ambiguity about which sources should be used.

Several questions remain open. Most notably, our lower and upper bounds for growing \(m\) do not fully match, leaving an intermediate regime where the exact feasibility of adaptation is unresolved. Another natural direction is to characterize the intrinsic cost of adaptation under unequal source sample sizes. In that setting, the learner must decide not only which sources are informative, but also how the heterogeneous sample sizes \(n_1,\dots,n_m\) interact with the unknown bias configuration. A sharp theory here would clarify whether the adaptation cost is governed by simple effective quantities, or whether it depends in an essential way on the full sample-size profile. It would also be useful to extend the framework beyond the sub-Gaussian local-estimator model, for example, to heavy-tailed settings, dependent sources, other losses, or non-Euclidean parameter spaces.


Overall, our work highlights that the value of auxiliary data depends not only on how close the sources are to the target but also on whether this closeness can be learned reliably. The intrinsic cost of adaptation formalizes this distinction and clarifies when transfer can be achieved without oracle knowledge, when additional structure is necessary, and when adaptation is fundamentally impossible.


\acks{
Abhinav Chakraborty was supported
by the Founder’s Postdoctoral Fellowship in Statistics at Columbia University.
Subha Maity was supported by the University of Waterloo Startup Grant and the NSERC Discovery Grant (RGPIN-2026-04520). 
}



\appendix

\section{Supplementary material for Section \ref{sec:problem-formulation}}


\subsection{M estimation}
\label{supp:m-estimation}

let $ \{Z_{i}\}_{i = 1}^{n} \subset \cZ $ be a random sample from the probability distribution $P$, where $\cZ$ is the sample space for the distribution. Given a convex set $K \subset \reals^d$,   let $\ell:  K \times \cZ \to \reals $ be a loss function such that for every $z \in \cZ$, the $\ell( \cdot, z)$ is twice continuously differentiable and convex. Define the local estimator as an M estimator
    \[
    \widetilde{\theta} := \arg\min_{\theta \in K} \frac{1}{n_k} \sum_{i = 1}^{n} \ell( \theta, Z_i)\,,
    \] and the local parameter value $\theta^\star$ as the minimizer of $L (\theta):= \Ex[\ell( \theta, Z_i)]$. 
    
\begin{assumption} \label{assmp:m-estimation}
    Consider the following:
    \begin{enumerate}
        \item {\bf Score sub-Gaussianity:} Define the scores $\psi(\theta; X) := \nabla_\theta \ell(\theta; X)$. or every fixed unit vector $v \in \mathbb{R}^d$ with
$\|v\|_2 = 1$, the projected score $Z_i(v) := \langle v, \psi(\theta^*;
X_i)\rangle$ is sub-Gaussian with parameter $\sigma_S$, i.e.\ for all $\lambda \in
\mathbb{R}$,
\[
    \mathbb{E}\!\Big[\exp\Big(\lambda \langle v, \psi(\theta^*;
X_i)\rangle\Big)\Big] \;\le\; e^{\sigma_S^2\lambda^2/2}.
\]
\item  {\bf Strong convexity:} There exists a  $\mu > 0$ such that for every unit vector $v \in \mathbb{R}^d$ and
every $\theta\in K$ and $X$ 
\[
    v^\top \nabla_\theta^2 \ell(\theta; X)\, v \;\ge\; \mu^{-1}.
\]
\item {\bf Lipschitz Hessian:}  The map $\theta \mapsto \nabla^2_\theta \ell(\theta; X_i)$ is
$L_H$-Lipschitz in operator norm almost surely:
\[
    \left\|\nabla^2_\theta \ell(\theta; X_i)
          - \nabla^2_\theta \ell(\theta'; X_i)\right\|_{\mathrm{op}}
    \;\le\; L_H \|\theta - \theta'\|_2
    \qquad \text{a.s.},
\]
for all $\theta, \theta' \in \mathcal{B}$. 
    \end{enumerate}
\end{assumption}


\begin{theorem}[Sub-Gaussian Concentration of Projected M-Estimators]
\label{thm:main}
Assume~\ref{assmp:m-estimation}, 
and that $\|\widetilde\theta-\theta^*\|_2 \le \tfrac{1}{2\mu L_H}$. Then, for any $v \in \mathbb{R}^d$ with $\|v\|_2 = 1$ and any
$t > 0$,
\begin{equation}
    \label{eq:main}
    \mathbb{P}\!\left(
        \langle v,\, \widetilde\theta - \theta^*\rangle \ge t
    \right)
    \;\le\;
    2\exp\!\left(-\frac{n\,t^2}{2\,\tau^2}\right),
\end{equation}
where $\tau^2 := 4\mu^2\sigma_S^2$.
\end{theorem}

\begin{proof}
The proof proceeds in three steps.

\paragraph{Step 1:}
Since $K$ is compact and convex, and $\widetilde\theta$ minimizes
$\hat{L}_n(\theta) := \frac{1}{n}\sum_{i=1}^n \ell(\theta,Z_i)$ over
$K$, the first-order optimality  condition
gives $\nabla\hat{L}_n(\widetilde\theta) = 0$.
Because $\ell(\,\cdot\,,z)$ is twice continuously differentiable, a
mean-value expansion yields
\begin{equation}
    \label{eq:taylor}
    0 = \nabla\hat{L}_n(\widetilde\theta)
    \;=\;
    \nabla\hat{L}_n(\theta^*)
    \;+\;
    \widehat{H}_n(\bar\theta)\,(\widetilde\theta - \theta^*),
\end{equation}
where $\bar\theta$ lies on the segment $[\theta^*,\widetilde\theta]
\subset K$ and $\widehat{H}_n(\theta)
    \;:=\;
    \frac{1}{n}\sum_{i=1}^n \nabla^2_\theta\ell(\theta,Z_i)$.
Then
\begin{equation}
    \label{eq:quad}
    (\widetilde\theta-\theta^*)^\top
    \widehat{H}_n(\bar\theta)\,
    (\widetilde\theta-\theta^*)
    \;=\;
    \bigl\langle
        \nabla\hat{L}_n(\theta^*),\;
        \theta^* - \widetilde \theta
    \bigr\rangle.
\end{equation}

\paragraph{Step 2: }
Decompose the empirical Hessian at $\bar\theta$ as
\begin{equation}
    \label{eq:hess_decomp}
    \widehat{H}_n(\bar\theta)
    \;=\;
    \widehat{H}_n(\theta^*)
    \;+\;
    \underbrace{
        \bigl[\widehat{H}_n(\bar\theta) - \widehat{H}_n(\theta^*)\bigr]
    }_{=:\,A}.
\end{equation}


\noindent\textit{Control of $A$ via a self-bounding argument.}
Unlike $B$, the term $A = \widehat{H}_n(\bar\theta) -
\widehat{H}_n(\theta^*)$ depends on the random point $\bar\theta$,
which lies on the segment $[\theta^*,\widetilde\theta]$.  By the
Lipschitz Hessian condition (Assumption~\ref{assmp:m-estimation}(4)),
\begin{equation}
    \label{eq:A_pointwise}
    \|A\|_{\mathrm{op}}
    \;\le\;
    L_H\,\|\bar\theta - \theta^*\|_2
    \;\le\;
    L_H\,\|\widetilde\theta - \theta^*\|_2,
\end{equation}
since $\bar\theta$ lies between $\theta^*$ and $\widetilde\theta$.
We do \emph{not} bound $\|\widetilde\theta - \theta^*\|_2$ by the
deterministic radius $R$ at this stage; instead we will obtain a
data-dependent bound on $\|\widetilde\theta - \theta^*\|_2$ and feed
it back.

\noindent\textit{Hessian lower bound on $\mathcal{G}$.}
On the event $\mathcal{G}$, using \eqref{eq:hess_decomp},
\eqref{eq:A_pointwise}, the definition of $\mathcal{G}$, and the
strong-convexity lower bound
$\nabla^2 L(\theta^*) \succeq \mu^{-1}I$
(Assumption~\ref{assmp:m-estimation}(2)):
\begin{equation}
    \label{eq:hess_lb_rough}
   \textstyle  \widehat{H}_n(\bar\theta)
    \;\succeq\;
    \left(
        \frac{1}{\mu}
        - L_H\|\widetilde\theta - \theta^*\|_2
    \right)I
    \; \succeq\; \frac{1}{2\mu} I .
\end{equation}
where the last inequality holds because
$\|\widetilde\theta-\theta^*\|_2 < \tfrac{1}{2\mu L_H}$.

\paragraph{Step 3: }


\noindent\textit{Refined norm bound on $\mathcal{G}$.}
Combining \eqref{eq:quad} and \eqref{eq:hess_lb_rough} with
Cauchy--Schwarz:
\begin{align}
 \textstyle    \frac{1}{2\mu}
    \|\widetilde\theta-\theta^*\|_2^2
    \;\le\;
    (\widetilde\theta-\theta^*)^\top
    \widehat{H}_n(\bar\theta)
    (\widetilde\theta-\theta^*) \notag 
    \;\le\;
    \|\nabla\hat{L}_n(\theta^*)\|_2\,
    \|\widetilde\theta-\theta^*\|_2, \label{eq:quad_cs}
\end{align}
Dividing both sides 
by $\|\widetilde\theta-\theta^*\|_2 > 0$ gives
\begin{equation}
    \label{eq:norm_bound}
  \textstyle   \|\widetilde\theta - \theta^*\|_2
     \; \le \; 2\mu \|\nabla\hat{L}_n(\theta^*)\|_2 
    \qquad\text{on } \mathcal{G}.
\end{equation}


\noindent\textit{Projected score bound.}
From the inversion of \eqref{eq:taylor} and the optimality condition $\widetilde\theta - \theta^*
    =
    -\widehat{H}_n(\bar\theta)^{-1}\nabla\hat{L}_n(\theta^*)$
Projecting onto $v$
and applying $\|\widehat{H}_n(\bar\theta)^{-1}\|_{\mathrm{op}}
\le 2\mu$:
\begin{equation}
    \label{eq:proj_final}
    \langle v,\widetilde\theta-\theta^*\rangle
    \;\le\;
    2\mu\,|S_n(v)| 
    \qquad\text{on }\mathcal{G}.
\end{equation}

\noindent\textit{Sub-Gaussianity of $S_n(v)$.}
Since $\mathbb{E}[\psi(\theta^*;Z_i)] = \nabla L(\theta^*) = 0$ and
each $\langle v,\psi(\theta^*;Z_i)\rangle$ is sub-Gaussian with
parameter $\sigma_S$ (Assumption~\ref{assmp:m-estimation}(1)),
$S_n(v)$ is sub-Gaussian with parameter $\sigma_S/\sqrt{n}$:
\begin{equation}
    \label{eq:score_tail}
  \textstyle   \mathbb{P}\bigl(|S_n(v)| \ge s\bigr)
    \;\le\;
    2\exp\!\left(-\frac{ns^2}{2\sigma_S^2}\right)
    \qquad\forall\,s>0.
\end{equation}
Given a $t > 0$, we let $s = t/(2\mu)$ and obtain
\begin{equation}
    \label{eq:main_tail}
  \textstyle   \mathbb{P}\!\left(
        \langle v,\widetilde\theta-\theta^*\rangle \ge t
    \right)
    \;\le\;
    \mathbb{P}\!\left(|S_n(v)| \ge \frac{t}{2\mu}\right)
    \;\le\;
    2\exp\!\left(-\frac{nt^2}{2\tau^2}\right),
\end{equation}
where $\tau^2 = 4\mu^2\sigma_S^2$. This completes the proof.

\end{proof}

\subsection{Proof of Theorem \ref{thm:oracle-minimax-rate}}
\label{proof:thm:oracle-minimax-rate}

\subsubsection{Proof of the lower bound} 
The main technical step is the multivariate lower bound stated in Lemma~\ref{lem:multivar-lower}, whose proof appears later.  
We then refine that result to obtain the set-wise characterization of the minimax rate.

\bigskip

\begin{proof} The proof is broken into multiple steps.
\paragraph{Step 1: Reduction via Lemma~\ref{lem:multivar-lower}.}
From Lemma~\ref{lem:multivar-lower}, for any $\eta>0$ satisfying
\[
\eta^2\sum_{k:\,\eta\ge\tfrac{h_k}{2\sqrt d}}\frac{n_k}{\tau^2}\le1,
\]
the minimax risk satisfies
\[
\inf_{\widehat\theta_0}\sup_{P\in\cP(h_1,\dots,h_m,\tau)}
\Ex_P\big\|\widehat\theta_0-\theta_0^\star\big\|_2^2
\ \gtrsim\ d\,\eta^2.
\]
Let
\[
\eta^\star
=\sup\left\{
\eta>0:\ 
\eta^2\!\!\sum_{k:\,\eta\ge\tfrac{h_k}{2\sqrt d}}\tfrac{n_k}{\tau^2}\le1
\right\}.
\]
Optimizing over $\eta$ gives
\[
\inf_{\widehat\theta_0}\sup_P
\Ex_P\|\widehat\theta_0-\theta_0^\star\|_2^2
\ \gtrsim\
d(\eta^\star)^2.
\]

\paragraph{Step 2: Reformulation in setwise form.}
To express this bound in terms of the problem parameters, fix $\delta>0$ and define the active index set
\[
S_\delta=\bigl\{k:\ \eta^\star+\delta\ge\tfrac{h_k}{2\sqrt d}\bigr\}.
\]
By maximality of $\eta^\star$, any $\eta>\eta^\star$ violates the constraint, implying
\[
(\eta^\star+\delta)^2
\ \ge\
\Bigl(\sum_{k\in S_\delta}\tfrac{n_k}{\tau^2}\Bigr)^{-1},
\qquad
\max_{k\in S_\delta}h_k^2
\ \le\
4d(\eta^\star+\delta)^2.
\]
Evaluate at $S_\delta$ and then minimize over $S$:
\[
\min_{\{0\}\subset S\subset\{0,\dots,m\}}
\left[
d\left(\sum_{k\in S}\frac{n_k}{\tau^2}\right)^{-1}
+ \max_{k\in S}h_k^2
\right]
\ \le\ d
\left(\sum_{k\in S_\delta}\frac{n_k}{\tau^2}\right)^{-1}
+ \max_{k\in S_\delta}h_k^2
\ \le\ (d+4d)\,(\eta^\star+\delta)^2.
\]
Letting $\delta\downarrow 0$ gives
\[
\min_{\{0\}\subset S\subset\{0,\dots,m\}}
\left[
d\left(\sum_{k\in S}\frac{n_k}{\tau^2}\right)^{-1}
+ \max_{k\in S}h_k^2
\right]
\ \lesssim\ d(\eta^\star)^2.
\]
Combining with the lower bound from Step 1 gives
\[
\inf_{\widehat\theta_0}\sup_{P\in\cP(h_1,\dots,h_m,\tau)}
\Ex_P\big\|\widehat\theta_0-\theta_0^\star\big\|_2^2
\ \gtrsim\
\min_{\{0\}\subset S\subset\{0,\dots,m\}}
\Biggl[
d\Bigl(\sum_{k\in S}\tfrac{n_k}{\tau^2}\Bigr)^{-1}
+\max_{k\in S}h_k^2
\Biggr],
\]
which completes the proof lower bound.

\end{proof}

\begin{lemma}[Multivariate lower bound via Assouad]
\label{lem:multivar-lower}
Fix integers $d\ge 1$ and $m\ge 1$. 
Suppose $(\widetilde \theta_0,\dots,\widetilde \theta_m)$ are independent random vectors satisfying
\[
\widetilde \theta_k \;\sim\; \mathcal{N}\!\big(\theta_k^\star,\, \sigma_k^2 I_d\big),
\qquad 
\sigma_k^2 = \frac{\tau^2}{n_k}.
\]
Assume $\| \theta_k^\star - \theta_0^\star \|_2 \le h_k$ for $k=1,\dots,m$. 
Then the minimax mean-squared error for estimating $\theta_0^\star$ over $\cP(h_1,\dots,h_m,\tau)$ satisfies
\[
\inf_{f}\;\sup_{P\in\cP(h_1,\dots,h_m,\tau)}\; 
\Ex_P\!\big\|f(\widetilde \theta_0,\dots,\widetilde \theta_m)-\theta_0^\star\big\|_2^2
\;\gtrsim\;
d\,\eta^2,
\]
where $\eta>0$ is chosen to be as large as possible, subject to
\[
\eta \;\le\;
\Biggl(
  \sum_{k:\,\eta \ge \tfrac{h_k}{2\sqrt d}}
  \frac{n_k}{\tau^2}
\Biggr)^{-1/2}.
\]
\end{lemma}

\begin{proof}
We proceed in a step-by-step fashion.

\paragraph{1. One-dimensional construction.}
Define two scalar parameter configurations:
\[
\alpha_1:\quad \theta_0^{(1)}=\eta,\quad \theta_k^{(1)}=\eta+\tfrac{h_k}{2},
\]
\[
\alpha_2:\quad \theta_0^{(2)}=0,\quad 
\theta_k^{(2)}=
\begin{cases}
\tfrac{h_k}{2}, & \text{if } \tfrac{h_k}{2}\le \eta,\\[3pt]
\eta+\tfrac{h_k}{2}, & \text{if } \tfrac{h_k}{2}>\eta.
\end{cases}
\]
Then $\|\Ex[\widetilde \theta_k]-\theta_0^\star\|_2\le h_k$ holds and
\[
\mathrm{KL}(\alpha_1,\alpha_2)
\;=\;
\eta^2 \sum_{k:\,\tfrac{h_k}{2}\le \eta}\sigma_k^{-2}
\;=\;
\eta^2\sum_{k:\,\tfrac{h_k}{2}\le \eta}\frac{n_k}{\tau^2}.
\]

\paragraph{2. Tensorized hypercube construction.}
Let $u\in\{1,2\}^d$ index hypotheses. 
For each coordinate $j$, place an independent copy of the 1-D pair above, but scale the bias by $1/\sqrt d$ so that the $\ell_2$ class constraints hold jointly.  
Define
\[
T_d \;=\; \bigl\{k:\, \tfrac{h_k}{2\sqrt d}\le \eta\bigr\},
\]
and set
\[
\theta_{0,j}(u) =
\begin{cases}
\eta, & u_j=1,\\
0, & u_j=2,
\end{cases}
\qquad
\theta_{k,j}(u) =
\begin{cases}
\eta+\tfrac{h_k}{2\sqrt d}, & u_j=1,\\[4pt]
\tfrac{h_k}{2\sqrt d}, & u_j=2,~k\in T_d,\\[4pt]
\eta+\tfrac{h_k}{2\sqrt d}, & u_j=2,~k\notin T_d.
\end{cases}
\]
Then $(\widetilde \theta_0,\dots,\widetilde \theta_m)$ satisfies the constraints of $\cP(h_1,\dots,h_m,\tau)$.

\paragraph{3. Class membership check.}
For each $k$ and $u$, we have
\[
\|\theta_k(u)-\theta_0(u)\|_2
\;\le\;
\sqrt d\big(\eta+\tfrac{h_k}{2\sqrt d}\big)
\;\le\;
h_k,
\]
so the bias constraint $\|\theta_k^\star-\theta_0^\star\|_2\le h_k$ holds.

\paragraph{4. Edge separation and divergence.}
Let $u^{(j)}$ flip only the $j$-th coordinate of $u$.  
Then
\[
\|\theta_0(u)-\theta_0(u^{(j)})\|_2^2=\eta^2,
\qquad
\mathrm{KL}(P_u,P_{u^{(j)}})
=\eta^2\!\!\sum_{k\in T_d\cup\{0\}}\!\!\frac{n_k}{\tau^2}.
\]

\paragraph{5. Application of Assouad’s lemma.}
By Assouad’s lemma for the $2^d$-hypercube, we obtain
\[
\inf_{f}\sup_{u}\Ex_u\|f-\theta_0(u)\|_2^2
\;\gtrsim\;
d\,\eta^2\!\left(1-\sqrt{C\,\eta^2\!\sum_{k\in T_d\cup\{0\}}\!\frac{n_k}{\tau^2}}\right).
\]
Choosing $\eta$ so that the KL term in parentheses remains bounded by a universal constant gives the claimed rate, i.e.
\[
\eta \;\le\; 
\Biggl(
  \sum_{k:\,\eta \ge \tfrac{h_k}{2\sqrt d}} \frac{n_k}{\tau^2}
\Biggr)^{-\!1/2}.
\]
\end{proof}

\subsubsection{Proof of the upper bound}

For an integration set $\{0\} \subset S \subset \{0, \dots , m\}$, define the estimator
\[
\begin{aligned}
    \widehat \theta_{0, S} = \sum_{k \in S}  \omega_k \widetilde \theta_k , ~~ \omega_k =  \frac{ n_k}{\sum_{j \in S}  n_j}\,. 
\end{aligned}
\]
Then an upper bound for mean squared error $\Ex [\|\widehat \theta _{0, S} - \theta_0^\star \|_2^2]$ is 
\[
\begin{aligned}
    \Ex \left[\left \|\widehat \theta_{0, S} - \theta_0^\star\right\|_2^2\right] & = \tr \left [\var\left (\widehat \theta_{0, S}\right)\right] +\left \| \Ex \left [\widehat \theta_{0, S}\right] - \theta_0^\star\right\|_2^2 \\
    &\le   \sum _{k \in S} \omega_k^2 \left (\frac{d\tau^2}{n_k}\right) + \left \|\sum_{k \in S} \omega_k \theta_k^\star - \theta_0^\star\right\|^2 \le   \frac{d\tau^2}{\sum_{k \in S} n_k} + \max_{k \in S} h_k^2   
\end{aligned}
\] 
Let $S^\star$ be a minimizer of the oracle rate in equation \eqref{eq:loca-loptimal-rate}. By the definition of $S^\star$, we have 
\[
\begin{aligned}
    \Ex \left[\left \|\widehat \theta_{0, \text{oracle}} - \theta_0^\star\right\|_2^2\right] & \le  \frac{d\tau^2}{\sum_{k \in S^\star} n_k} + \max_{k \in S^\star} h_k^2  = \min_{\{0\}\subset S \subset \{0,\dots,m\}}
\left[
 \frac{d\tau^2}{\sum_{k \in S} n_k}
+ \max_{k\in S} h_k^2
\right]\,. 
\end{aligned}
\]
This completes the proof. 

\section{Supplementary material for Section \ref{sec:bias-adaptation}}

\subsection{Proof of the Theorem \ref{thm:cost-fixed-m}}
\label{supp:proof:thm:cost-fixed-m}

\subsubsection{Proof of the lower bound}
We establish the bound for $m = 2$. Later, we shall discuss how to extend the proof to a finite $m \ge 3$. 

\paragraph{Proof for $m = 2$:} 
For any $0<g_1<g_2<1$, consider the following configuration:
\[
\widetilde \theta_k \sim N\Big(\theta_k, (\tau^2 / n_k) \bI_d \Big), ~~ k = 0, 1, \text{ and } 2,
\] where $n_1 = n_2 = n$. For $e_1 = (1, 0, \dots, 0) ^\top \in \reals^d$, we let $\theta_1 = g_1 e_1$, $\theta_2 = g_2 e_1$, and $\theta_0 \in \{0_d , (g_1 +g_2)e_1\}$. We denote the two joint distributions of $(\widetilde \theta_0, \widetilde \theta_1, \widetilde \theta_2)$ as $P_0$ and $P_1$. First, we shall establish a lower bound for
\[
\inf_{\widehat\theta_0}
\sup_{\{P_0,P_1\}}
\Ex\!\left[\|\widehat\theta_0-\theta_0\|_2^2\right]
\]
The Kulback-Leibler divergence between the two distributions is $  \mathrm{KL} (P_0, P_1)  = \frac{n_0(g_1 + g_2)^2}{2\tau^2}$.
Then, for any $g_2 \le \tau / (2\sqrt{n_0})$, from LeCam's two-point bound, we have 
\[
\begin{aligned}
   \textstyle  \inf_{\widehat\theta_0}
\sup_{\{P_0,P_1\}}
\Ex\!\left[(\widehat\theta_0-\theta_0)^2\right] & \ge c(g_1 + g_2)^2  \Big(1 - \TV(P_0, P_1)\Big)  \\
& \textstyle  \ge c(g_1 + g_2)^2  \left(1 - \sqrt{\frac{\KL(P_0, P_1)}{2}}\right)  \\
& \textstyle  \ge c(g_1 + g_2)^2  \left(1 - \sqrt{\frac{n_0(g_1 + g_2)^2}{4\tau^2}}\right)  \ge \frac {c g_2^2} 2   \\
\end{aligned}
\] where the final inequality holds because $g_1 + g_2 \le 2g_2 \le \tau /\sqrt{n_0}$. 
Now, for $\bn = (n_0, n, n)$, note that $P_0 \in \cP((g_1, g_2), \bn, \tau)$ and $P_1 \in \cP((g_2, g_1), \bn, \tau)$. Then, we have the upper bound 
\[
\textstyle \fR((g_1, g_2), \bn, \tau) \vee \fR((g_2, g_1), \bn, \tau) \asymp \frac{d\tau^2}{n_0} \wedge \left( \frac{d\tau^2}{n} \vee g_1^2\right)\,.
\]
\begin{equation} \label{eq:lb-m2-eq-1-supp}
    \begin{aligned}
\text{Consequently,}\qquad &\inf_{\widehat\theta_0}
\sup_{(h_1,h_2)\in\{(g_1,g_2),(g_2,g_1)\}}
\frac{
\sup_{P\in\cP((h_1,h_2), \bn , \tau)}
\Ex[\|\widehat\theta_0-\theta_0\|_2^2]
}{
R(h_1,h_2;\tau,N)
}\\[4pt]
&\qquad\ge\;
\inf_{\widehat\theta_0}
\sup_{(h_1,h_2)\in\{(g_1,g_2),(g_2,g_1)\}}
\textstyle \frac{
\sup_{P\in\cP((h_1,h_2), \bn , \tau)}
\Ex[\|\widehat\theta_0-\theta_0\|_2^2]
}{
\frac{d\tau^2}{n_0} \wedge \left( \frac{d\tau^2}{n} \vee g_1^2\right)
}
\\[6pt]
&\qquad\ge\; \textstyle 
c_L\,\frac{g_2^2 }{
\frac{d\tau^2}{n_0} \wedge \left( \frac{d\tau^2}{n} \vee g_1^2\right)
}.
\end{aligned}
\end{equation}

\medskip\noindent
\emph{Taking the supremum over all $(h_1,h_2)\in[0,1]^2$.}
By the definition of the intrinsic adaptation cost,
\[
\begin{aligned}
    & \inf_{\widehat\theta_0}
\sup_{(h_1,h_2)\in[0,1]^2}
\frac{
\sup_{P\in\cP((h_1,h_2), \bn , \tau)}
\Ex[\|\widehat\theta_0-\theta_0\|_2^2]
}{
R(h_1,h_2;\tau,N)
}\\
&
=
\inf_{\widehat\theta_0}
\sup_{0<g_1<g_2<1}
\sup_{(h_1,h_2)\in\{(g_1,g_2),(g_2,g_1)\}}
\frac{
\sup_{P\in\cP((h_1,h_2), \bn , \tau)}
\Ex[\|\widehat\theta_0-\theta_0\|_2^2]
}{
R(h_1,h_2;\tau,N)
}.
\end{aligned}
\]
Lower bounding the outer supremum by restricting to the small-signal
regime $g_2\le\tau/\sqrt{n_0}$, in which
Lemma~\ref{lem:two-point} applies, we obtain
\[
\begin{aligned}
\mathfrak C^\star([0, 1]^2)
&\ge
\sup_{0<g_1<g_2<\tau/\sqrt{n_0}}
\textstyle c_L\,\frac{g_2^2 }{
\frac{d\tau^2}{n_0} \wedge \left( \frac{d\tau^2}{n} \vee g_1^2\right)
}\\[4pt]
&=
\textstyle c_L'\,
\left(\frac{n}{dn_0} \vee 1 \right)\,.
\end{aligned}
\]
This completes the lower bound.

\begin{lemma}[Two--point lower bound for a fixed configuration]\label{lem:two-point}
Fix numbers $0 < g_1 \le g_2$ and consider the scalar Gaussian model
\[
\tilde\theta_k \sim N(\theta_k,\tau^2/n_k),\qquad k=0,1,2,
\]
independent.  Fix the configuration with $\theta_1 = g_1 $ and $ \theta_2 = g_2$, 
and consider 
\[
\theta_0 \in \{0,\, g_1+g_2\}.
\]
Assume, moreover, that $g_2 \lesssim \tau/\sqrt{n_0}$.
Then there exists a universal constant $c>0$ such that, for any estimator
$\widehat\theta_0=\widehat\theta_0(\tilde\theta_0,\tilde\theta_1,\tilde\theta_2)$,
\[
\inf_{\widehat\theta_0}
\;\sup_{\theta_0\in\{0,g_1+g_2\}}
\Ex\big[(\widehat\theta_0-\theta_0)^2\big]
\;\ge\;
c\,g_2^2.
\]
\end{lemma}

\begin{proof}
\emph{Step 1: Reduction to a scalar problem.}
Since $(\tilde\theta_1, \tilde\theta_2)$ has distribution
$N(g_1, \tau^2/n_1)\otimes N(g_2, \tau^2/n_2)$ regardless of $\theta_0$,
the joint law factors as $P_i = P_i^{(0)} \otimes Q$ for $i = 0,1$,
where $Q$ is the common law of $(\tilde\theta_1,\tilde\theta_2)$.
Hence the likelihood ratio depends only on $\tilde\theta_0$, and
\[
\inf_{\widehat\theta_0(\tilde\theta_0,\tilde\theta_1,\tilde\theta_2)}
\sup_{\theta_0\in\{0,g_1+g_2\}}
\Ex\bigl[(\widehat\theta_0-\theta_0)^2\bigr]
=
\inf_{\widehat\theta_0(\tilde\theta_0)}
\sup_{\theta_0\in\{0,g_1+g_2\}}
\Ex\bigl[(\widehat\theta_0-\theta_0)^2\bigr].
\]

\medskip
\emph{Step 2: Le Cam's two-point bound.}
Set $\Delta := g_1 + g_2$. The marginal laws of $\tilde\theta_0$ are
$P_0^{(0)} = N(0, \tau^2/n_0)$ and $P_1^{(0)} = N(\Delta, \tau^2/n_0)$.
By Le Cam's lemma,
\[
\inf_{\widehat\theta_0(\tilde\theta_0)}
\sup_{\theta_0\in\{0,\Delta\}}
\Ex\bigl[(\widehat\theta_0-\theta_0)^2\bigr]
\ge
\frac{\Delta^2}{4}\Bigl(1 - \TV(P_0^{(0)}, P_1^{(0)})\Bigr).
\]
Using the standard formula
$\TV(N(\mu_1,\sigma^2), N(\mu_2,\sigma^2))
= 2\Phi\!\left(-\frac{|\mu_1-\mu_2|}{2\sigma}\right)$,
we obtain
\[
1 - \TV(P_0^{(0)}, P_1^{(0)})
= 1 - 2\,\Phi\!\left(-\frac{\Delta\sqrt{n_0}}{2\tau}\right)
\asymp 1 - \Phi\!\left(\frac{\Delta\sqrt{n_0}}{\tau}\right),
\]
so the minimax risk is
$\displaystyle\gtrsim \Delta^2\!\left(1-\Phi\!\left(
\frac{\Delta\sqrt{n_0}}{\tau}\right)\right)$.

\medskip
\emph{Step 3: Applying the small-signal condition.}
Since $g_1 \le g_2 \lesssim \tau/\sqrt{n_0}$, we have
$\Delta = g_1+g_2 \le 2g_2 \lesssim \tau/\sqrt{n_0}$,
so $\Delta\sqrt{n_0}/\tau$ is bounded by a universal constant.
Therefore $1 - \Phi(\Delta\sqrt{n_0}/\tau) \ge c_0 > 0$ for a
universal constant $c_0$. Moreover, $g_2 \le \Delta \le 2g_2$
gives $\Delta^2 \asymp g_2^2$. Combining,
\[
\inf_{\widehat\theta_0}
\sup_{\theta_0\in\{0,g_1+g_2\}}
\Ex\bigl[(\widehat\theta_0-\theta_0)^2\bigr]
\gtrsim
g_2^2,
\]
which completes the proof.
\end{proof}

\paragraph{Proof of $m > 2$:} Repeat the proof for $m = 2$ for the following class: 
\[
\widetilde \theta_k \sim N\Big(\theta_k, (\tau^2 / n_k) \bI_d \Big), \quad \text{where} \quad \theta_k = \begin{cases}
    g_1 e_1 & k = 1 , \dots , \lfloor m/2 \rfloor \\
    g_2 e_1 & k = \lfloor m/2 \rfloor + 1 , \dots, m
\end{cases}
\] Then the first $\lfloor m/2 \rfloor$ source domains can be combined into one with an effective sample size of $\lfloor m/2 \rfloor n \asymp n$. Similarly, the final $m - \lfloor m/2 \rfloor$ source domains can be combined into one with an effective sample size of $(m - \lfloor m/2 \rfloor) n \asymp n$. Then the rate for the bound remains unchanged. 

Alternatively, the bound also holds by letting $m$ be fixed in Theorem \ref{thm:cost-lower-bound-general-m}.

\subsubsection{Proof of the upper bound} 
The upper bound is obtained by a direct application of Corollary \ref{cor:adaptation-cost-ub-general-m}, where, by keeping $m$ finite, we obtain 
\[
\textstyle \fC^\star ([0, 1]^m) \lesssim \left(\frac{n}{dn_0} \vee 1 \right)\,.
\] Thus, the upper bound is established for a constant $C >0$.


\subsection{Proof of Lemma \ref{lemma:upper-bound-cost-m2-restricted}}

\label{proof:lemma:upper-bound-cost-m2-restricted}

Recall the Lemma \ref{lemma:expectation-ub-1}. Define $\nu_1 := \|\theta_1 - \theta_0\|_2 $,  $\nu_2:= \|\theta_2 - \theta_0\|_2 $, and 
    \[
   \textstyle  K^2 := 2C \tau^2 \left(\frac{d + \log(1/\delta)}{n} \vee \frac{\log(1/\delta)}{n_0}\right)\,,
 \] where $C > 0$ is the constant within $\cE_T$, as defined in Lemma \ref{lem:diff-sq-concentration}. Let $\delta = n^{-2}$. According to Lemma \ref{lemma:expectation-ub-1},  there exists a constant $C_1 > 0$ such that
 \[
 \begin{aligned}
     & \textstyle \bR(\widehat \theta_0, \nu_1, \nu_2) := \textstyle \Ex \left[ \|\widehat \theta_0 - \theta_0\|_2 ^2\right]\\
     & \le \begin{cases}
     C_1 \frac{d\tau^2}{n}  & \nu_1, \nu_2 \le \tau\sqrt{\frac{d}{n}}\\
     C_1 \left\{ \frac{d\tau^2}{n_0} \wedge \left(\frac{d\tau^2}{n} + \nu_2^2 +K ^2\right)\right\} \log n & \nu_1 > \tau\sqrt{\frac{d}{n}}, ~\nu_2 \le \nu_1\\
     C_1 \left\{ \frac{d\tau^2}{n_0} \wedge \left(\frac{d\tau^2}{n} + \nu_1^2 \right)\right\} \log n & \nu_2 > \tau\sqrt{\frac{d}{n}}, ~\nu_1 \le \nu_2, ~ n_0 \le \frac{ n\log(1/\delta)}{16 C^2 \{d + \log(1/\delta)\}}\,.
 \end{cases}
 \end{aligned}
 \]
\paragraph{Case 1:} For $h_1, h_2 \le \tau \sqrt{d/n}$, we have $\fR ((h_1, h_2), \bn, \tau) = d\tau^2 / n$. Then, 
\[
\begin{aligned}
     \frac{\max\limits_{\nu_1 \le h_1, \nu_2 \le h_2}\bR(\widehat \theta_0, \nu_1, \nu_2)}{\fR ((h_1, h_2), \bn, \tau)} \le \frac{\max\limits_{\nu_1 , \nu_2 \le \tau \sqrt{\frac{d}{n}}}\bR(\widehat \theta_0)}{\fR ((h_1, h_2), \bn, \tau)} \le C_1 \frac{\frac{d\tau^2}{n}}{\frac{d\tau^2}{n}} \le C_1\,.
\end{aligned}
\]

\paragraph{Case 2:} Let  $h_2 > \tau \sqrt{d/n}$ and $h_1 \le h_2$. Then $\fR ((h_1, h_2), \bn, \tau) = \frac{d\tau^2}{n_0} \wedge \left(\frac{d\tau^2}{n} + h_1^2 \right)$. 
Since $n_0 \le \frac{ n\log(1/\delta)}{16 C^2 \{d + \log(1/\delta)\}}$, then $K^2 = C\tau^2 \frac{\log(1/\delta)}{n_0} = 2C\tau^2 \frac{\log n}{n_0}$ and 
\[
\begin{aligned}
    & \max_{\nu_1 \le h_1, \nu_2 \le h_2}\bR(\widehat \theta_0, \nu_1, \nu_2)\\
    & \le   \left\{\max_{\nu_1 , \nu_2 \le  \tau\sqrt{\frac{d}{n}}}\bR(\widehat \theta_0, \nu_1, \nu_2)\right\} \vee \left\{\max_{\nu_1 \le \nu_2, \nu_1 \le h_1,   \nu_2 \le h_2, \nu_2  >  \tau\sqrt{\frac{d}{n}}}\bR(\widehat \theta_0, \nu_1, \nu_2)\right\}\\
    & \qquad \vee \left\{\max_{\nu_1 > \nu_2, \nu_1 \le h_1,   \nu_2 \le h_2, \nu_1  >  \tau\sqrt{\frac{d}{n}}}\bR(\widehat \theta_0, \nu_1, \nu_2)\right\}\\
\end{aligned}
\] 
\[
\begin{aligned}
  & \le \textstyle C_1 \cdot \left\{\frac{d\tau^2}{n}\right\} \vee \left\{\max_{\nu_1  \le h_1}\left\{ \frac{d\tau^2}{n_0} \wedge \left(\frac{d\tau^2}{n} + \nu_1^2 \right)\right\} \log n\right\}\\
    & \qquad \vee \textstyle \left\{\max_{   \nu_2 < \nu_1 \le h_1}\left\{ \frac{d\tau^2}{n_0} \wedge \left(\frac{d\tau^2}{n} + \nu_2^2 +K ^2\right)\right\} \log n \right\}\\
    & \le \textstyle C_1 \cdot  \left\{\left\{ \frac{d\tau^2}{n_0} \wedge \left(\frac{d\tau^2}{n} + h_1^2 + 2C\tau^2 \frac{\log n}{n_0} \right)\right\} \log n\right\}  \le C_2 \cdot  \left\{ \frac{d\tau^2}{n_0} \wedge \left(\frac{d\tau^2}{n} + h_1^2  \right)\right\} (\log n)^2\\  
\end{aligned}
\]
resulting 
\[
\frac{\max_{\nu_1 \le h_1, \nu_2 \le h_2}\bR(\widehat \theta_0, \nu_1, \nu_2)}{\fR ((h_1, h_2), \bn, \tau)} \le C_2 (\log n)^2\,.
\]

\paragraph{Case 3:} Let  $ \tau \sqrt{d/n} \le h_1 \le \tau / (4C\sqrt{n_0})$ and $h_2 \le h_1$. Then $\fR ((h_1, h_2), \bn, \tau) = \frac{d\tau^2}{n} + h_2^2$
and it necessarily implies 
\[
\textstyle n_0 \le \frac{n \log(1/\delta)}{16C^2 d \log(1/\delta)} \le \frac{ n\log(1/\delta)}{16 C^2 \{d + \log(1/\delta)\}}, ~~ \text{and} ~~ K^2 = C\tau^2 \frac{\log(1/\delta)}{n_0} = 2C\tau^2 \frac{\log n}{n_0}
\]
Then, it holds
\[
\begin{aligned}
     \max_{\nu_1 \le h_1, \nu_2 \le h_2}\bR(\widehat \theta_0, \nu_1, \nu_2)
    & \le   \left\{\max_{\nu_1 , \nu_2 \le  \tau\sqrt{\frac{d}{n}}}\bR(\widehat \theta_0, \nu_1, \nu_2)\right\} \vee \left\{\max_{\nu_1 \le \nu_2, \nu_1 \le h_1,   \nu_2 \le h_2, \nu_2  >  \tau\sqrt{\frac{d}{n}}}\bR(\widehat \theta_0, \nu_1, \nu_2)\right\}\\
    & \qquad \vee \left\{\max_{\nu_1 > \nu_2, \nu_1 \le h_1,   \nu_2 \le h_2, \nu_1  >  \tau\sqrt{\frac{d}{n}}}\bR(\widehat \theta_0, \nu_1, \nu_2)\right\}\\
\end{aligned}
\] 
\[
\begin{aligned}
    &\textstyle \le C_1 \cdot \left\{\frac{d\tau^2}{n}\right\} \vee \left\{\max_{\nu_1 \le \nu_2 \le h_2}\left\{ \frac{d\tau^2}{n_0} \wedge \left(\frac{d\tau^2}{n} + \nu_1^2 \right)\right\} \log n\right\}\\
    & \qquad \vee\textstyle  \left\{\max_{   \nu_2 \le h_2}\left\{ \frac{d\tau^2}{n_0} \wedge \left(\frac{d\tau^2}{n} + \nu_2^2 +K ^2\right)\right\} \log n \right\}\\
    & \le\textstyle  C_1 \cdot  \left\{\max_{\nu_1 \le \nu_2 \le h_2}\left\{ \frac{d\tau^2}{n} + \nu_1^2 \right\} \log n\right\} \vee \left\{\max_{   \nu_2 \le h_2}\left\{ \frac{2d\tau^2\log n}{n_0} \wedge \left(\frac{d\tau^2}{n} + \nu_2^2 +K ^2\right)\right\} \log n \right\}\\
    & \le\textstyle  C_1 \cdot  \left\{\max_{\nu_1 \le \nu_2 \le h_2}\left\{ \frac{d\tau^2}{n} + \nu_1^2 \right\} \log n\right\} \vee \left\{\max_{   \nu_2 \le h_2}\left\{ \frac{d\tau^2}{n} + \nu_2^2 +2C\tau^2 \frac{\log n}{n_0}\right\} \log n \right\}\\
    &  \le\textstyle  C_2 \cdot  \left\{ \frac{d\tau^2}{n} + h_2^2  + \frac{\tau^2}{n_0}\right\} (\log n)^2\\
\end{aligned}
\]
which leads to 
\[
\begin{aligned}
    & \max_{\tau \sqrt{d/n} \le h_1 \le \tau / (4C\sqrt{n_0}), h_2 \le h_1}\frac{\max_{\nu_1 \le h_1, \nu_2 \le h_2}\bR(\widehat \theta_0, \nu_1, \nu_2)}{\fR ((h_1, h_2), \bn, \tau)} \\
    & \le C_2 \cdot \max_{\tau \sqrt{d/n} \le h_1 \le \tau / (4C\sqrt{n_0}), h_2 \le h_1} \textstyle \frac{\frac{d\tau^2}{n} + h_2^2  + \frac{\tau^2}{n_0}}{\frac{d\tau^2}{n} + h_2^2} \cdot (\log n)^2\\
    & \le C_2 \textstyle \cdot \frac{n}{d n_0} \cdot (\log n)^2
\end{aligned}
\]
\paragraph{Case 4:} Let  $  h_1 > \tau / (4C\sqrt{n_0})$ and $h_2 \le h_1$. Then $\fR ((h_1, h_2), \bn, \tau) = \frac{d\tau^2}{n} + h_2^2$
and it holds
\[
\begin{aligned}
     \max_{\nu_1 \le h_1, \nu_2 \le h_2}\bR(\widehat \theta_0, \nu_1, \nu_2)
    & \le   \left\{\max_{\nu_1 , \nu_2 \le  \tau\sqrt{\frac{d}{n}}}\bR(\widehat \theta_0, \nu_1, \nu_2)\right\} \vee \left\{\max_{\nu_1 \le \nu_2, \nu_1 \le h_1,   \nu_2 \le h_2, \nu_2  >  \tau\sqrt{\frac{d}{n}}}\bR(\widehat \theta_0, \nu_1, \nu_2)\right\}\\
    & \qquad \vee \left\{\max_{\nu_1 > \nu_2, \nu_1 \le h_1,   \nu_2 \le h_2, \nu_1  >  \tau\sqrt{\frac{d}{n}}}\bR(\widehat \theta_0, \nu_1, \nu_2)\right\}\\
\end{aligned}
\]
\[
\begin{aligned}
    & \le \textstyle C_1 \cdot \left\{\frac{d\tau^2}{n}\right\} \vee \left\{\max_{\nu_1 \le \nu_2 \le h_2}\left\{ \frac{d\tau^2}{n_0} \wedge \left(\frac{d\tau^2}{n} + \nu_1^2 + K^2  \right)\right\} \log n\right\}\\
    & \qquad \vee  \textstyle  \left\{\max_{   \nu_2 \le h_2}\left\{ \frac{d\tau^2}{n_0} \wedge \left(\frac{d\tau^2}{n} + \nu_2^2 +K ^2\right)\right\} \log n \right\}\\
    & \le \textstyle  C_1 \cdot   \left\{ \frac{2d\tau^2\log n}{n_0} \wedge \left(\frac{d\tau^2}{n} + h_2^2 +K ^2\right)\right\} \log n \\
    &  \le \textstyle  C_2 \cdot  \left\{ \frac{d\tau^2}{n_0} \wedge \left( \frac{d\tau^2}{n} + h_2^2 \right)\right\} (\log n)^2\\
\end{aligned}
\] 
\[
\text{Then,} \qquad \frac{\max_{\nu_1 \le h_1, \nu_2 \le h_2}\bR(\widehat \theta_0, \nu_1, \nu_2)}{\fR ((h_1, h_2), \bn, \tau)} \le C_2 (\log n)^2\,.
\]

\subsubsection{Supporting lemmas}

\begin{lemma}[Concentration of $T$]
\label{lem:diff-sq-concentration}
Define $T := \|\widetilde\theta_1 - \widetilde\theta_0\|_2^2 - \|\widetilde\theta_2 - \widetilde\theta_0\|_2^2$ and $\mu_j := \theta_j^\star - \theta_0^\star$. There exists an absolute constant $C>0$ such that for every $\delta\in(0,1)$, with probability at least $1-\delta$,
\[
\begin{aligned}
    \bigl|T - (\|\mu_1\|_2^2 - \|\mu_2\|_2^2)\bigr|
& \;\leq\; \textstyle
C\tau\sqrt{\log\tfrac{1}{\delta}}\,(\|\mu_1\|_2+\|\mu_2\|_2)\!\left(\tfrac{1}{\sqrt{n}}+\tfrac{1}{\sqrt{n_0}}\right)\\
& \qquad \textstyle + C\tau^2\!\left(\tfrac{d+\log(1/\delta)}{n} + \sqrt{\tfrac{d+\log(1/\delta)}{n}}\sqrt{\tfrac{\log(1/\delta)}{n_0}}\right).
\end{aligned}
\]
\end{lemma}

\begin{proof}
Write $\widetilde\theta_j = \theta_j^\star + \varepsilon_j$ so that $\widetilde\theta_j - \widetilde\theta_0 = \mu_j + Z_j$ where $Z_j := \varepsilon_j - \varepsilon_0$. Then
\[
T - (\|\mu_1\|_2^2 - \|\mu_2\|_2^2)
= 2(\langle\mu_1,Z_1\rangle - \langle\mu_2,Z_2\rangle)
+ (\|Z_1\|_2^2 - \|Z_2\|_2^2).
\]

\textit{Linear term.} Each $\langle \mu_j, \varepsilon_k\rangle$ is sub-Gaussian with proxy $\tau^2\|\mu_j\|_2^2/n_k$. A union bound gives, with probability $\geq 1-\delta/2$,
\[
|2(\langle\mu_1,Z_1\rangle - \langle\mu_2,Z_2\rangle)|
\leq C\tau\sqrt{\log\tfrac{1}{\delta}}\,(\|\mu_1\|_2+\|\mu_2\|_2)\!\left(\tfrac{1}{\sqrt{n}}+\tfrac{1}{\sqrt{n_0}}\right).
\]

\textit{Quadratic term.} Since $\|Z_j\|_2^2 - \|Z_j\|_2^2 = \|\varepsilon_1\|_2^2 - \|\varepsilon_2\|_2^2 - 2\langle\varepsilon_0, \varepsilon_1-\varepsilon_2\rangle$, we bound each part. Sub-Gaussian norm concentration gives $\|\varepsilon_j\|_2 \leq C\tau\sqrt{(d+\log(1/\delta))/n}$ with probability $\geq 1-\delta/4$, so $|\|\varepsilon_1\|_2^2 - \|\varepsilon_2\|_2^2| \leq C\tau^2(d+\log(1/\delta))/n$. Conditioning on $V:=\varepsilon_1-\varepsilon_2$ (which satisfies $\|V\|_2 \leq C\tau\sqrt{(d+\log(1/\delta))/n}$) and using the independence of $\varepsilon_0$, $\langle\varepsilon_0,V\rangle$ is sub-Gaussian with proxy $\tau^2\|V\|_2^2/n_0$, giving 
\[
|\langle\varepsilon_0,V\rangle| \leq C\tau^2\!\left(\tfrac{d+\log(1/\delta)}{n} + \sqrt{\tfrac{d+\log(1/\delta)}{n}}\sqrt{\tfrac{\log(1/\delta)}{n_0}}\right).
\]
 with probability $\geq 1-\delta/4$.

A union bound over all four events completes the proof.
\end{proof}

\begin{lemma} \label{lemma:bias-bounds}
  Define $\nu_1 = \|\theta_1 - \theta_0\|_2$,  $\nu_2 = \|\theta_2 - \theta_0\|_2$,  and $K ^2 :=  2C \tau^2 \left(\frac{d + \log(1/\delta)}{n} \vee \frac{\log(1/\delta)}{n_0}\right)$.
Under the event in Lemma \ref{lem:diff-sq-concentration},
the following holds: 
\begin{enumerate}
    \item If $\nu_2 > \nu_1$ and 
    \[
     \frac{C^2 \{d + \log(1/\delta)\}}{n} \le \frac{ \log(1/\delta)}{n_0}\,,
    \]
 then the event $\{T > (\tau^2/n_0) \log(1/\delta)\}$ implies 
 \[
 \nu_1^2 \ge \frac{\tau^2\log(1/\delta)}{2n_0}, \qquad 
\nu_2 \le \nu_1 + 2 C \tau  \sqrt{\frac{\log(1/\delta)}{n_0}}\,.
 \]
    \item  If $\nu_1 > \nu_2$, then the event $\{T \le (\tau^2/n_0) \log(1/\delta)\}$ implies $\nu_1 \le 3K + \nu_2$. 
\end{enumerate}
\end{lemma}

\begin{proof}{\bf of Lemma \ref{lemma:bias-bounds}.} 

\paragraph{Case 1.} We establish an upper bound on $\nu_2$ within the event $\{T > (\tau^2/n_0) \log(1/\delta)\}$ and $\cE_T(\delta)$. If $\nu_2 > \nu_1$, then the events imply
\[
\begin{aligned}
\textstyle    \frac{\tau^2\log(1/\delta)}{n_0}  < T  & \textstyle  \le \left (\nu_1^2 - \nu_2^2\right) + C \tau \sqrt{\log\frac1\delta}
\big(\nu_1 + \nu_2\big)
\left(
\frac1{\sqrt n} + \frac1{\sqrt{n_0}}
\right)\\
& \textstyle  \quad 
+ C \tau^2
\left[
\frac{d + \log(1/\delta)}{n}
+
\sqrt{\frac{d + \log(1/\delta)}{n}}
\sqrt{\frac{\log(1/\delta)}{n_0}}
\right]
\end{aligned}
\]
Since  $C^2 \{d + \log(1/\delta)\}/n \le \log (1/\delta)/(16n_0)$, then 
\[
\begin{aligned}
    & \textstyle  C \tau^2
\left[
\frac{d + \log(1/\delta)}{n}
+
\sqrt{\frac{d + \log(1/\delta)}{n}}
\sqrt{\frac{\log(1/\delta)}{n_0}}
\right] \\
& \textstyle  \le C^2 \tau^2 \frac{d + \log(1/\delta)}{n} + C \tau^2 \sqrt{\frac{d + \log(1/\delta)}{n}}
\sqrt{\frac{\log(1/\delta)}{n_0}}\\
& \textstyle  \le \frac{\tau^2 \log(1/\delta)}{16n_0} + \frac{\tau^2 \log(1/\delta)}{4n_0} \le \frac{\tau^2 \log(1/\delta)}{2n_0}\,.
\end{aligned}
\]
which further implies 
\[
\begin{aligned}
    0 & < \textstyle  \left (\nu_1^2 - \nu_2^2\right) + C \tau \sqrt{\log\frac1\delta}
\big(\nu_1 + \nu_2\big)
\left(
\frac1{\sqrt n} + \frac1{\sqrt{n_0}}
\right)\\
& \textstyle  \quad 
+ C \tau^2
\left[
\frac{d + \log(1/\delta)}{n}
+
\sqrt{\frac{d + \log(1/\delta)}{n}}
\sqrt{\frac{\log(1/\delta)}{n_0}}
\right] - \frac{\tau^2\log(1/\delta)}{n_0}\\
& \le \textstyle  \left (\nu_1^2 - \nu_2^2\right) + 4 C \tau \nu_2 \sqrt{\frac{\log(1/\delta)}{n_0}} - \frac{\tau^2\log(1/\delta)}{2n_0}\\
& \le \textstyle   - \nu_2^2 +  4 C \tau \nu_2 \sqrt{\frac{\log(1/\delta)}{n_0}}  + \left(\nu_1^2 - \frac{\tau^2\log(1/\delta)}{2n_0}\right)\,.
\end{aligned}
\] This necessarily implies $\nu_1^2 \ge \frac{\tau^2\log(1/\delta)}{2n_0}$ and $\nu_2 \le \nu_1 + 2 C \tau  \sqrt{\frac{\log(1/\delta)}{n_0}}$.


\paragraph{Case 2:} The event $\{T \le  (\tau^2/n_0) \log(1/\delta)\}$ and $\nu_1 > \nu_2$ imply
\[
\begin{aligned}
\textstyle      \frac{\tau^2\log(1/\delta)}{n_0}  \ge  T  & \ge  \textstyle  \left (\nu_1^2 - \nu_2^2\right) - C \tau \sqrt{\log\frac1\delta}
\big(\nu_1 + \nu_2\big)
\left(
\frac1{\sqrt n} + \frac1{\sqrt{n_0}}
\right)\\
& \textstyle  \quad 
- C \tau^2
\left[
\frac{d + \log(1/\delta)}{n}
+
\sqrt{\frac{d + \log(1/\delta)}{n}}
\sqrt{\frac{\log(1/\delta)}{n_0}}
\right]
\end{aligned}
\]
Then, 
\[
\begin{aligned}
    0 & < \textstyle \left (\nu_2^2 - \nu_1^2\right) + C \tau \sqrt{\log\frac1\delta}
\big(\nu_1 + \nu_2\big)
\left(
\frac1{\sqrt n} + \frac1{\sqrt{n_0}}
\right)\\
&\textstyle  \quad 
+ C \tau^2
\left[
\frac{d + \log(1/\delta)}{n}
+
\sqrt{\frac{d + \log(1/\delta)}{n}}
\sqrt{\frac{\log(1/\delta)}{n_0}}
\right] + \frac{\tau^2\log(1/\delta)}{n_0}\\
&\textstyle  < \left (\nu_2^2 - \nu_1^2\right) + C \tau \sqrt{\log\frac1\delta}
\big(\nu_1 + \nu_2\big)
\left(
\frac1{\sqrt n} + \frac1{\sqrt{n_0}}
\right)\\
&\textstyle  \quad 
+ C \tau^2
\left[
\frac{d + \log(1/\delta)}{n}
+
\sqrt{\frac{d + \log(1/\delta)}{n}}
\sqrt{\frac{\log(1/\delta)}{n_0}}
 + \frac{\tau^2\log(1/\delta)}{n_0}\right]\\
& \le \left (\nu_2^2 - \nu_1^2\right) + 4 C \tau \nu_1 \sqrt{\frac{\log(1/\delta)}{n_0}} + 4C \tau^2 \left(\frac{d + \log(1/\delta)}{n} \vee \frac{\log(1/\delta)}{n_0}\right)\\
& \le \nu_2 ^2 - \nu_1^2 + 2 K\nu_1  + K^2\,,
\end{aligned}
\]  Then, we have $\nu_1^2 - 2K \nu_1 - (\nu_2^2 + K^2) < 0 $, which implies  $\nu_1 \le K + \sqrt{2K^2 + \nu_2^2}\le 3K + \nu_2$.
This completes the proof.

\end{proof}

\begin{lemma}\label{lemma:expectation-ub-1}
    Define $\nu_1 := \|\theta_1 - \theta_0\|_2 $,  $\nu_2:= \|\theta_2 - \theta_0\|_2 $, and $K^2 := 2C \tau^2 \left(\frac{d + \log(1/\delta)}{n} \vee \frac{\log(1/\delta)}{n_0}\right)$
 where $C > 0$ is the constant within $\cE_T$, as defined in Lemma \ref{lem:diff-sq-concentration}. If $\delta = n^{-2}$, then there exists a constant $C_1 > 0$ such that
 \[
 \begin{aligned}
     & \bR(\widehat \theta_0, \nu_1, \nu_2) := \Ex \left[ \|\widehat \theta_0 - \theta_0\|_2 ^2\right]\\
     & \le \begin{cases}
     C_1 \frac{d\tau^2}{n}  & \nu_1, \nu_2 \le \tau\sqrt{\frac{d}{n}}\\
     C_1 \left\{ \frac{d\tau^2}{n_0} \wedge \left(\frac{d\tau^2}{n} + \nu_2^2 +K ^2\right)\right\} \log n & \nu_1 > \tau\sqrt{\frac{d}{n}}, ~\nu_2 \le \nu_1\\
     C_1 \left\{ \frac{d\tau^2}{n_0} \wedge \left(\frac{d\tau^2}{n} + \nu_1^2 \right)\right\} \log n & \nu_2 > \tau\sqrt{\frac{d}{n}}, ~\nu_1 \le \nu_2, ~\frac{16 C^2 \{d + \log(1/\delta)\}}{n} \le \frac{\log(1/\delta)}{n_0}    \,.
 \end{cases}
 \end{aligned}
 \]

\end{lemma}

\begin{proof}{\bf of Lemma \ref{lemma:expectation-ub-1}}
    Define  the events $\cE_1 := \{T \le (\tau^2/n_0) \log(1/\delta)\}$ and $\cE_2 := \{T > (\tau^2/n_0) \log(1/\delta)\}$.
 Applying Theorem \ref{thm:intersection-oracle-max}, for any $k \in \{1, 2\}$, with a probability of at least $1 - \delta_1$, we have 
\[
\textstyle \cF_k(\delta_1) := \left\{\left \|\check \theta_k - \theta_0\right\|_2 ^2 \le C_1\left\{ \frac{d\tau^2}{n_0} \wedge \left(\frac{d\tau^2}{n} + \nu_k^2\right)\right\} \log \left(\frac{1}{\delta_1}\right)\right\}\,,
\]
for some $C_1 > 0$.  

\paragraph{Step 1:} Assume that $\nu_1^2 , \nu_2^2 \le d\tau^2 / n$. Then, under the event $\cF_1(\delta) \cap \cF_2 (\delta)$, we have 
\[
\begin{aligned}
    \|\widehat \theta_0 - \theta_0\|_2 ^2 & \le \max_{k = 1, 2}  \|\check \theta_k - \theta_0\|_2 ^2\\
    &\textstyle  \le C_1\left\{ \frac{d\tau^2}{n_0} \wedge \left(\frac{d\tau^2}{n} + \max_{k = 1, 2}\nu_k^2\right)\right\} \log \left(\frac{1}{\delta_1}\right)  \le C_1\left\{ \frac{d\tau^2}{n} \right\} \log \left(\frac{1}{\delta_1}\right)\,.
\end{aligned}
\] Since $P\{\cF_1(\delta) \cap \cF_2 (\delta)\}\ge 1- 2\delta_1$, we have a sub-Gaussian concentration of $\|\widehat \theta_0 - \theta_0\|_2 $ at a rate $\tau\sqrt{d/n}$. This implies, $\Ex \left[ \|\widehat \theta_0 - \theta_0\|_2 ^2\right] \le C_1\left\{ \frac{d\tau^2}{n} \right\}$.
\paragraph{Step 2:} Assume that $d\tau^2 / n < \nu_1^2$ and $\nu_2^2 \le \nu_1^2$.  Then, by Lemma \ref{lemma:bias-bounds}, $\cE_1\cap \cE_T$ necessarily implies $\nu_1 \le 3K + \nu_2$. Thus,
\[
\begin{aligned}
    & \textstyle  P \left\{ \|\widehat \theta_0 - \theta_0 \|_2^2 > C_1\left\{ \frac{d\tau^2}{n_0} \wedge \left(\frac{d\tau^2}{n} + (\nu_2 + 3K) ^2\right)\right\} \log \left(\frac{1}{\delta}\right) \right\}\\
    &\textstyle  \le  P \left\{ \|\widehat \theta_0 - \theta_0 \|_2^2 > C_1\left\{ \frac{d\tau^2}{n_0} \wedge \left(\frac{d\tau^2}{n} + (\nu_2 + 3K) ^2\right)\right\} \log \left(\frac{1}{\delta}\right), \; \cE_1 \cap \cE_T \right\} + P(\cE_T^c)\\
    &\textstyle  \le  P \left\{ \|\check \theta_1 - \theta_0 \|_2^2 > C_1\left\{ \frac{d\tau^2}{n_0} \wedge \left(\frac{d\tau^2}{n} + \nu_1 ^2\right)\right\} \log \left(\frac{1}{\delta}\right)  \right\} + \delta  \le 2\delta\,.
\end{aligned}
\] Consequently, defining $\cE_3 := \left\{ \|\widehat \theta_0 - \theta_0 \|_2^2 > C_1\left\{ \frac{d\tau^2}{n_0} \wedge \left(\frac{d\tau^2}{n} + (\nu_2 + 3K) ^2\right)\right\} \log \left(\frac{1}{\delta}\right) \right\}$
\[
\begin{aligned}
     \Ex \left[ \|\widehat \theta_0 - \theta_0\|_2 ^2\right]
    & = \Ex \left[ \|\widehat \theta_0 - \theta_0\|_2 ^2 \bbI_{\cE_3^c}\right] + \Ex \left[ \|\widehat \theta_0 - \theta_0\|_2 ^2 \bbI_{\cE_3}\right]\\
    & \le \textstyle C_1\left\{ \frac{d\tau^2}{n_0} \wedge \left(\frac{d\tau^2}{n} + \nu_1 ^2\right)\right\} \log \left(\frac{1}{\delta}\right) + C_2 \sqrt{\delta}
\end{aligned}
\] 
where the first part of the inequality holds because, within the event $\cE_3 ^c$, we have 
\[
\textstyle \|\widehat \theta_0 - \theta_0 \|_2^2 \le C_1\left\{ \frac{d\tau^2}{n_0} \wedge \left(\frac{d\tau^2}{n} + (\nu_2 + 3K) ^2\right)\right\} \log \left(\frac{1}{\delta}\right)\,,
\] and the second part of the inequality holds with the following application of the Cauchy-Schwarz inequality
\[
\begin{aligned}
    & \Ex \left[ \|\widehat \theta_0 - \theta_0\|_2 ^2 \bbI_{\cE_3}\right] \le \left\{P(\cE_3) \Ex \left[ \|\widehat \theta_0 - \theta_0\|_2 ^4 \right]\right\}^{\frac12}\le \sqrt{\delta} \bigg\{ \sum_{k = 0, 1, 2}\Ex \left[ \|\widetilde \theta_k - \theta_0\|_2 ^4 \right] \bigg\}^{\frac12}\le C_2 \sqrt{\delta}\,.
\end{aligned}
\]
By letting $\delta = n^{-2}$, we obtain 
\[
\begin{aligned}
  \textstyle  \Ex \left[ \|\widehat \theta_0 - \theta_0\|_2 ^2\right] & \le \textstyle 2C_1\left\{ \frac{d\tau^2}{n_0} \wedge \left(\frac{d\tau^2}{n} + (\nu_2 + 3K) ^2\right)\right\} \log n  + \frac{C_2}{n}\\
    & \le\textstyle  C_3 \left\{ \frac{d\tau^2}{n_0} \wedge \left(\frac{d\tau^2}{n} + \nu_2^2 +K ^2\right)\right\} \log n\,.
\end{aligned}
\]


Let $16C^2 \{d + \log(1/\delta)\}/n \le \log (1/\delta)/n_0$.   For $d\tau^2 / n < \nu_2^2$ and $\nu_1^2 \le \nu_2^2$,  by Lemma \ref{lemma:bias-bounds}, the event $\cE_2 \cap \cE_T$ implies 
\[
\textstyle \nu_1^2 \ge \frac{\tau^2\log(1/\delta)}{2n_0}, \qquad 
\nu_2 \le \nu_1 + 2 C \tau  \sqrt{\frac{\log(1/\delta)}{n_0}}\,.
\]
Similarly to {\bf Step 2}, we have 
\[
\begin{aligned}
    \textstyle \Ex \left[ \|\widehat \theta_0 - \theta_0\|_2 ^2\right] 
    & \textstyle \le C_4 \left\{ \frac{d\tau^2}{n_0} \wedge \left(\frac{d\tau^2}{n} + \nu_1^2 + \frac{\tau^2\log(1/\delta)}{n_0}\right)\right\} \log n \\
    & \textstyle \le C_5 \left\{ \frac{d\tau^2}{n_0} \wedge \left(\frac{d\tau^2}{n} + \nu_1^2 \right)\right\} \log n \,.
\end{aligned}
\]

\end{proof}

\subsection{Proof of the Theorem \ref{thm:cost-lower-bound-general-m}}

\label{proof:thm:cost-lower-bound-general-m}

In this subsection we prove Theorem \ref{thm:cost-lower-bound-general-m}. Throughout, we assume
\[
n_1=\cdots=n_m=n.
\]
The proof combines two distinct lower-bound constructions, corresponding to the two terms in the
statement of the theorem. The first construction captures the genuinely growing-$m$ obstruction and
is based on a Rademacher prior on the source cluster labels together with a Fano argument over a
packing of target directions. The second construction is a two-point argument of Le Cam type and
produces the target-noise-scale term.

\paragraph{Step 1: A random-sign two-cluster construction.}
We first prove that
\begin{equation}
\label{eq:growing-m-first-term-goal}
\textstyle \fC^\star([0,1]^m)
\gtrsim
\frac{n_0+mn}{d\tau^2}
\min\!\left\{
\frac{\tau^2 d}{n_0},
\frac{\tau^2}{n}\sqrt{\frac{d}{m}}
\right\}.
\end{equation}
Let $A:=\frac{dn}{n_0}$ and $
B:=\sqrt{\frac{d}{m}}$.
Then the right-hand side of \eqref{eq:growing-m-first-term-goal} may be rewritten as
\begin{equation}
\label{eq:growing-m-first-term-rewrite}
\textstyle \min\left\{
1+\frac{mn}{n_0},
\frac{n_0}{n\sqrt{dm}}+\sqrt{\frac{m}{d}}
\right\}
=
\min\left\{
1+\frac{A}{B^2},
\frac{B}{A}+\frac{1}{B}
\right\}.
\end{equation}
We first treat the small-signal regime
\begin{equation}
\label{eq:small-signal-growing-m}
\min\{A,B\}\le c_\star,
\end{equation}
for a sufficiently small universal constant $c_\star>0$. The complementary regime is treated at
the end of this step.

Let $\{v_1,\dots,v_M\}\subset S^{d-1}$ be a $c_0$-packing of the unit sphere satisfying
\[
\|v_j-v_\ell\|_2\ge c_0,\qquad j\neq \ell,
\]
and $\log M \ge c_1 d$
for universal constants $c_0,c_1>0$. Let $\xi_1,\dots,\xi_m$ be i.i.d.\ Rademacher variables and
fix a scale $\alpha>0$, to be chosen later. For each $j\in\{1,\dots,M\}$ and each realization
$\xi=(\xi_1,\dots,\xi_m)\in\{\pm1\}^m$, let $P_{j,\xi}$ denote the product Gaussian law under
which
\[
\textstyle \widetilde \theta_0 \sim N\!\left(\alpha v_j,\frac{\tau^2}{n_0}I_d\right),
\qquad
\widetilde \theta_i \sim N\!\left(\xi_i\alpha v_j,\frac{\tau^2}{n}I_d\right),
\quad i=1,\dots,m,
\]
independently.

For every realization $\xi$, the associated bias vector has coordinates
\[
h_i(\xi)=\|\xi_i\alpha v_j-\alpha v_j\|_2
=
\begin{cases}
0, & \xi_i=1,\\
2\alpha, & \xi_i=-1.
\end{cases}
\]
Hence $P_{j,\xi}\in\cP(\bh(\xi),\bn,\tau)$ provided $2\alpha\le 1$. Let
\[
K(\xi):=\sum_{i=1}^m \bbI\{\xi_i=1\}
\sim \mathrm{Bin}(m,1/2)
\]
denote the number of sources aligned with the target. Since the oracle may pool the target with
the aligned sources and ignore the rest, we have
\[
\fR(\bh(\xi),\bn,\tau)\lesssim \frac{d\tau^2}{n_0+K(\xi)n}.
\]
Let $\cE_m:=\left\{\xi:K(\xi)\ge \frac{m}{4}\right\}$.
By Hoeffding's inequality,
\[
\mathbb P(\cE_m^c)\le e^{-c m}
\]
for a universal constant $c>0$. Consequently,
\begin{equation}
\label{eq:bayes-oracle-growing-m}
\mathbb E_\xi\big[\fR(\bh(\xi),\bn,\tau)\big]
\lesssim
\frac{d\tau^2}{n_0+mn}.
\end{equation}
Indeed, on $\cE_m$ we have $K(\xi)\ge m/4$, while on $\cE_m^c$ the crude bound
$\fR(\bh(\xi),\bn,\tau)\le d\tau^2/n_0$ is sufficient, and the latter contribution is negligible
for large $m$.

Now define the mixture law
\[
\overline P_j:=\mathbb E_\xi[P_{j,\xi}],
\qquad j=1,\dots,M,
\]
where the expectation is taken with respect to the i.i.d.\ Rademacher prior on $\xi$. Let $\pi$ be
the prior that chooses $j$ uniformly from $\{1,\dots,M\}$ and then draws $\xi$ from the i.i.d.\
Rademacher law. For any estimator $\widehat\theta_0$, the weighted-average inequality
$\sup_i a_i/b_i \ge (\sum_i a_i)/(\sum_i b_i)$ gives
\[
\sup_{\bh\in[0,1]^m}
\frac{\sup_{P\in\cP(\bh,\bn,\tau)}
\Ex_P\!\left[\|\widehat\theta_0-\theta_0^\star\|_2^2\right]}{
\fR(\bh,\bn,\tau)}
\ge
\frac{
\frac1M\sum_{j=1}^M
\mathbb E_{\overline P_j}\!\left[\|\widehat\theta_0-\alpha v_j\|_2^2\right]
}{
\frac1M\sum_{j=1}^M \mathbb E_\xi\big[\fR(\bh(\xi),\bn,\tau)\big]
}.
\]
Therefore, by \eqref{eq:bayes-oracle-growing-m},
\begin{equation}
\label{eq:ratio-reduction-growing-m}
\fC^\star([0,1]^m)
\gtrsim
\frac{n_0+mn}{d\tau^2}
\inf_{\widehat\theta_0}
\frac1M\sum_{j=1}^M
\mathbb E_{\overline P_j}\!\left[\|\widehat\theta_0-\alpha v_j\|_2^2\right].
\end{equation}

The remaining task is to lower bound the Bayes estimation risk over the finite family
$\{\overline P_j:1\le j\le M\}$. For this, we compare all mixture laws to a common reference law.
Let
\[
P_\star
:=
N\!\left(0,\frac{\tau^2}{n_0}I_d\right)\otimes Q_0^{\otimes m},
\qquad
Q_0:=N\!\left(0,\frac{\tau^2}{n}I_d\right).
\]
The key input is the following lemma.

\begin{lemma}
\label{lem:kl-growing-m-rademacher}
There exists a universal constant $C>0$ such that, for every $j$,
\[
D_{\mathrm{KL}}(\overline P_j\|P_\star)
\le
C\left(
\frac{n_0\alpha^2}{\tau^2}
\;+\;
\frac{mn^2\alpha^4}{\tau^4}
\right).
\]
\end{lemma}

Choose
\begin{equation}
\label{eq:choice-alpha-growing-m}
\alpha^2
:=
c_2
\min\!\left\{
\frac{\tau^2 d}{n_0},
\frac{\tau^2}{n}\sqrt{\frac{d}{m}}
\right\}
\end{equation}
with $c_2>0$ sufficiently small. Under \eqref{eq:small-signal-growing-m}, this choice satisfies
\[
\frac{n\alpha^2}{\tau^2}
\le
c_2 \min\{A,B\}
\le
c_2 c_\star,
\]
so the lemma applies. Moreover, Lemma \ref{lem:kl-growing-m-rademacher} and $\log M\asymp d$
imply
\[
\max_{1\le j\le M} D_{\mathrm{KL}}(\overline P_j\|P_\star)
\le
\frac{1}{8}\log M,
\]
provided $c_2>0$ is chosen sufficiently small. Moreover, the packing property yields
\[
\|\alpha v_j-\alpha v_\ell\|_2^2
\ge
c_0^2\alpha^2,
\qquad j\neq \ell.
\]
Hence the standard Fano inequality with a common reference law gives
\[
\inf_{\widehat \theta_0}
\frac1M\sum_{j=1}^M
\mathbb E_{\overline P_j}\!\left[\|\widehat\theta_0-\alpha v_j\|_2^2\right]
\gtrsim
\alpha^2.
\]
Substituting this bound into \eqref{eq:ratio-reduction-growing-m} and using
\eqref{eq:choice-alpha-growing-m} proves \eqref{eq:growing-m-first-term-goal}.

It remains to consider the complementary regime $\min\{A,B\}>c_\star$. We claim that in this
regime the right-hand side of \eqref{eq:growing-m-first-term-goal} is bounded by a universal
constant. Indeed, if $A\le B$, then
\[
1+\frac{A}{B^2}\le 1+\frac{1}{B}\le 1+\frac{1}{c_\star},
\]
while if $A>B$, then
\[
\frac{B}{A}+\frac{1}{B}\le 1+\frac{1}{c_\star}.
\]
By \eqref{eq:growing-m-first-term-rewrite}, the first term in the theorem is therefore bounded by
a universal constant in this regime. On the other hand, by the definition of
$\fC^\star([0,1]^m)$ and of the local benchmark $\fR(\bh,\bn,\tau)$, we always have
\[
\fC^\star([0,1]^m)\ge 1.
\]
Indeed, for every estimator $\widehat\theta_0$ and every configuration $\bh$,
\[
\sup_{P\in \cP(\bh,\bn,\tau)}
\Ex_P\!\left[\|\widehat\theta_0-\theta_0^\star\|_2^2\right]
\ge
\fR(\bh,\bn,\tau),
\]
and taking the supremum over $\bh$ and then the infimum over $\widehat\theta_0$ preserves this
lower bound. This completes the proof of \eqref{eq:growing-m-first-term-goal}.

\paragraph{Step 2: A balanced two-cluster two-point construction.}
We now prove that
\begin{equation}
\label{eq:growing-m-second-term-goal}
\fC^\star([0,1]^m)
\gtrsim
\frac{n_0+mn}{d\,n_0}.
\end{equation}
Fix a unit vector $u\in S^{d-1}$ and let
\[
\mu_1:=0_d,
\qquad
\mu_2:=\delta u,
\]
where $\delta>0$ will be chosen below. Let $P_0$ and $P_1$ be the two product Gaussian laws under
which
\[
\widetilde \theta_i
\sim
N\!\left(\mu_1,\frac{\tau^2}{n}I_d\right),
\quad i=1,\dots,m/2,
\qquad
\widetilde \theta_i
\sim
N\!\left(\mu_2,\frac{\tau^2}{n}I_d\right),
\quad i=m/2+1,\dots,m,
\]
while
\[
\widetilde \theta_0
\sim
N\!\left(\mu_1,\frac{\tau^2}{n_0}I_d\right)
\quad\text{under }P_0,
\qquad
\widetilde \theta_0
\sim
N\!\left(\mu_2,\frac{\tau^2}{n_0}I_d\right)
\quad\text{under }P_1.
\]
Under either law, exactly $m/2$ sources are aligned with the target and the remaining $m/2$
sources are at Euclidean distance $\delta$ from the target. Hence the corresponding bias vector
belongs to $[0,1]^m$ provided $\delta\le 1$, and the oracle risk satisfies
\[
\fR(\bh,\bn,\tau)\lesssim \frac{d\tau^2}{n_0+mn}.
\]

The source observations have the same distribution under $P_0$ and $P_1$, so all information
distinguishing the two hypotheses comes from the target observation $\widetilde\theta_0$. Thus
\[
D_{\mathrm{KL}}(P_0\|P_1)
=
\frac{n_0}{2\tau^2}\|\mu_1-\mu_2\|_2^2
=
\frac{n_0\delta^2}{2\tau^2}.
\]
Choose $\delta^2=c_3\tau^2/n_0$ with $c_3>0$ sufficiently small, so that
$D_{\mathrm{KL}}(P_0\|P_1)\le c_4<1$. Then Le Cam's two-point lemma yields
\[
\inf_{\widehat \theta_0}
\max\left\{
\Ex_{P_0}\!\left[\|\widehat\theta_0-\mu_1\|_2^2\right],
\Ex_{P_1}\!\left[\|\widehat\theta_0-\mu_2\|_2^2\right]
\right\}
\gtrsim
\delta^2
\asymp
\frac{\tau^2}{n_0}.
\]
Dividing by the oracle rate above proves \eqref{eq:growing-m-second-term-goal}.

\paragraph{Conclusion.}
Combining \eqref{eq:growing-m-first-term-goal} and \eqref{eq:growing-m-second-term-goal}, we
obtain
\[
\fC^\star([0,1]^m)
\gtrsim
\max\left[
\frac{n_0+mn}{d\tau^2}
\min\!\left\{
\frac{\tau^2 d}{n_0},
\frac{\tau^2}{n}\sqrt{\frac{d}{m}}
\right\},
\;
\frac{n_0+mn}{d\,n_0}
\right],
\]
which is the claim of Theorem \ref{thm:cost-lower-bound-general-m}.

\medskip
\noindent
\emph{Proof of Lemma \ref{lem:kl-growing-m-rademacher}.}
We decompose the mixture laws into their target and source components. Write
\[
\overline P_j=\overline P_j^{(0)}\otimes \overline P_j^{(1:m)},
\qquad
\;P_\star=P_\star^{(0)}\otimes P_\star^{(1:m)},
\]
where $\overline P_j^{(0)}$ denotes the marginal law of the target estimator
$\widetilde\theta_0$, and $\overline P_j^{(1:m)}$ denotes the joint marginal law of the source
estimators $(\widetilde\theta_1,\dots,\widetilde\theta_m)$. Here
\[
\overline P_j^{(0)}=N\!\left(\alpha v_j,\frac{\tau^2}{n_0}I_d\right),
\qquad
\;P_\star^{(0)}=N\!\left(0,\frac{\tau^2}{n_0}I_d\right),
\]
and therefore
\[
D_{\mathrm{KL}}(\overline P_j^{(0)}\|P_\star^{(0)})
=
\frac{n_0}{2\tau^2}\|\alpha v_j\|_2^2
\lesssim
\frac{n_0\alpha^2}{\tau^2}.
\]

For the source part, independence of the Rademacher labels implies that
\[
\overline P_j^{(1:m)}=Q_j^{\otimes m},
\qquad
\;P_\star^{(1:m)}=Q_0^{\otimes m},
\]
where
\[
Q_j
=
\frac{1}{2}N\!\left(\alpha v_j,\frac{\tau^2}{n}I_d\right)
\;+\;
\frac{1}{2}N\!\left(-\alpha v_j,\frac{\tau^2}{n}I_d\right),
\qquad
Q_0=N\!\left(0,\frac{\tau^2}{n}I_d\right).
\]
Consequently,
\[
D_{\mathrm{KL}}(\overline P_j^{(1:m)}\|P_\star^{(1:m)})
=
m\,D_{\mathrm{KL}}(Q_j\|Q_0).
\]

It remains to control the one-source divergence. Let $\sigma^2=\tau^2/n$, and let $\phi_\sigma$
denote the density of $N(0,\sigma^2 I_d)$. A direct calculation gives
\[
\frac{dQ_j}{d\phi_\sigma}(x)
=
\exp\!\left(-\frac{\alpha^2}{2\sigma^2}\right)
\cosh\!\left(\frac{\alpha}{\sigma^2}v_j^\top x\right).
\]
Hence
\[
1+\chi^2(Q_j\|Q_0)
=
\int \left(\frac{dQ_j}{d\phi_\sigma}\right)^2 d\phi_\sigma
=
\exp\!\left(-\frac{\alpha^2}{\sigma^2}\right)
\mathbb E\left[\cosh^2\!\left(\frac{\alpha}{\sigma^2}v_j^\top X\right)\right],
\qquad X\sim N(0,\sigma^2 I_d).
\]
Since $(\alpha/\sigma^2)v_j^\top X\sim N(0,\alpha^2/\sigma^2)$, the identity
$\cosh^2(t)=\{1+\cosh(2t)\}/2$ and the Gaussian moment formula
$\mathbb E[\cosh(2Y)]=e^{2\Var(Y)}$ for centered Gaussian $Y$ imply
\[
1+\chi^2(Q_j\|Q_0)
=
\cosh\!\left(\frac{\alpha^2}{\sigma^2}\right).
\]
Using $D_{\mathrm{KL}}(Q_j\|Q_0)\le \chi^2(Q_j\|Q_0)$ and the bound
$\cosh(u)-1\lesssim u^2$ for $|u|\le c_\star c_2$, it follows that
\[
D_{\mathrm{KL}}(Q_j\|Q_0)
\lesssim
\frac{\alpha^4}{\sigma^4}
=
\frac{n^2\alpha^4}{\tau^4}.
\]
Combining the target and source bounds yields
\[
D_{\mathrm{KL}}(\overline P_j\|P_\star)
\le
D_{\mathrm{KL}}(\overline P_j^{(0)}\|P_\star^{(0)})
\;+\;
D_{\mathrm{KL}}(\overline P_j^{(1:m)}\|P_\star^{(1:m)})
\lesssim
\frac{n_0\alpha^2}{\tau^2}
\;+\;
\frac{mn^2\alpha^4}{\tau^4},
\]
which proves the lemma.

\subsection{Proof of Theorem \ref{thm:model-selection}}
\label{proof:thm:model-selection}

\begin{proof}{\bf of Theorem \ref{thm:model-selection}.}
    Applying Lemma \ref{lemma:model-selection-1}, there exists a constant $C_1 > 0$ such that 
\begin{equation}\label{eq:thm-ms-1}
    \begin{aligned}
      \Ex\left[\left\|\widehat \theta_0^{\mathrm{MS}} - \theta_0\right\|_2^2\Big\vert \xi_1, \dots, \xi_M\right] \le C_1 \cdot \left[ \left(\min_{j=1, \dots, M} \|\xi_j- \theta_0\|_2^2\right) +  \tau^2 \frac{\log M \wedge d}{n_0}\right]  
    \end{aligned}
\end{equation}
By taking the expectation on both sides, we obtain
    \[
    \begin{aligned}
        \Ex\left[\left\|\widehat \theta_0^{\mathrm{MS}} - \theta_0\right\|_2^2\right] \le C_1 \cdot \left[ \Ex\left(\min_{j=1, \dots, M} \|\xi_j- \theta_0\|_2^2\right) +  \tau^2 \frac{\log M \wedge d}{n_0}\right]\,. 
    \end{aligned}
    \]
    Since, 
    \[
    \begin{aligned}
        \Ex \left[\min_{j=1, \dots, M} \|\xi_j- \theta_0\|_2^2\right] & \le \min_{j=1, \dots, M}\Ex \left[ \|\xi_j- \theta_0\|_2^2\right]\\
        & \le \min_{j = 1, \dots, M} \left(\left \|\frac{\sum_{k \in U_j} n_k  \theta_k }{\sum_{k \in U_j} n_k + \frac{n_0}{2}} - \theta_0\right\|_2^2 + \frac{d\tau^2}{\sum_{k \in U_j} n_k + \frac{n_0}{2}} \right)\,,
    \end{aligned}
    \] this completes the proof.
\end{proof}

\begin{lemma}\label{lemma:model-selection-1}
    Let $M \ge 1$ be a positive integer and $\{\beta_j : j = 1, \dots, M\}$ be a fixed set of vectors. Let $\widetilde{\theta}_0^{(1)}$ be an estimator of $\theta_0$ satisfying, for all $t > 0$ and $a \in \mathbb{R}^d$,
    \[
        P\!\left(a^\top \bigl\{\widetilde{\theta}_0^{(1)} - \theta_0\bigr\} > t\right)
        \le C_\tau \exp\!\left(-\frac{n_0 t^2}{4\tau^2 \|a\|_2^2}\right).
    \]
    Define the model-selection estimator $\beta_{\hat{j}}$, where
    \[
        \hat{j} := \operatorname*{arg\,min}_{j=1,\dots,M} \left\|\beta_j - \widetilde{\theta}_0^{(1)}\right\|_2^2.
    \]
    Then there exists a constant $C > 0$ such that
    \[
        \mathbb{E}\!\left[\left\|\beta_{\hat{j}} - \theta_0\right\|_2^2\right]
        \le 2\!\left(\min_{j=1,\dots,M} \|\beta_j - \theta_0\|_2^2\right)
        + \frac{4C^2\tau^2\bigl((\log M \wedge d) + 1\bigr)}{n_0}.
    \]
\end{lemma}

\begin{proof}
Define $\mu_j := \beta_j - \theta_0$ and $\varepsilon_0 := \widetilde{\theta}_0^{(1)} - \theta_0$, and let
$j_0 := \operatorname*{arg\,min}_{j} \|\mu_j\|_2^2$.
Expanding the squared norm gives
\[
    \left\|\beta_j - \widetilde{\theta}_0^{(1)}\right\|_2^2
    = \|\mu_j\|_2^2 + 2\mu_j^\top \varepsilon_0 + \|\varepsilon_0\|_2^2,
\]
so controlling $|\mu_j^\top \varepsilon_0|$ uniformly over $j$ is the key step.

\paragraph{Uniform control of $|\mu_j^\top \varepsilon_0|$.}
We handle two cases based on the relative sizes of $d$ and $\log M$.

\begin{itemize}
    \item \textbf{Case 1} ($d \ge \log M$): For each $j$, the inner product
    $\mu_j^\top \varepsilon_0$ is sub-Gaussian with variance proxy
    $2\tau^2\|\mu_j\|_2^2/n_0$. Applying a union bound over $j = 1, \dots, M$
    and setting $\delta = M\eta$, there exists $C > 0$ such that with
    probability at least $1 - \delta$,
    \[
         \bigl|\mu_j^\top \varepsilon_0\bigr|
        \le C\tau\|\mu_j\|_2\sqrt{\frac{\log M + \log(1/\delta)}{n_0}}\quad \text{for all} ~ j = 1, \dots, M.
    \]

    \item \textbf{Case 2} ($d < \log M$): By Cauchy--Schwarz,
    $|\mu_j^\top \varepsilon_0| \le \|\mu_j\|_2\|\varepsilon_0\|_2$. Since
    $\varepsilon_0$ is a sub-Gaussian vector in $\mathbb{R}^d$, there exists
    $C > 0$ such that with probability at least $1 - \delta$,
    \[
        \|\varepsilon_0\|_2 \le C\tau\sqrt{\frac{d + \log(1/\delta)}{n_0}},
    \]
    which yields the same bound with $\log M$ replaced by $d$.
\end{itemize}

Combining both cases, with probability at least $1 - \delta$, the following
event $\mathcal{E}$ holds simultaneously for all $j$:
\[
    \left|\left\|\beta_j - \widetilde{\theta}_0^{(1)}\right\|_2^2
    - \|\varepsilon_0\|_2^2 - \|\mu_j\|_2^2\right|
    \le C\tau\|\mu_j\|_2
    \sqrt{\frac{(\log M \wedge d) + \log(1/\delta)}{n_0}}.
\]

\paragraph{Bounding $\|\mu_{\hat{j}}\|_2^2$ on $\mathcal{E}$.}
On the event $\mathcal{E}$, whenever $\hat{j} \ne j_0$ the definition of
$\hat{j}$ implies
$\|\beta_{\hat{j}} - \widetilde{\theta}_0^{(1)}\|_2^2
\le \|\beta_{j_0} - \widetilde{\theta}_0^{(1)}\|_2^2$,
which gives
\[
    \|\mu_{\hat{j}}\|_2^2 - \|\mu_{j_0}\|_2^2
    \le 2C\tau\bigl(\|\mu_{\hat{j}}\|_2 + \|\mu_{j_0}\|_2\bigr)
    \sqrt{\frac{(\log M \wedge d) + \log(1/\delta)}{n_0}}.
\]
Let $\kappa := C\tau\sqrt{[(\log M \wedge d) + \log(1/\delta)]/n_0}$.
Rearranging yields
\[
    \|\mu_{\hat{j}}\|_2^2 - 2\kappa\|\mu_{\hat{j}}\|_2 - \|\mu_{j_0}\|_2^2 \le 0,
\]
so by the quadratic formula
$\|\mu_{\hat{j}}\|_2 \le \kappa + \sqrt{\kappa^2 + \|\mu_{j_0}\|_2^2}$,
and hence
\[
    \|\mu_{\hat{j}}\|_2^2 \le 2\|\mu_{j_0}\|_2^2 + 4\kappa^2
    = \underbrace{2\|\mu_{j_0}\|_2^2
    + \frac{4C^2\tau^2(\log M \wedge d)}{n_0}}_{=: A}
    + \underbrace{\frac{4C^2\tau^2}{n_0}}_{=: B}\log(1/\delta).
\]
This inequality also holds trivially when $\hat{j} = j_0$. Thus,
$P\!\left(\|\mu_{\hat{j}}\|_2^2 > A + Bt\right) \le e^{-t}$ for all $t \ge 0$.

\paragraph{Computing the expectation.}
Using the layer-cake formula,
\[
    \mathbb{E}\!\left[\|\mu_{\hat{j}}\|_2^2\right]
    = \int_0^\infty P\!\left(\|\mu_{\hat{j}}\|_2^2 > u\right)du
    \le A + B\int_0^\infty e^{-t}\,dt
    = A + B,
\]
which gives
\[
\begin{aligned}
     \mathbb{E}\!\left[\|\beta_{\hat{j}} - \theta_0\|_2^2\right]
    &\le 2\|\mu_{j_0}\|_2^2 + \frac{4C^2\tau^2\bigl((\log M \wedge d) + 1\bigr)}{n_0}\\
    &
    = 2\!\left(\min_{j}\|\beta_j - \theta_0\|_2^2\right)
    + \frac{4C^2\tau^2\bigl((\log M \wedge d) + 1\bigr)}{n_0}
\end{aligned}
\]
\end{proof}

\subsection{Proof of Corollary \ref{cor:adaptation-cost-ub-general-m}}

\begin{proof}
 Letting $\cM = 2^{[m]}$, we note $\log M\asymp m$ and 
    \[
   \max_{\{\theta_k^\star\}_{k = 0}^m \in \Theta(\bh, \|\cdot\|_2) } \min_{j = 1, \dots, M} \left(\textstyle \frac{ 1}{N_j^2} \Big \|\sum_{k \in S_j} n_k  (\theta_k^\star - \theta_0^\star) \Big\|_2^2  + \frac{d\tau^2}{N_j} \right) \lesssim \fR (\bh, \bn, \tau).
    \] Thus, we obtain the upper bound
    \[
    \begin{aligned}
        \sup_{P \in \cP(\bh, \bn, \tau)} \frac{\Ex\left[\|\widehat \theta_0^{\mathrm{MS}} - \theta_0^\star\|_2^2\right]}{\fR (\bh, \bn, \tau)} \lesssim 1 + \frac{\tau^2 \frac{m \wedge d}{n_0}}{\fR (\bh, \bn, \tau)}
    \end{aligned}
    \] Taking supremum over $\bh \in [0, 1]^m$ on both sides,
    \[
    \begin{aligned}
        \fC(\widehat \theta_0^{\mathrm{MS}}, [0, 1]^m) & \lesssim 1+ \frac{{(m \wedge d)}/{n_0}}{{ d}/{(mn)}} = 1 + \left(1 \wedge \frac{m}{d}\right)\frac{mn}{n_0},
    \end{aligned}
    \] which leads to the desired upper bound.
\end{proof}
\section{Supplementary material for Section \ref{sec:adaptation-general}}

\subsection{A proxy for the sub-Gaussian scale}
\label{supp:tau-proxy}

In practice, the sub-Gaussian parameter $\tau$ in equation \eqref{eq:sub-gaussian} 
can be estimated from data. We describe natural data-driven proxies $\widehat{\tau}$ 
for the three examples in Section \ref{subsec:domain-specific}.

\paragraph{Normal mean estimation:}
In Example~\ref{example:normal-mean}, a natural data-driven proxy for $\tau^2$ is
\[
\widehat{\tau}^2 := \max_{k = 0, \dots, m}\lambda_{\max}(\widehat{\Sigma}_k),
\]
where $\widehat{\Sigma}_k := n_k^{-1} \sum_{i=1}^{n_k} (Z_{k,i} - \widetilde{\theta}_k)
(Z_{k,i} - \widetilde{\theta}_k)^\top$ is the sample covariance matrix.

\paragraph{Linear regression:}  In Example~\ref{example:linear-regression}, the sub-Gaussian parameter satisfies
\[
\tau^2 = \max_{k = 0, \dots, m} \frac{\sigma_k^2}{\lambda_{\min}(\widehat{\Sigma}_k)},
\]
where $\widehat{\Sigma}_k := n_k^{-1} \sum_{i=1}^{n_k} X_{k,i} X_{k,i}^\top$ 
is the scaled Gram matrix. A natural data-driven proxy is
\[
\widehat{\tau}^2 := \max_{k = 0, \dots, m}\frac{\widehat{\sigma}_k^2}{\lambda_{\min}(\widehat{\Sigma}_k)},
\]
where $\widehat{\sigma}_k^2 := (n_k - d)^{-1} \sum_{i=1}^{n_k} 
(Y_{k,i} - X_{k,i}^\top \widetilde{\theta}_k)^2$ is the residual variance estimate.

\paragraph{M-estimation:} In Example~\ref{example:m-estimation}, a natural data-driven proxy is
\[
\widehat{\tau}^2 := \max_{k =0, \dots, m}\lambda_{\max}\!\left(\widehat{H}_k^{-1} \widehat{V}_k 
\widehat{H}_k^{-1}\right),
\]
where the empirical Hessian and score covariance are given by
\[
\widehat{H}_k := \frac{1}{n_k} \sum_{i=1}^{n_k} 
\nabla_\theta^2 \ell(Z_{k,i}, \widetilde{\theta}_k), \qquad
\widehat{V}_k := \frac{1}{n_k} \sum_{i=1}^{n_k} 
\nabla_\theta \ell(Z_{k,i}, \widetilde{\theta}_k)\, 
\nabla_\theta \ell(Z_{k,i}, \widetilde{\theta}_k)^\top.
\]

\subsection{Proof of the Theorem \ref{thm:intersection-oracle-max}}
\label{proof:thm:intersection-oracle-max}

\begin{proof}
\textbf{Step 0: Oracle property of $r^\star$.}
By Lemma \ref{lem:oracle-index-max} we have
$R_{r^\star}=\min_{0\le r\le m} R_r$.

\textbf{Step 1: Non-emptiness of $\mathcal C_{r^\star}$.}
Let $\theta^\sharp := \Ex[\widehat\theta^{(r^\star)}]$.
We show that on $\Omega_0$, $\theta^\sharp \in \bigcap_{r=0}^{r^\star}\mathcal B_r$.

Fix any $r<r^\star$. Using triangle inequality and the definition of $\Omega_0$,
\[
\|\theta^\sharp-\widehat\theta^{(r)}\|_2
\le
\|\theta^\sharp-\theta_0^\star\|_2
+
\|\widehat\theta^{(r)}-\Ex[\widehat\theta^{(r)}]\|_2
\le
h_{r^\star} + \rho_r.
\]
By the definition of $r^\star$, we have $h_{r^\star}\le \sigma_{r^\star-1}$.
Since $r\le r^\star-1$ and $(\sigma_r)$ is nonincreasing, $\sigma_{r^\star-1}\le \sigma_r$,
hence $h_{r^\star}\le \sigma_r$. Therefore, for $C_0$ large enough in~\eqref{eq:rad-def-section},
\[
\|\theta^\sharp-\widehat\theta^{(r)}\|_2 \le \sigma_r+\rho_r \le 2\rho_r,
\]
so $\theta^\sharp\in\mathcal B_r$.

For $r=r^\star$, on $\Omega_0$ we directly have
\[
\|\theta^\sharp-\widehat\theta^{(r^\star)}\|_2
=
\|\Ex[\widehat\theta^{(r^\star)}]-\widehat\theta^{(r^\star)}\|_2
\le \rho_{r^\star},
\]
so $\theta^\sharp\in\mathcal B_{r^\star}$ as well.

Thus on $\Omega_0$,
\[
\theta^\sharp \in \bigcap_{r=0}^{r^\star}\mathcal B_r = \mathcal C_{r^\star},
\qquad\text{so}\qquad
\mathcal C_{r^\star}\neq\varnothing.
\]
By definition of $\hat t$, it follows that $\hat t\ge r^\star$ on $\Omega_0$.

\textbf{Step 2: Risk bound on $\Omega_0$.}
Since $\hat t\ge r^\star$ and $\hat\theta\in\mathcal C_{\hat t}$, we have $\hat\theta\in\mathcal B_{r^\star}$,
hence
\[
\|\hat\theta-\widehat\theta^{(r^\star)}\|_2 \le 2\rho_{r^\star}.
\]
Also, by triangle inequality and the definition of $\Omega_0$,
\[
\|\widehat\theta^{(r^\star)}-\theta_0^\star\|_2
\le
\|\widehat\theta^{(r^\star)}-\Ex[\widehat\theta^{(r^\star)}]\|_2
+
\|\Ex[\widehat\theta^{(r^\star)}]-\theta_0^\star\|_2
\le
\rho_{r^\star}+h_{r^\star}.
\]
Therefore,
\[
\|\hat\theta-\theta_0^\star\|_2
\le
\|\hat\theta-\widehat\theta^{(r^\star)}\|_2
+
\|\widehat\theta^{(r^\star)}-\theta_0^\star\|_2
\le
h_{r^\star}+3\rho_{r^\star}.
\]
Squaring and absorbing constants gives
\[
\|\hat\theta-\theta_0^\star\|_2^2
\;\lesssim\;
h_{r^\star}^2 + \rho_{r^\star}^2
\;\lesssim\;
h_{r^\star}^2 + \sigma_{r^\star}^2\log\frac{m+1}{\delta}.
\]
Finally, since $h_{r^\star}^2 \vee \sigma^2_{r^*} \le R_{r^\star}$, we obtain on $\Omega_0$,
\[
\|\hat\theta-\theta_0^\star\|_2^2
\;\lesssim\;
R_{r^\star} \log\frac{m+1}{\delta}.
\]

This proves~\eqref{eq:intersection-oracle-final-max}.
\end{proof}

\begin{lemma}[Oracle index for the max-risk criterion]
\label{lem:oracle-index-max}
Assume $0=h_0\le h_1\le\cdots\le h_m$ and $\sigma_0\ge \sigma_1\ge\cdots\ge \sigma_m$.
Adopt the convention $\sigma_{-1}:=+\infty$.
Define
\[
r^\star \in \arg\max\Bigl\{0\le r\le m:\ \sigma_{r-1}\ge h_r\Bigr\},
\qquad
R_r := h_r^2 \vee \sigma_r^2,\ \ r=0,\dots,m.
\]
Then
\[
R_{r^\star} \;=\; \min_{0\le r\le m} R_r.
\]
\end{lemma}

\begin{proof}
First note that the feasibility set $\{r:\sigma_{r-1}\ge h_r\}$ is nonempty since
$\sigma_{-1}=+\infty\ge h_0=0$, so $r^\star$ is well-defined.

\emph{Case 1: $r>r^\star$.}
By maximality of $r^\star$, $r$ is infeasible, i.e. $\sigma_{r-1}<h_r$.
Using monotonicity $h_{r-1}\le h_r$ and $\sigma_r\le \sigma_{r-1}$, we get
\[
R_{r-1} = h_{r-1}^2\vee \sigma_{r-1}^2 \;\le\; h_r^2
\;\le\; h_r^2\vee \sigma_r^2 = R_r.
\]
Thus $R_{r-1}\le R_r$ for all $r>r^\star$, so $(R_r)_{r\ge r^\star}$ is nondecreasing.

\emph{Case 2: $1\le r\le r^\star$.}
Since $r\le r^\star$ and $\sigma_{r*-1}\ge h_{r*}$, so $\sigma_{r-1}\ge \sigma_{r*-1}\ge h_{r*} \ge h_r$ i.e $\sigma_{r-1}\ge h_r$.
Therefore
\[
R_{r-1} = h_{r-1}^2\vee \sigma_{r-1}^2 = \sigma_{r-1}^2
\;\ge\; h_r^2\vee \sigma_r^2 = R_r,
\]
where the last inequality uses $\sigma_{r-1} \ge h_r$ and $\sigma_{r-1}\ge \sigma_r$ .
Hence $R_{r-1}\ge R_r$ for all $1\le r\le r^\star$, so $(R_r)_{r\le r^\star}$ is nonincreasing.

Combining the two cases, $R_r$ decreases up to $r^\star$ and increases thereafter, so it
achieves its minimum at $r^\star$, i.e. $R_{r^\star}=\min_{0\le r\le m}R_r$.
\end{proof}

\subsection{Proof of the Theorem \ref{thm:two-cluster-mse}}
\label{proof:thm:two-cluster-mse}

\begin{proof}{\bf of Theorem~\ref{thm:two-cluster-mse}:}
Write $\sigma^2:=\frac{\tau^2}{n}$.
As throughout this section, all losses and risks are understood up to universal
constant factors, the constant loss of effective sample size caused by sample
splitting is absorbed into the universal constants.

For any subset \(S\subset[m]\), define the
pooled estimator
\[
\hat\theta_S
:=
\frac{1}{1+|S|}
\left(
\widetilde\theta_0^{(1)}+\sum_{i\in S}\widetilde\theta_i^{(3)}
\right),
\qquad
\hat\theta_\varnothing:=\widetilde\theta_0^{(1)}.
\]
If \(\hat G_1,\hat G_2\) are the two clusters returned by
Algorithm~\ref{alg:sample-split-clustering}, then the three candidates used by
Algorithm~\ref{alg:two-cluster-adaptive-short} are
\[
\hat\theta^{(0)}:=\hat\theta_\varnothing,
\qquad
\hat\theta^{(1)}:=\hat\theta_{\hat G_1},
\qquad
\hat\theta^{(2)}:=\hat\theta_{\hat G_2}.
\]
The final estimator is selected by validation:
\[
\hat a
\in
\arg\min_{a\in\{0,1,2\}}
\bigl\|\hat\theta^{(a)}-\widetilde\theta_0^{(2)}\bigr\|_2^2,
\qquad
\hat\theta_{\mathrm{ad}}:=\hat\theta^{(\hat a)}.
\]

Let $R_a:=\|\hat\theta^{(a)}-\theta_0^\star\|_2^2$ for $a\in\{0,1,2\}$.
By Lemma~\ref{lem:three-way-validation-selector},
\[
\mathbb E\bigl[\|\hat\theta_{\mathrm{ad}}-\theta_0^\star\|_2^2\bigr]
\le
C\left(
\mathbb E\Bigl[\min_{a\in\{0,1,2\}}R_a\Bigr]
+\sigma^2
\right).
\]
It therefore suffices to show that
\[
\mathbb E\Bigl[\min_{a\in\{0,1,2\}}R_a\Bigr]
\le
C\left(M_{\mathrm{cl}}+\sigma^2\widetilde\Delta_{m,d}\right).
\]

We split into two regimes.

\medskip
\noindent
\textbf{Case 1: large separation.}
Assume first that
\[
\Delta^2 \ge A\sigma^2\widetilde\Delta_{m,d}
\]
for a sufficiently large universal constant \(A\). Set $\delta:=(md)^{-10}$.
Since \(m_{\min}\ge c m\), we have
\[
\Gamma_{m,m_{\min},d}(\delta)
:=
\frac{d+\log(1/\delta)}{m_{\min}}
+
\frac{\sqrt{m\bigl(d+\log(1/\delta)\bigr)}}{m_{\min}}
+
\log\frac{m}{\delta}
\;\le\;
C\,\widetilde\Delta_{m,d}.
\]
Indeed,
\[
\log(1/\delta)=10\log(md),
\qquad
\log\frac{m}{\delta}=\log m+10\log(md)\le 11\log(md),
\]
and \(m_{\min}\ge c m\). Therefore, after enlarging \(A\) if necessary,
Proposition~\ref{prop:sample-split-recovery} applies and yields
\[
\mathbb P(\mathcal E_{\mathrm{rec}}^c)\le (md)^{-10},
\qquad
\mathcal E_{\mathrm{rec}}
:=
\bigl\{\{\hat G_1,\hat G_2\}=\{G_1,G_2\}\bigr\}.
\]

On \(\mathcal E_{\mathrm{rec}}\), the three candidates
\(\hat\theta^{(0)},\hat\theta^{(1)},\hat\theta^{(2)}\) coincide, up to relabeling,
with the three oracle actions
\[
\hat\theta_\varnothing,
\qquad
\hat\theta_{G_1},
\qquad
\hat\theta_{G_2}.
\]
Therefore,
\[
\min_{a\in\{0,1,2\}}R_a
\le
\mathbf 1_{\mathcal E_{\mathrm{rec}}}
\min\!\left\{
\|\hat\theta_\varnothing-\theta_0^\star\|_2^2,\,
\|\hat\theta_{G_1}-\theta_0^\star\|_2^2,\,
\|\hat\theta_{G_2}-\theta_0^\star\|_2^2
\right\}
+
\mathbf 1_{\mathcal E_{\mathrm{rec}}^c}
\|\hat\theta_\varnothing-\theta_0^\star\|_2^2.
\]
Hence
\[
\mathbb E\Bigl[\min_{a\in\{0,1,2\}}R_a\Bigr]
\le
\mathbb E\!\left[
\min\!\left\{
\|\hat\theta_\varnothing-\theta_0^\star\|_2^2,\,
\|\hat\theta_{G_1}-\theta_0^\star\|_2^2,\,
\|\hat\theta_{G_2}-\theta_0^\star\|_2^2
\right\}
\right]
+
\mathbb E\!\left[
\mathbf 1_{\mathcal E_{\mathrm{rec}}^c}
\|\hat\theta_\varnothing-\theta_0^\star\|_2^2
\right].
\]

For the first term, using
\(\mathbb E[\min_j X_j]\le \min_j \mathbb E[X_j]\),
\[
\mathbb E\!\left[
\min\!\left\{
\|\hat\theta_\varnothing-\theta_0^\star\|_2^2,\,
\|\hat\theta_{G_1}-\theta_0^\star\|_2^2,\,
\|\hat\theta_{G_2}-\theta_0^\star\|_2^2
\right\}
\right]
\le
C\,M_{\mathrm{cl}}.
\]

For the second term, since
\(\hat\theta_\varnothing-\theta_0=\widetilde\theta_0^{(1)}-\theta_0\) is
sub-Gaussian with scale \(\sigma\), its squared norm has second moment bounded
by \(C\sigma^4 d^2\). Thus by Cauchy--Schwarz,
\[
\mathbb E\!\left[
\mathbf 1_{\mathcal E_{\mathrm{rec}}^c}
\|\hat\theta_\varnothing-\theta_0^\star\|_2^2
\right]
\le
\bigl(\mathbb E\|\hat\theta_\varnothing-\theta_0^\star\|_2^4\bigr)^{1/2}
\mathbb P(\mathcal E_{\mathrm{rec}}^c)^{1/2}
\le
C\sigma^2 d\,(md)^{-5}
\le
C\sigma^2.
\]
Therefore,
\[
\mathbb E\Bigl[\min_{a\in\{0,1,2\}}R_a\Bigr]
\le
C\left(M_{\mathrm{cl}}+\sigma^2\right)
\le
C\left(M_{\mathrm{cl}}+\sigma^2\widetilde\Delta_{m,d}\right),
\]
since \(\widetilde\Delta_{m,d}\ge 1\) for \(m,d\ge 1\).

\medskip
\noindent
\textbf{Case 2: small separation.}
Assume now that $\Delta^2 < A\sigma^2\widetilde\Delta_{m,d}$.
Let $\hat r\in\arg\max_{r\in\{1,2\}} |\hat G_r|$ and $\hat G_{\max}:=\hat G_{\hat r}$.
Then \(|\hat G_{\max}|\ge m/2\). Since
\(\hat\theta^{(\hat r)}=\hat\theta_{\hat G_{\max}}\), Lemma~\ref{lem:arbitrary-subset-risk}
gives
\[
\mathbb E[R_{\hat r}]
=
\mathbb E\bigl[\|\hat\theta_{\hat G_{\max}}-\theta_0^\star\|_2^2\bigr]
\le
C\left(M_{\mathrm{cl}}+\Delta^2\right)
\le
C\left(M_{\mathrm{cl}}+\sigma^2\widetilde\Delta_{m,d}\right).
\]
Since $\min_{a\in\{0,1,2\}}R_a \le R_{\hat r}$,
we conclude that
\[
\mathbb E\Bigl[\min_{a\in\{0,1,2\}}R_a\Bigr]
\le
C\left(M_{\mathrm{cl}}+\sigma^2\widetilde\Delta_{m,d}\right).
\]

Combining the two cases yields
\[
\mathbb E\Bigl[\min_{a\in\{0,1,2\}}R_a\Bigr]
\le
C\left(M_{\mathrm{cl}}+\sigma^2\widetilde\Delta_{m,d}\right).
\]
Substituting this into the selector inequality from
Lemma~\ref{lem:three-way-validation-selector} gives
\[
\mathbb E\bigl[\|\hat\theta_{\mathrm{ad}}-\theta_0^\star\|_2^2\bigr]
\le
C\left(
M_{\mathrm{cl}}+\sigma^2\widetilde\Delta_{m,d}+\sigma^2
\right).
\]
Recalling that \(\sigma^2=\tau^2/n\) proves the theorem.
\end{proof}

\begin{lemma}[Three-way validation selector]
\label{lem:three-way-validation-selector}
There exists a universal constant \(C>0\) such that
\[
\mathbb E\bigl[\|\hat\theta_{\mathrm{ad}}-\theta_0^\star\|_2^2\bigr]
\;\le\;
C\left(
\mathbb E\Bigl[\min_{a\in\{0,1,2\}}
\|\hat\theta^{(a)}-\theta_0^\star\|_2^2\Bigr]
+\sigma^2
\right).
\]
\end{lemma}

\begin{proof}
Condition on the candidate estimators
\(\hat\theta^{(0)},\hat\theta^{(1)},\hat\theta^{(2)}\), and write
\[
R_a:=\|\hat\theta^{(a)}-\theta_0^\star\|_2^2,
\qquad
W:=\widetilde\theta_0^{(2)}-\theta_0^\star.
\]
Let
\[
a_\star\in\arg\min_{a\in\{0,1,2\}} R_a,
\qquad
v_a:=\hat\theta^{(a)}-\hat\theta^{(a_\star)}.
\]
Since \(\hat a\) minimizes
\(\|\hat\theta^{(a)}-(\theta_0^\star+W)\|_2^2\), we have
\[
\|\hat\theta^{(\hat a)}-(\theta_0^\star+W)\|_2^2
\le
\|\hat\theta^{(a_\star)}-(\theta_0^\star+W)\|_2^2.
\]
Expanding both sides gives
\[
R_{\hat a}-R_{a_\star}
\le
2\langle W,v_{\hat a}\rangle.
\]
Define
\[
Z
:=
\max_{a\in\{0,1,2\}}
\left\{
2\langle W,v_a\rangle-\frac14\|v_a\|_2^2
\right\}.
\]
Then
\[
2\langle W,v_{\hat a}\rangle
\le
Z+\frac14\|v_{\hat a}\|_2^2.
\]
Also,
\[
\|v_{\hat a}\|_2^2
=
\|\hat\theta^{(\hat a)}-\hat\theta^{(a_\star)}\|_2^2
\le
2R_{\hat a}+2R_{a_\star}.
\]
Therefore
\[
R_{\hat a}-R_{a_\star}
\le
Z+\frac12(R_{\hat a}+R_{a_\star}),
\]
and hence
\[
R_{\hat a}\le 3R_{a_\star}+2Z.
\]

It remains to bound \(Z\). Since \(W\) is mean-zero sub-Gaussian with scale
\(\sigma\), there exists a universal constant \(c>0\) such that for every
deterministic \(v\in\mathbb R^d\),
\[
\mathbb E\!\left[
\exp\!\left(
\frac{c}{\sigma^2}
\left(
2\langle W,v\rangle-\frac14\|v\|_2^2
\right)
\right)
\right]
\le 1.
\]
Conditioning on the candidates and applying this bound with \(v=v_a\), followed
by the union bound over the three values of \(a\), yields
\[
\mathbb E[Z \mid \hat\theta^{(0)},\hat\theta^{(1)},\hat\theta^{(2)}]
\le C\sigma^2.
\]
Thus
\[
\mathbb E[R_{\hat a}\mid \hat\theta^{(0)},\hat\theta^{(1)},\hat\theta^{(2)}]
\le
3\min_{a\in\{0,1,2\}}R_a + C\sigma^2.
\]
Taking expectations proves the claim.
\end{proof}

\begin{lemma}[Risk of pooling an arbitrary subset]
\label{lem:arbitrary-subset-risk}
Let \(S\subset[m]\), and write
\[
a:=|S\cap G_1|,
\qquad
b:=|S\cap G_2|,
\qquad
s:=|S|=a+b.
\]
Then
\[
\mathbb E\bigl[\|\hat\theta_S-\theta_0^\star\|_2^2\bigr]
\;\le\;
C\left(
\frac{d\sigma^2}{1+s}
+
\min\{b_1^2,b_2^2\}
+
\Delta^2
\right).
\]
In particular, if \(s\ge m/2\) and \(m_{\min}\ge c m\), then
\[
\mathbb E\bigl[\|\hat\theta_S-\theta_0^\star\|_2^2\bigr]
\;\le\;
C\left(M_{\mathrm{cl}}+\Delta^2\right).
\]
\end{lemma}

\begin{proof}
Write
\[
\textstyle \hat\theta_S-\theta_0
=
\frac{1}{1+s}
\left(
(\widetilde\theta_0^{(1)}-\theta_0^\star)
+
\sum_{i\in S}(\widetilde\theta_i^{(3)}-\theta_i^\star)
\right)
+
\frac{1}{1+s}
\sum_{i\in S}(\theta_i-\theta_0^\star).
\]
Since the estimation noises are independent, mean-zero, and sub-Gaussian with
covariance bounded by \(C\sigma^2 I_d\), the variance term is bounded by
\[
\textstyle \mathbb E\left\|
\frac{1}{1+s}
\left(
(\widetilde\theta_0^{(1)}-\theta_0^\star)
+
\sum_{i\in S}(\widetilde\theta_i^{(3)}-\theta_i^\star)
\right)
\right\|_2^2
\le
C\,\frac{d\sigma^2}{1+s}.
\]
The bias term is
\[
\textstyle \frac{1}{1+s}\sum_{i\in S}(\theta_i-\theta_0^\star)
=
\frac{a(\mu_1-\theta_0^\star)+b(\mu_2-\theta_0^\star)}{1+s}.
\]

Assume without loss of generality that \(b_1\le b_2\). Then
\[
a(\mu_1-\theta_0^\star)+b(\mu_2-\theta_0^\star)
=
(a+b)(\mu_1-\theta_0^\star)+b(\mu_2-\mu_1),
\]
so
\[
\textstyle \left\|
\frac{a(\mu_1-\theta_0^\star)+b(\mu_2-\theta_0^\star)}{1+s}
\right\|_2
\le
\frac{s}{1+s}b_1+\frac{b}{1+s}\Delta
\le b_1+\Delta.
\]
Hence
\[
\textstyle \left\|
\frac{a(\mu_1-\theta_0^\star)+b(\mu_2-\theta_0^\star)}{1+s}
\right\|_2^2
\le
2b_1^2+2\Delta^2
=
2\min\{b_1^2,b_2^2\}+2\Delta^2.
\]
Combining the variance and bias bounds proves the first display.

For the second display, if \(s\ge m/2\) and \(m_{\min}\ge c m\), then
\[
\textstyle \frac{d\sigma^2}{1+s}\le C\frac{d\sigma^2}{m}
\le
C\min_{r\in\{1,2\}}
\frac{d\sigma^2}{1+m_r},
\]
since \(m_r\ge m_{\min}\ge c m\). Therefore
\[
\textstyle \mathbb E\bigl[\|\hat\theta_S-\theta_0^\star\|_2^2\bigr]
\le
C\left(
\min_{r\in\{1,2\}}
\left\{
\frac{d\sigma^2}{1+m_r}+b_r^2
\right\}
+\Delta^2
\right)
\le
C\left(M_{\mathrm{cl}}+\Delta^2\right),
\]
as claimed.
\end{proof}

\begin{proposition}[Sufficient separation for consistent clustering]
\label{prop:sample-split-recovery}
There exists a universal constant \(C>0\) such that the following holds for
every \(\delta\in(0,1/2)\). If
\[
\textstyle \|\mu_1-\mu_2\|_2^2
\;\ge\;
C\,\frac{\tau^2}{n}\,\Gamma_{m,m_{\min},d}(\delta),
\]
where
\[
\textstyle \Gamma_{m,m_{\min},d}(\delta)
:=
\frac{d+\log(1/\delta)}{m_{\min}}
+
\frac{\sqrt{m\bigl(d+\log(1/\delta)\bigr)}}{m_{\min}}
+
\log\frac{m}{\delta},
\]
then Algorithm~\ref{alg:sample-split-clustering} recovers the latent partition
exactly with probability at least \(1-\delta\), that is,
\[
\{\hat G_1,\hat G_2\}=\{G_1,G_2\}
\qquad\text{with probability at least }1-\delta.
\]
\end{proposition}
\begin{proof}
This is a standard exact-recovery argument for a two-component sub-Gaussian
mixture; compare, for example, \citet{fei2018hidden} and
\citet{jiang2020recovery}. We sketch the short adaptation to the present
sample-split clustering rule.

Let
\[
\sigma^2:=\frac{\tau^2}{n},
\qquad
u_\star:=\frac{\mu_1-\mu_2}{\|\mu_1-\mu_2\|_2}.
\]

\paragraph{Step 1: recover the cluster direction from the first split.}
Let
\[
\bar \theta^{(1)}:=\frac1m\sum_{i=1}^m \widetilde\theta_i^{(1)},
\qquad
\hat\Sigma^{(1)}
:=
\frac1m\sum_{i=1}^m
\bigl(\widetilde\theta_i^{(1)}-\bar \theta^{(1)}\bigr)\bigl(\widetilde\theta_i^{(1)}-\bar \theta^{(1)}\bigr)^\top.
\]
Under the two-cluster model, the centered source means lie on the one-dimensional
signal direction \(u_\star\), and the corresponding population covariance has the
form
\[
\Sigma_\star
=
\rho \Delta^2 u_\star u_\star^\top + \Sigma_\eta,
\]
where \(\Sigma_\eta\) is isotropic of order \(\sigma^2 I_d\), and
\[
\rho=\frac{m_1m_2}{m^2}\asymp \frac{m_{\min}}{m}
\]
under the balanced-cluster assumption \(m_{\min}\ge c m\). Therefore, the
population eigengap is of order
\[
\lambda_1(\Sigma_\star)-\lambda_2(\Sigma_\star)
\asymp
\frac{m_{\min}}{m}\Delta^2.
\]

Standard operator-norm concentration for sub-Gaussian covariance matrices shows
that, with probability at least \(1-\delta/3\),
\[
\|\hat\Sigma^{(1)}-\Sigma_\star\|_{\op}
\le
C\left[
\frac{\sigma\Delta\sqrt{m_{\min}}}{m}\sqrt{d+\log(1/\delta)}
+
\sigma^2
\left(
\sqrt{\frac{d+\log(1/\delta)}{m}}
+
\frac{d+\log(1/\delta)}{m}
\right)
\right].
\]
The assumed lower bound on \(\Delta^2\) implies that the right-hand side is at
most a sufficiently small constant multiple of the eigengap. Hence
Davis--Kahan yields
\[
\sin\angle(\hat u,u_\star)\le \frac18
\]
with probability at least \(1-\delta/3\), where \(\hat u\) is a leading
eigenvector of \(\hat\Sigma^{(1)}\).

\paragraph{Step 2: classify the second split on the recovered axis.}
Conditional on the first split, define
\[
\nu_r:=\hat u^\top(\mu_r-\bar \theta^{(1)}),
\qquad r=1,2.
\]
On the event above,
\[
|\nu_1-\nu_2|
=
|\hat u^\top(\mu_1-\mu_2)|
\ge c_0\Delta
\]
for a universal constant \(c_0>0\).

Now set
\[
t:=A\sigma\sqrt{\log\frac{6m}{\delta}},
\]
where \(A>0\) is a sufficiently large universal constant. Because the second
split is independent of the first split, a union bound for sub-Gaussian tails
gives
\[
\max_{1\le i\le m}|\hat u^\top\eta_i^{(2)}|\le t
\]
with probability at least \(1-\delta/6\). Moreover, the one-dimensional
\(2\)-means centers computed from the first split lie within distance \(t\) of
the projected population centers with probability at least \(1-\delta/6\), by
the same one-dimensional separation argument used in standard two-cluster
mixture recovery.

Hence, on an event of probability at least \(1-\delta\), every projected
second-split point lies within distance \(2t\) of its correct center, whereas
its distance to the incorrect center is at least
\[
|\nu_1-\nu_2|-2t\ge c_0\Delta-2t.
\]
If
\[
\Delta^2\ge C\sigma^2\log\frac{m}{\delta},
\]
then \(c_0\Delta\ge 6t\), so each second-split point is strictly closer to the
correct center than to the incorrect one. Therefore the classification step is
exact, and
\[
\{\hat G_1,\hat G_2\}=\{G_1,G_2\}.
\]

\paragraph{Step 3: combine the two requirements.}
The spectral step requires a lower bound of order
\[
\sigma^2
\left(
\frac{d+\log(1/\delta)}{m_{\min}}
+
\frac{\sqrt{m\bigl(d+\log(1/\delta)\bigr)}}{m_{\min}}
\right),
\]
while the one-dimensional classification step requires
\[
\sigma^2\log\frac{m}{\delta}.
\]
Both are implied by the assumed bound
\[
\Delta^2
\ge
C\,\sigma^2\,\Gamma_{m,m_{\min},d}(\delta).
\]
This proves the proposition.
\end{proof}

\subsection{Proof of the Lemma \ref{lemma:minimizer-iff-bias-1}}
\label{proof:lemma:minimizer-iff-bias}

\begin{proof}
Let $h^\star = h_{S^\star}$ and $f^\star = f(S^\star) = (h^\star)^2 \vee \tfrac{d\tau^2}{N_{S^\star}}$.

\textbf{($\Rightarrow$) Necessity.} If $h_k \le h^\star$ but $k \notin S^\star$, then
$S' = S^\star \cup \{k\}$ satisfies $h_{S'} = h^\star$ and $N_{S'} > N_{S^\star}$,
so $f(S') < f^\star$, contradicting minimality. Hence $S^\star = \{k : h_k \le h^\star\}$.
Optimality over sets with $h_S < h^\star$ requires $f(\{k: h_k < h^\star\}) \ge f^\star$,
yielding $(h^\star)^2 \le d\tau^2/{\sum_{k: h_k < h^\star} n_k}$.
For any $k \notin S^\star$, minimality against $S^\star \cup \{k\}$ requires
$h_k^2 \vee \tfrac{d\tau^2}{N_{S^\star}+n_k} \ge f^\star$, and since the variance
term decreases, we get $h_k > \max\{h^\star, \tau\sqrt{d/N_{S^\star}}\}$.

\textbf{($\Leftarrow$) Sufficiency.} For any competing $S$: if $h_S \le h^\star$ then
$S \subseteq S^\star$, so $N_S \le N_{S^\star}$ and $f(S) \ge \tfrac{d\tau^2}{N_S} \ge
\tfrac{d\tau^2}{N_{S^\star}}$; the condition on $h^\star$ ensures
$\tfrac{d\tau^2}{N_S} \ge (h^\star)^2$, so $f(S) \ge f^\star$.
If $h_S > h^\star$, some $k \notin S^\star$ lies in $S$, giving
$f(S) \ge h_k^2 > (h^\star)^2 \vee \tfrac{d\tau^2}{N_{S^\star}} = f^\star$.
Hence $S^\star$ is a minimizer.
\end{proof}

\subsection{Proof of the Theorem \ref{thm:adaptation-test-then-combine}}

\begin{proof} 
The proof proceeds in two steps. We first show that $\widehat{S} = S$ with high probability and then use this to bound the mean squared error of $\widehat{\theta}_0$.

\medskip
\noindent\textbf{Step 1: Correct identification of $S$.}

\medskip
\noindent\textit{Concentration of $\widetilde{\theta}_k - \widetilde{\theta}_0$.}
Define $\delta_k := \theta_k^\star - \theta_0^\star$. By combining the sub-Gaussian concentration
of $\widetilde{\theta}_k - \theta_k^\star$ and $\widetilde{\theta}_0 - \theta_0^\star$, we obtain
\[
P\!\left(a^\top\bigl\{\widetilde{\theta}_k - \widetilde{\theta}_0 - \delta_k\bigr\} > t\right)
\;\le\; \exp\!\left(-\frac{(n_0 \wedge n_k)\,t^2}{4\tau^2}\right).
\]
Applying Lemma~\ref{lemma:tech-1-t2c}, this one-dimensional bound lifts to a norm bound:
for some universal constant $c > 0$ and all $t > c\tau\sqrt{d/(n_0 \wedge n_k)}$,
\[
P\!\left(\bigl\|\widetilde{\theta}_k - \widetilde{\theta}_0 - \delta_k\bigr\|_2 > t\right)
\;\le\; \exp\!\left(-\frac{(n_0 \wedge n_k)\,t^2}{32\tau^2}\right).
\]
Setting $t_k := \tau\sqrt{32\alpha d\log(n_0 \vee n_k)/(n_0 \wedge n_k)}$ for $\alpha \ge 1$,
the exponent equals $-\alpha d \log(n_0 \vee n_k)$, giving
\begin{equation}\label{eq:concentration-diff}
P\!\left(\bigl\|\widetilde{\theta}_k - \widetilde{\theta}_0 - \delta_k\bigr\|_2 > t_k\right)
\;\le\; \frac{1}{(n_0 \vee n_k)^{\alpha d}}.
\end{equation}

\medskip
\noindent\textit{Correct inclusion ($k \in S$).}
For $k \in S$, the bias satisfies $\|\delta_k\|_2 \le h_k \le \tau\sqrt{d/n_0}$.
By the triangle inequality,
\[
\bigl\|\widetilde{\theta}_k - \widetilde{\theta}_0 - \delta_k\bigr\|_2
\;\ge\; \bigl\|\widetilde{\theta}_k - \widetilde{\theta}_0\bigr\|_2 - \|\delta_k\|_2,
\]
so if $\|\widetilde{\theta}_k - \widetilde{\theta}_0\|_2 > 9\tau\sqrt{d\alpha\log(n_0 \vee n_k)/(n_0 \wedge n_k)}$, then
\[
\textstyle \bigl\|\widetilde{\theta}_k - \widetilde{\theta}_0 - \delta_k\bigr\|_2
\;>\; 9\tau\sqrt{\frac{d\alpha\log(n_0 \vee n_k)}{n_0 \wedge n_k}}
     - \tau\sqrt{\frac{d}{n_0}}
\;>\; t_k,
\]
where the last inequality holds for $\alpha \ge 1$. Hence, by~\eqref{eq:concentration-diff},
\begin{equation}\label{eq:false-exclusion}
\textstyle P\!\left(\|\widetilde{\theta}_k - \widetilde{\theta}_0\|_2 > 9\tau\sqrt{\frac{d\alpha\log(n_0 \vee n_k)}{n_0 \wedge n_k}}\right)
\;\le\; \frac{1}{(n_0 \vee n_k)^{\alpha d}}.
\end{equation}

\medskip
\noindent\textit{Correct exclusion ($k \notin S$).}
For $k \notin S$, the separation condition gives
$\|\delta_k\|_2 > 17\tau\sqrt{\frac{d\alpha\log(n_0 \vee n_k)}{(n_0 \wedge n_k)}}$.
Analogously, if $\|\widetilde{\theta}_k - \widetilde{\theta}_0\|_2 \le 9\tau\sqrt{d\alpha\log(n_0 \vee n_k)/(n_0 \wedge n_k)}$,
then by the triangle inequality,
\[
\begin{aligned}
    \bigl\|\widetilde{\theta}_k - \widetilde{\theta}_0 - \delta_k\bigr\|_2
\;\ge\; \|\delta_k\|_2 - \|\widetilde{\theta}_k - \widetilde{\theta}_0\|_2
&\;>\; \textstyle 17\tau\sqrt{\frac{d\alpha\log(n_0\vee n_k)}{n_0\wedge n_k}}
     - 9\tau\sqrt{\frac{d\alpha\log(n_0\vee n_k)}{n_0\wedge n_k}}\\
     &
\;=\; \textstyle 8\tau\sqrt{\frac{d\alpha\log(n_0\vee n_k)}{n_0\wedge n_k}}
\;>\; t_k.
\end{aligned}
\]
Therefore, by~\eqref{eq:concentration-diff},
\begin{equation}\label{eq:false-inclusion}
\textstyle P\!\left(\|\widetilde{\theta}_k - \widetilde{\theta}_0\|_2 \le 9\tau\sqrt{\frac{d\alpha\log(n_0 \vee n_k)}{n_0 \wedge n_k}}\right)
\;\le\; \frac{1}{(n_0 \vee n_k)^{\alpha d}}.
\end{equation}

\medskip
\noindent\textit{Union bound.}
The event $\{\widehat{S} \neq S\}$ implies that at least one of the following occurs:
some $k \in S$ is falsely excluded, or some $k \notin S$ is falsely included.
Applying a union bound over all $k \in \{1, \dots, m\}$ and using
\eqref{eq:false-exclusion}--\eqref{eq:false-inclusion}, we obtain
\[
\textstyle P\!\left(\widehat{S} \neq S\right)
\;\le\; \sum_{k=1}^{m} \frac{1}{(n_0 \vee n_k)^{\alpha d}}
\;\le\; \frac{m}{n_0^{\alpha d}},
\]
which proves the first part of the theorem.

\medskip
\noindent\textbf{Step 2: Mean squared error bound.}
For any index set $\{0\} \subset J \subset \{0, \dots, m\}$, define the weighted average
$\widehat{\theta}_J := \sum_{k \in J} n_k \widetilde{\theta}_k \,/\, \sum_{k \in J} n_k$.
We decompose the expected squared error of $\widehat{\theta}_{\widehat{S}}$ by conditioning on
whether $\widehat{S} = S$:
\[
\mathbb{E}\!\left[\|\widehat{\theta}_{\widehat{S}} - \theta_0^\star\|_2^2\right]
\;\le\;
\underbrace{\mathbb{E}\!\left[\|\widehat{\theta}_S - \theta_0^\star\|_2^2\right]}_{\text{(I): oracle error}}
\;+\;
\underbrace{P\!\left(\widehat{S} \neq S\right)
\max_{\{0\} \subset J \subset \{0,\dots,m\}}
\mathbb{E}\!\left[\|\widehat{\theta}_J - \theta_0^\star\|_2^2\right]}_{\text{(II): misidentification penalty}}.
\]

\noindent\textit{Bounding (I).}
Since $S$ is the true set of integrable domains, the upper bound established in
Theorem~\ref{thm:oracle-minimax-rate} gives
\[
\mathbb{E}\!\left[\|\widehat{\theta}_S - \theta_0^\star\|_2^2\right]
\;\le\; \frac{d\tau^2}{N_S} + \max_{k \in S} h_k^2
\;\asymp\; \fR(\bh, \bn,\tau).
\]

\noindent\textit{Bounding (II).}
For any fixed $J$, the same oracle bound yields
\[
\mathbb{E}\!\left[\|\widehat{\theta}_J - \theta_0^\star\|_2^2\right]
\;\lesssim\; \frac{d\tau^2}{N_J} + \max_{k \in J} h_k^2
\;\lesssim\; 1,
\]
since $\max_{k \in [m]} h_k^2 \le 1$. Combining this with Step 1: $\text{(II)} \;\lesssim\; \frac{m}{n_0^{\alpha d}}$.

\noindent\textit{Conclusion.}
Combining the bounds on (I) and (II),
\[
\mathbb{E}\!\left[\|\widehat{\theta}_{\widehat{S}} - \theta_0^\star\|_2^2\right]
\;\lesssim\; \fR(\bh, \bn,\tau) + \frac{m}{n_0^{\alpha d}}.
\]
The second term is dominated by the first whenever
$m/n_0^{\alpha d} \lesssim d\tau^2/N_{\mathrm{total}}$,
which holds provided $\alpha \;\ge\; \frac{\log\!\left(m N_{\mathrm{total}}/(d\tau^2)\right)}{d\log(n_0)}$.
Together with the requirement $\alpha \ge 1$, this gives the stated condition on $\alpha$, and therefore
\[
\mathbb{E}\!\left[\|\widehat{\theta}_{\widehat{S}} - \theta_0^\star\|_2^2\right]
\;\lesssim\; \fR(\bh, \bn,\tau),
\]
uniformly over all $S$, $\mathbf{h} \in \mathcal{H}(S)$, and $P \in \mathcal{P}'(\mathbf{h},\mathbf{n},\tau,\kappa,S)$.
This completes the proof.
\end{proof}

\begin{lemma}\label{lemma:tech-1-t2c}
    Let $X \in \mathbb{R}^d$ satisfy the following for some $\beta > 0$: $\max_{\|a\|_2 = 1} P\!\left(a^\top X > t\right) \le \exp\!\left(-\frac{t^2}{\beta^2}\right)$ for all $t > 0$.
     Then, for $c = 2\sqrt{2\log 5}$ and all $t > c\beta\sqrt{d}$,
    \[
   \textstyle  P\!\left(\|X\|_2 > t\right) \le \exp\!\left(-\frac{t^2}{2\beta^2}\right).
    \]
\end{lemma}

\begin{proof}
    Let $\mathcal{N}$ be a $\tfrac{1}{2}$-net of $\mathcal{S}^{d-1}$ with $|\mathcal{N}| \le 5^d$.
    For any $a \in \mathcal{S}^{d-1}$, pick $a_0 \in \mathcal{N}$ with $\|a - a_0\|_2 \le \tfrac{1}{2}$; then
    \[
    a^\top X = a_0^\top X + (a-a_0)^\top X \le a_0^\top X + \tfrac{1}{2}\|X\|_2.
    \]
    Taking the maximum over $a \in \mathcal{S}^{d-1}$ and using $\|X\|_2 = \max_{a \in \mathcal{S}^{d-1}} a^\top X$ gives
    $\tfrac{1}{2}\|X\|_2 \le \max_{a_0 \in \mathcal{N}}\, a_0^\top X$.
    Hence, by a union bound and the assumed tail condition,
    \[
  \textstyle   P\!\left(\|X\|_2 > t\right)
    \le P\!\left(\max_{a_0 \in \mathcal{N}} a_0^\top X > \tfrac{t}{2}\right)
    \le 5^d \exp\!\left(-\frac{t^2}{4\beta^2}\right)
    = \exp\!\left(d\log 5 - \frac{t^2}{4\beta^2}\right).
    \]
    This is at most $\exp(-t^2/(2\beta^2))$ whenever $t > 2\beta\sqrt{2d\log 5} = c\beta\sqrt{d}$.
\end{proof}

\section{Supplementary material for Section \ref{sec:simuations}}

\subsection{A modification on the clustering based estimator}

\begin{algorithm}[!ht]
  \caption{Clustering based estimator}
  \label{alg:clustering-estimator}
  \begin{algorithmic}[1]
    \Require Local estimators $\{\widetilde \theta_k : k = 0, \dots, m\}$, number of clusters $K \ge 1$, and $C > 0$

\medskip 
    
    \State \textcolor{green!40!black}{//Low-dimensional projection}
    \State Let $a = d \wedge \{(K-1) \vee 1\}$ and define 
    \[
   \bar \theta:= \frac{\sum_{k = 1}^m n_k \widetilde \theta_k  }{\sum_{k = 1}^m n_k}, \qquad  \hat \Sigma := \frac{\sum_{k = 1}^m n_k (\widetilde \theta_k - \bar \theta) (\widetilde \theta_k - \bar \theta)^\top }{\sum_{k = 1}^m n_k}
    \]
    \State For $a = d \wedge \{(K-1) \vee 1\}$ let $A \in \reals^{d \times a}$ be the matrix whose column vectors are the $a$ highest eigen-vectors of $\hat \Sigma$. 

\medskip 
    
\State \textcolor{green!40!black}{//Clustering}
\State Let $\{S_1, \dots, S_K\}$ be partition of $\{1, \dots, m\}$ which are cluster assignments obtained using K-means clustering on $\left\{A^\top \widetilde \theta_k: k = 1, \dots, m\right\}$
\For{$j = 1,\dots, K $}
\State Define $N_j := \sum_{k \in S_j} n_j$, $\mu_j := \frac{1}{N_j}\sum_{k \in S_j} n_k \widetilde \theta_k$ and $\Delta_j = \left \| \mu_j - \widetilde \theta_0 \right\|_2$.
\EndFor

\medskip 

\State \textcolor{green!40!black}{//Integrable source domains}

\State Define $j_0 := \arg\min_j \Delta_j$ and $I = \{0\} \cup S_j$. 
\For{$j = 1, \dots, K$ and $j \neq j_0$}
\State Define the test statistic a $T_j = \Delta_j^2 - \Delta_{j_0}^2$  and $\lambda_j = C \max \left\{ \frac{d}{N_j \wedge N_{j_0}}, \frac{1}{n_0} \right\}$.
\If{$T_j \le \lambda_j$}
\State $I \gets  I \cup S_j$ \qquad // expand the index set of integrable sources
\EndIf
\EndFor

\medskip 

\Return  the final estimator 
\[
\widehat \theta_0 := \frac{\sum_{k \in I} n_k \widetilde \theta_k}{\sum_{k \in I} n_k}\,.
\]
 
  \end{algorithmic}
\end{algorithm}

\vskip 0.2in
\bibliography{sample,seamus,references-mendeley}

\end{document}